\documentclass[pdflatex,sn-mathphys-num]{sn-jnl}% Math and Physical Sciences Numbered Reference Style
\usepackage[numbers,sort&compress]{natbib}

\usepackage{graphicx}%
\usepackage{multirow}%
\usepackage{amsmath,amssymb,amsfonts}%
\usepackage{amsthm}%
\usepackage{mathrsfs}%
\usepackage[title]{appendix}%
\usepackage{xcolor}%
\usepackage{float}
\usepackage{textcomp}%
\usepackage{manyfoot}%
\usepackage{booktabs}%
\usepackage{algorithm}%
\usepackage{algorithmicx}%
\usepackage{algpseudocode}%
\usepackage{listings}%
\usepackage{lineno}
\usepackage{placeins}
\usepackage{moresize}
\usepackage[bb=boondox]{mathalfa}

\usepackage[symbol]{footmisc}

\usepackage[thinlines]{easytable}
\usepackage{needspace}

\newtheorem{theor}{Theorem}
\newtheorem{lemma}{Lemma}
\newtheorem{prop}{Proposition}
\theoremstyle{definition}
\newtheorem{defin}{Definition}
\newtheorem{rem}{Remark}

\newcommand{\rev}[1]{{\color{black} #1}}

\begin{document}

\title[High-order nonstandard multistep multistage methods]{Some properties of high-order nonstandard multistep multistage methods}


\author[1,2,*]{\fnm{Balint M.} \sur{Takacs}\footnote[0]{*Corresponding author. E-mail: takacs.balint.mate@ttk.bme.hu}}%\email{takacs.balint.mate@ttk.bme.hu}

% \author{\fnm{ } \sur{ }}


\affil[1]{\orgdiv{Department of Analysis and Operations Research, Institute of Mathematics}, \orgname{Budapest University of Technology and Economics}, \orgaddress{\street{Muegyetem rakpart 3.}, \city{Budapest}, \postcode{H-1111},  \country{Hungary}}}

\affil[2]{\orgname{HUN-REN-ELTE Numerical Analysis and Large Networks Research Group}, \orgaddress{\street{Pázmány Péter sétány 1/C}, \city{Budapest}, \postcode{H-1117}, \country{Budapest}}}

% \affil[3]{\orgdiv{Department}, \orgname{Organization}, \orgaddress{\street{Street}, \city{City}, \postcode{610101}, \state{State}, \country{Country}}}


\abstract{In this paper, we introduce nonstandard versions of multistep multistage methods. While proving the convergence of these schemes, we also define nonstandard general linear methods. We show that the nonstandard methods can attain the same order as their standard counterparts while preserving certain qualitative properties (e.g., boundedness) for all positive step sizes. These results are also demonstrated by some numerical experiments.}




\maketitle

%\linenumbers

\section{Introduction}\label{sec:intro}
Apart from the aim to design numerical methods that approximate a given differential equation in a time-efficient manner (i.e., the order of a given method is high), one of the key goals of numerical modeling is the construction of numerical schemes that behave in a qualitatively reasonable way, meaning that they preserve some properties of the original, continuous model. Although the application of adaptive step sizes is a leading area in this field \cite{hadjimichael,nusslein,arevalo,charous}, another approach is the use of nonstandard finite differences. In the last year alone, the latter methods were used in epidemiology \cite{morani, zinihi, marime, tasse}, cybersecurity \cite{wacker, hoang26}, ecology \cite{faheem, eskandari, ozdogan}, among others \cite{mohye, fazayel, rehman}. The main advantage of these methods, originally developed by Mickens \cite{mickens}, is that they typically preserve certain properties of the original continuous model for all positive step sizes.

One disadvantage of nonstandard methods is that in most cases, they only attain a convergence of order one. However, in recent years, there has been some work in the construction of higher-order methods. One early example was the investigation of higher-order nonstandard Runge-Kutta methods \cite{dimitrov, dimitrov2}, while more recent approaches include:
\begin{itemize}
    \item the employment of Theta methods \cite{ gupta2, kojouharov, dimitrov3} with several applications \cite{anguelov03, anguelov2, lubuma},
    \item the use of Richardson extrapolation \cite{wacker2,bassenne,gonzalez},
    \item the change of the denominator function into a more general form \cite{alal, alal23, gupta, hoang24, hoang24new, conte25},
    \item or a clever split of the right-hand side of the equation \cite{anguelov22, hoang22},
\end{itemize}
among others \cite{chen, kojo04, martin, ejere}. One usual idea of these methods is to restrict ourselves to a set of equations in some special form, and therefore construct a method for that given type of equations.

Another possibility is to start from well-known 'standard' methods, and by changing the timestep $\Delta t$ to $\varphi(\Delta t)$, we can construct nonstandard forms of these schemes. This transformation has been applied to Runge-Kutta \cite{dang, mosleh26} and linear multistep methods \cite{anguelov, takacs26}. The core idea of this approach is that by rewriting the 'standard' schemes into a strong-stability preserving (SSP) form (see \cite{sspbook}), we can guarantee the preservation of some linear properties. Note that the ideas of SSP methods have also been applied in the nonstandard framework before, but in a different context \cite{anguelov3, khal17}.

The main aim of the present paper is to generalize nonstandard Runge-Kutta and nonstandard linear multistep methods by introducing nonstandard versions of multistep multistage methods. In the process, we also define nonstandard versions of general linear methods. By changing $\Delta t$ to $\varphi(\Delta t)$, these new, nonstandard schemes can preserve some properties (e.g., boundedness) for every possible $\Delta t >0$, while also attaining the same orders as the 'standard' methods. It is worth noting that the work relies heavily on the paper of Constantinescu and Sandu \cite{const}, who constructed the standard forms of these multistep multistage methods.

The paper is structured as follows. In Section \ref{sec:prelim}, we detail the form of the numerical method, along with some core ideas. The main results about the convergence, the order, and the preservation of qualitative properties are stated in Section \ref{sec:main}, with their proofs presented in \ref{sec:proofs}. The performance of the methods is demonstrated in Section \ref{sec:num} via several numerical experiments. Some closing remarks and possible extensions are highlighted in Section \ref{sec:conc}.

\section{Preliminaries}\label{sec:prelim}
In this paper, we consider an initial problem of an autonomous ordinary differential equation given in the form
\begin{equation}\label{eq:ODE}
    \begin{cases}
        x'(t) & = f(x(t)), \\
        x(t_0) & = \tilde{x},
    \end{cases}  
\end{equation}
where $x: [t_0,\;T] \rightarrow \mathbb{R}^N$ ($[t_0,\;T] \subset \mathbb{R}$, $N \in \mathbb{Z}^+$), $x(t) = \left( x_1(t),\; x_2(t), \; \dots, \; x_N(t) \right)^T$, $f: \mathbb{R}^N \rightarrow \mathbb{R}^N$ and $\tilde{x} \in \mathbb{R}^N$. If $f$ has some nice properties, then the equation has a unique solution (for a comprehensive list of sufficient conditions, see \cite{agarwal}). 

Later in Section \ref{sec:main}, we also require some additional properties to hold for equation \eqref{eq:ODE}. Namely,
\begin{itemize}
    \item[(P1)] we assume that there exists a closed set $S \subset \mathbb{R}^N$ such that if $\tilde{x} \in S$, then $x(t) \in S$ for every $t \in [t_0, T]$.
    \item[(P2)] Assume that for every $z \in S$, $\Vert z \Vert \leq M$ and $\Vert f (z) \Vert \leq L M$ hold for some $M, L \in \mathbb{R}^+$.
    \item[(P3)] Moreover, for every $z_1, z_2 \in S$, $\Vert f(z_1) - f(z_2)\Vert \leq L \Vert z_1 - z_2 \Vert$ holds for $L \in \mathbb{R}^+$.
\end{itemize}

In order to acquire an approximate solution of \eqref{eq:ODE}, we split our time-interval $[t_0, \;T]$ with equidistant points \linebreak $\{ t_n := t_0 + n \Delta t \}_{n = 0,\; 1, \; \dots, \; \mathcal{N}}$, where $t_{\mathcal{N}}=T$ and $\Delta t > 0$ is the step-size. For the approximations of $x(t)$ at point $t=t_n$, we use the notation $y_n \approx x(t_n)$. 

Our main objective of this work is to extend explicit multistep multistage methods given in the Shu-Osher form (\cite{const, sspbook})
\begin{equation}\label{eq:SMMM}
\begin{aligned}
    y_{[n-1]}^{(1)} = & \; y_{n-1},\\%[10pt]
   y_{[n-1]}^{(i)} = & \; \sum_{\ell=2}^k \left(\sum_{j=1}^{s} \alpha_{[n-\ell]}^{i,j} y_{[n-\ell]}^{(j)} + \beta_{[n-\ell]}^{i,j} \; \Delta t \; f\left(y_{[n-\ell]}^{(j)}\right) \right) + \\%[10pt]
  & + \sum_{j=1}^{i-1} \alpha_{[n-1]}^{i,j} y_{[n-1]}^{(j)} + \beta_{[n-1]}^{i,j} \; \Delta t \; f\left(y_{[n-1]}^{(j)}\right) \quad (i=2,\; 3, \; \dots,\; s+1),\\%[10pt]
  y_n \quad \;\; = & \; y_{[n-1]}^{(s+1)},
  \end{aligned}
\end{equation}
where $\alpha_{[n]}^{i,j}, \beta_{[n]}^{i,j} \geq 0$ $(i=2,\; 3, \; \dots,\; s+1, \; n=0, \; 1, \; \dots, \mathcal{N}-1)$. Later, the values of $\alpha_{[n-\ell]}^{i,j}$ and $\beta_{[n-\ell]}^{i,j}$ not present in scheme \eqref{eq:SMMM} (e.g., $\alpha_{[n-1]}^{1,1}$) will be chosen to be zeros. Also, it is assumed that $\mathcal{N}>k$ (so the method can take more than one step).

From now on, we refer to values $y_{[n]}^{(i)}$ as inner stage values, where it is assumed that $y_{[n]}^{(i)} \approx x\left(t_n + c_i \Delta t\right)$ for some $c_i \in [0,\; 1]$. The coefficients of the multistep multistage methods observed in this paper (calculated in \cite{const}) can be found in Appendix A. It is easy to see that Runge-Kutta methods and multistep methods are both special cases of \eqref{eq:SMMM}: namely, the former corresponds to the case $k=1$, while the latter is involved in the case $s=1$.

In accordance with the nonstandard modeling rules of Mickens \cite{mickens}, by replacing the term $\Delta t$ in \eqref{eq:SMMM} with $\varphi(\Delta t)$, we get the nonstandard forms of these methods. Namely,
\begin{equation}\label{eq:NSMMM}
\begin{aligned}
    y_{[n-1]}^{(1)} = & \; y_{n-1},\\%[10pt]
   y_{[n-1]}^{(i)} = & \; \sum_{\ell=2}^k \left(\sum_{j=1}^{s} \alpha_{[n-\ell]}^{i,j} y_{[n-\ell]}^{(j)} + \beta_{[n-\ell]}^{i,j} \; \varphi(\Delta t) \; f\left(y_{[n-\ell]}^{(j)}\right) \right) + \\%[10pt]
  & + \sum_{j=1}^{i-1} \alpha_{[n-1]}^{i,j} y_{[n-1]}^{(j)} + \beta_{[n-1]}^{i,j} \; \varphi(\Delta t) \; f\left(y_{[n-1]}^{(j)}\right) \quad (i=2,\; 3, \; \dots,\; s+1),\\%[10pt]
  y_n \quad \;\; = & \; y_{[n-1]}^{(s+1)}.
  \end{aligned}
\end{equation}
Here $\varphi: \mathbb{R} \rightarrow \mathbb{R}$ is assumed to be at least continuously differentiable, and 
\begin{equation}\label{eq:phicond}
    \varphi(h) = h + O(h^2) \qquad \text{for sufficiently small values of } h.
\end{equation}
As we will see later in Section \ref{sec:order}, to attain the high orders of the standard multistep multistage methods \eqref{eq:SMMM}, some further conditions should also be fulfilled by $\varphi$. Moreover, to preserve the qualitative properties of the continuous model, the boundedness of such functions is also needed: see Section \ref{sec:preserv} for details.

The main reason for the replacement of $\Delta t$ by $\varphi(\Delta t)$ is the idea that by doing so, we can construct a method that preserves the qualitative properties of the original model for any choice of the stepsize $\Delta t$. The exact condition is given in Section \ref{sec:preserv}.

\begin{rem}\label{rem:initial}
    It is clear that to be able to start methods \eqref{eq:SMMM} and \eqref{eq:NSMMM}, the values of $y_{[n]}^{(i)}$, $n=0, \dots, k-1$, $i=1, \dots, s+1$ are needed. Later, we will use the notation $\psi(\Delta t) : (0, \infty) \rightarrow \mathbb{R}^{Nk(s+1)}$ for this method (sometimes referred to as 'starting procedure') that calculates these values for a given choice of $\Delta t$. The exact construction of such a starting method is mentioned in Remark \ref{rem:initialRK}.
\end{rem}

In the following section, we state the main properties of these methods: the convergence and the order are analyzed, and it is also shown that for a sufficiently chosen function $\varphi$, these new nonstandard methods preserve some of the properties of the continuous solutions for every choice of $\Delta t>0$.

\section{Main results}\label{sec:main}
In this section, the main results of the paper are stated, while the proofs not presented here can be found in Section \ref{sec:proofs}. First, we show that the convergence of the standard scheme is equivalent to the convergence of the corresponding nonstandard one. Later, the orders of the methods are analyzed. Finally, the preservation of qualitative properties also holds under some conditions on the function $\varphi$.

\subsection{Convergence of the method}\label{sec:conv}
The usual way the convergence of a standard multistep multistage method \eqref{eq:SMMM} is proven is to rewrite this method as a general linear method of the form
\begin{equation}\label{eq:GLM}
    \begin{aligned}
        G & = \Delta t \; A f(G) + U g_{[n-1]},\\
    g_{[n]} & = \Delta t \; B f(G) + V g_{[n-1]},
    \end{aligned}
\end{equation}
where $g_{[n]} = \left( (g_{[n]}^{1})^T, \; (g_{[n]}^{2})^T, \dots , \; (g_{[n]}^{r})^T   \right)^T$, $G = \left( (G^1)^T, \; (G^2)^T, \dots , \; (G^{s+1})^T \right)^T$, (here $g_{[n]}^{i}, G^{j} \in \mathbb{R}^N$ for every $i=1, 2, \dots r, \; j=1, \dots s+1$), $f(G)$ is the vector collecting all the different $f(G^i)$ values, $A \in \mathbb{R}^{N(s+1) \times N(s+1)}$, $U \in \mathbb{R}^{N(s+1) \times N r}$, $B \in \mathbb{R}^{Nr \times N(s+1)}$, $V \in \mathbb{R}^{Nr \times Nr}$. 

Note that matrices $A$, $U$, $B$ and $V$ are all results of a Kronecker product of some matrix with the $N \times N$ identity matrix $I_N$, namely, $A= \widehat{A} \otimes I_N$, $U= \widehat{U} \otimes I_N$, $B= \widehat{B} \otimes I_N$ and $V= \widehat{V} \otimes I_N$ for some matrices $\widehat{A} \in \mathbb{R}^{(s+1)\times (s+1)}$, $\widehat{U} \in \mathbb{R}^{(s+1)\times r}$, $\widehat{B} \in \mathbb{R}^{r\times (s+1)}$, $\widehat{V} \in \mathbb{R}^{r\times r}$. Since the matrices in \eqref{eq:GLM} usually have a lot of repeating elements because of the structure just mentioned, later we will use the notation $a_{ij}$, $U_{ij}$, $b_{ij}$ and $V_{ij}$ for the elements of matrices $\widehat{A}$, $\widehat{U}$, $\widehat{B}$ and $\widehat{V}$, respectively, instead of the elements of the matrices in \eqref{eq:GLM}.

Let us introduce the notations $Y_{[n]} = \left( (y_{[n]}^{(1)})^T, \dots (y_{[n]}^{(s+1)})^T \right)^T$, $\Lambda_{[n]} = \left[ \alpha_{[n]}^{i,j} \right] \in \mathbb{R}^{(s+1)\times (s+1)}$, $\Gamma_{[n]} = \left[ \beta_{[n]}^{i,j} \right] \in \mathbb{R}^{(s+1)\times (s+1)}$, $\widetilde{\Lambda}_{[n]} = \Lambda_{[n]} \otimes I_N$, $\widetilde{\Gamma}_{[n]} = \Gamma_{[n]} \otimes I_N$, and let $e_1$ be a vector of length $N(s+1)$ with ones in the first $N$-many places and zeros for all the others. Then, method \eqref{eq:SMMM} can be expressed as
\begin{equation}\label{eq:Yequation1}
\begin{aligned}
    Y_{[n-1]} = \sum_{\ell=2}^k & \left( \widetilde{\Lambda}_{[n-\ell]} Y_{[n-\ell]} + \Delta t \; \widetilde{\Gamma}_{[n-\ell]} f(Y_{[n-\ell]}) \right) + \\
    & \hspace{20pt} + e_1 y_{n-1} + \widetilde{\Lambda}_{[n-1]} Y_{[n-1]} + \Delta t \; \widetilde{\Gamma}_{[n-1]} f(Y_{[n-1]}).
\end{aligned}
\end{equation}
Since those elements of $\widetilde{\Lambda}$ and $\widetilde{\Gamma}$ not present in scheme \eqref{eq:SMMM} can be thought of as zeros, the matrix $(I-\widetilde{\Lambda}_{[n-1]})$ is a lower-triangular matrix with ones in its main diagonal. Therefore, it is an invertible matrix, and equation \eqref{eq:Yequation1} can be rewritten as
\begin{equation}\label{eq:Yequation2}
    Y_{[n-1]} = \sum_{\ell=2}^k \left( \overline{\Lambda}_{[n-\ell]} Y_{[n-\ell]} + \Delta t \; \overline{\Gamma}_{[n-\ell]} f(Y_{[n-\ell]}) \right) + \overline{e_1} y_{n-1} + \Delta t \; \overline{\Gamma}_{[n-1]} f(Y_{[n-1]}),
\end{equation}
where
\begin{align*}
    \overline{\Lambda}_{[n-\ell]} & = \left( I - \widetilde{\Lambda}_{[n-1]}\right)^{-1} \widetilde{\Lambda}_{[n-\ell]},  \qquad  (2 \leq \ell \leq k), \\
    \overline{\Gamma}_{[n-\ell]} & = \left( I - \widetilde{\Lambda}_{[n-1]}\right)^{-1} \widetilde{\Gamma}_{[n-\ell]}, \qquad (1 \leq \ell \leq k), \\
    \overline{e_1} & = \left( I - \widetilde{\Lambda}_{[n-1]}\right)^{-1} e_1. 
\end{align*}
Consequently, method \eqref{eq:SMMM} can be rewritten as a general linear method of form \eqref{eq:GLM} with the following choices:
\begin{align}\label{eq:GLM_forms}
\scriptsize
     \notag & g_{[n-1]} = \\
     & = \left( Y_{[n-k+1]}^T, \; Y_{[n-k+2]}^T, \; \dots, \; Y_{[n-2]}^T, \; Y_{[n-1]}^T, \right. \\ 
     \notag & \hspace{50pt} \left. \Delta t \left(f(Y_{[n-k+1]})\right)^T, \; \dots, \; \Delta t \left(f(Y_{[n-1]})\right)^T \right)^T, \\
     \normalsize
     A & =\overline{\Gamma}_{[n-1]},\\
     B & = (0,\;  \dots, \; 0, A^T, 0,\;  \dots, \; 0, \; \overline{I})^T,\\
     U & = (\overline{\Lambda}_{[n-k]},\; \overline{\Lambda}_{[n-k+1]},  \dots, \; \overline{\Lambda}_{[n-2]}, \overline{\Gamma}_{[n-k]}, \dots, \overline{\Gamma}_{[n-2]}),\\
    V & =\begin{pmatrix}
    0 & \overline{I} & 0 & \dots & 0 & 0 & \dots & 0 \\
    0 & 0 & \overline{I} & \dots & 0 & 0 & \dots & 0 \\
    \vdots & \vdots & \vdots & \ddots & \vdots & \vdots & \ddots & \vdots \\
    0 & 0 & 0 & \dots & \overline{I} & 0 & \dots & 0 \\
    \overline{\Lambda}_{[n-k]} & \overline{\Lambda}_{[n-k+1]} & \overline{\Lambda}_{[n-k+2]} & \dots & \overline{\Lambda}_{[n-2]} & \overline{\Gamma}_{[n-k]} & \dots & \overline{\Gamma}_{[n-2]} \\
    0 & 0 & 0 & \dots & 0 & \overline{I} & \dots & 0 \\
    \vdots & \vdots & \vdots & \ddots & \vdots & \vdots & \ddots & \vdots \\
    0 & 0 & 0 & \dots & 0 & 0 & \dots & \overline{I} \\
    0 & 0 & 0 & \dots & 0 & 0 & \dots & 0 \\ \label{eq:GLM_forms2}
\end{pmatrix},
\end{align}
where $\overline{I}$ is an identity matrix of size $N(s+1) \times N(s+1)$. Moreover, $G = Y_{[n-1]}$.

To prove the convergence of nonstandard method \eqref{eq:NSMMM}, we apply a similar line of thought as in the case of the standard scheme. Namely, let us introduce the nonstandard version of a general linear method, defined as
\begin{equation}\label{eq:NSGLM}
    \begin{aligned}
        G & = \varphi(\Delta t) \; A f(G) + U g_{[n-1]},\\
    g_{[n]} & = \varphi(\Delta t) \; B f(G) + V g_{[n-1]},
    \end{aligned}
\end{equation}
where $\varphi$ has the same properties as discussed in Section \ref{sec:prelim}. Since the only difference between the standard and nonstandard methods is the form of the timestep, the transformation discussed above can also be applied to nonstandard multistep multistage methods. Particularly, by the choices \eqref{eq:GLM_forms}--\eqref{eq:GLM_forms2}, the nonstandard multistep multistage method \eqref{eq:NSMMM} can be reformulated as a nonstandard general linear method \eqref{eq:NSGLM}.

The main question now is the convergence of method \eqref{eq:NSGLM}. For this, we define the consistency and stability of these methods.
\begin{defin}[See e.g. \cite{butcherbook}]
    The 'standard' general linear method \eqref{eq:GLM} is said to be consistent, if there exists vectors $u$ and $v$ such that
    \begin{align*}
        Vu &= u, \qquad Uu=\mathbb{1},\qquad   B \mathbb{1} + Vv = u + v,
    \end{align*}
    where $\mathbb{1}$ is the all-one vector.
\end{defin}
In the standard case, we assume that the elements of vector $g$ approximate a linear combination of the values of the solution and its derivative, and vectors $u$ and $v$ collect the corresponding coefficients. More precisely,
\begin{equation}\label{eq:constmeaning}
    g_{[n]}^{(i)} = u_i x(t_n) + v_i \Delta t f(x(t_n)) + O((\Delta t)^2). 
\end{equation}
It can also be shown that consistent general linear methods solve the equation $x'(t)=1$, $x(0)=a$ exactly, meaning that $g_{[n]}^{(i)} = u_i (nt + a) + v_i \Delta t$.

Now, let us define a similar concept for the nonstandard methods.
\begin{defin}
    The nonstandard general linear method \eqref{eq:NSGLM} is said to be consistent, if there exists vectors $u$ and $v$ such that
    \begin{align*}
        Vu &= u, \qquad Uu=\mathbb{1},\qquad   B \mathbb{1} + Vv = u + v,
    \end{align*}
    where $\mathbb{1}$ is the all-one vector.
\end{defin}
Since it was assumed that $\varphi (h) = h + O(h^2)$ for small values of $h$, for consistent nonstandard general linear methods \eqref{eq:constmeaning} also holds. Moreover, for a given $\Delta t$, these methods solve the equation $x'(t)=\kappa$, $x(0)=a$ exactly, where $\kappa = \frac{\varphi(\Delta t)}{\Delta t}$, in the sense that $g_{[n]}^{(i)} = u_i \left(n\varphi(\Delta t) + a\right) + v_i \varphi(\Delta t)$. 

As a result, we can state the connection between the two notions of consistency, which is a trivial consequence of the previous two definitions.
\begin{prop}\label{th:cons}
    A nonstandard general linear method is consistent if and only if its standard counterpart is also consistent.
\end{prop}

Next, let us consider stability, which is similar to the previous case.
\begin{defin}[See e.g. \cite{butcherbook}]
    The 'standard' general linear method \eqref{eq:GLM} is stable if there exists a positive constant $C$ such that $\Vert V^n \Vert \leq C$ holds for every $n \in \mathbb{Z}^+$.
\end{defin}
\begin{defin}
    The nonstandard general linear method \eqref{eq:NSGLM} is stable if there exists a positive constant $C$ such that $\Vert V^n \Vert \leq C$ holds for every $n \in \mathbb{Z}^+$.
\end{defin}
The following statement is a trivial consequence of these previous definitions.
\begin{prop}\label{th:stable}
    A nonstandard general linear method is stable if and only if its standard counterpart is also stable.
\end{prop}

Finally, the convergence of the method can be established by these two previous notions.
\begin{defin}[See e.g. \cite{butcherbook}]
    A 'standard' general linear method \eqref{eq:GLM} is convergent if for any initial value problem \eqref{eq:ODE} (subject to some condition ensuring the existence of a unique solution, see \cite{agarwal}), there exists a non-zero vector $u$ and a starting procedure $\psi: (0,\infty)\rightarrow\mathbb{R}^r$ such that $\displaystyle\lim_{\Delta t \rightarrow 0} \psi_i(\Delta t) = u_i x(t_0)$ and such that for any $\hat{t}>t_0$ ($\hat{t}\leq T$), the sequence of vectors $g_{[n]}$ (computed using $n$ steps with stepsize $\Delta t = (\hat{t}-t_0)/n$ and $g_{[0]} = \psi(\Delta t)$) converges to $u x(\hat{t})$.
\end{defin}
The convergence of the nonstandard method \eqref{eq:NSGLM} can be defined the same way.

From now on, we only consider covariant methods. Covariance means that when a general linear method is applied to solve the following two ordinary differential equations:
\begin{align*}
    x_1'(t) & = f(x_1(t)),\qquad \quad \;\; x_1(t_0) = x_0,\\
    x_2'(t) & = f(x_2(t) - \eta),\qquad x_2(t_0) = x_0 + \eta,
\end{align*}
then the two numerical schemes should produce the same numerical solution, only shifted by $\eta$. In other words, if we shift arbitrary component $g_{[0]}^i$ by $u_i \eta$, then $g_{[1]}^i$ should also be shifted by the same amount. This should also hold for internal stages too.

For the next theorems, we also assume that properties (P1)--(P3) (defined in Section \ref{sec:prelim}) hold.

\begin{theor}[\cite{butcherbook} Th. 513A, 514A and 515D]\label{th:Lax_S}
    A 'standard' covariant general linear method \eqref{eq:GLM} is convergent if and only if it is consistent and stable.
\end{theor}
A similar connection can also be stated in the nonstandard case.
\begin{theor}\label{th:Lax_NS}
    A nonstandard covariant general linear method \eqref{eq:NSGLM} is convergent if and only if it is consistent and stable.
\end{theor}
%The proof of this statement can be found in Section \ref{sec:proofs}.

Now, we can state the main theorems of this section.
\begin{theor}
    A nonstandard general linear method \eqref{eq:NSGLM} is convergent if and only if the standard counterpart \eqref{eq:GLM} is also convergent.
\end{theor}
Because of the simplicity of the proof, we present it here.
\begin{proof}
    If a nonstandard general linear method is convergent, then by Proposition \ref{th:Lax_NS}, it is consistent and stable. By Propositions \ref{th:cons} and \ref{th:stable}, the standard counterpart is consistent and stable too. Therefore, by Theorem \ref{th:Lax_S}, the standard method is convergent. The same argument also holds in the other direction.
\end{proof}
An easy consequence of the previous theorem is the following, which establishes the convergence of nonstandard multistep multistage methods.
\begin{theor}\label{th:MMMconv}
    A nonstandard general multistep multistage method \eqref{eq:NSMMM} is convergent if and only if the standard counterpart \eqref{eq:SMMM} is also convergent.
\end{theor}

\begin{rem}
    It is worth noting that Theorem \ref{th:MMMconv} can also be proven without the introduction of consistency and stability. However, because of the pivotal role of these notions in numerical analysis, their analysis in the nonstandard context seemed to be natural.
\end{rem}

\subsection{Order of the method}\label{sec:order}
Before we analyze the order of method \eqref{eq:NSMMM}, we state a statement about the nonstandard form of the Taylor series. Note that a similar result was already discussed in \cite{takacs26}, but there only the one-variable case was discussed: the present result deals with the general, multivariable case.

\begin{lemma}[Nonstandard form of Taylor's theorem]\label{lem:taylor}
Let $f: \mathbb{R}^{N_1} \rightarrow \mathbb{R}^{N_2}$ ($N_1, N_2 \in \mathbb{Z}^+$) be a function which is continuously differentiable at least $p+1$-many times in a closed neighborhood of $a \in \mathbb{R}^{N_1}$ denoted by $U_a$, $\delta \in \mathbb{R}^{N_1}$, for every $\xi \in [0,1]$ $a+\xi \delta \in U_a$ holds, and $\Vert f^{(n+1)} \Vert$ is bounded inside $U_a$. Moreover, let us assume that $\varphi: \mathbb{R} \rightarrow \mathbb{R}$ is $p+1$-many times differentiable for values sufficiently close to zero. Then, the following are equivalent:
\begin{itemize}
    \item[(C1)] $\varphi(h) = h + O(h^{p+1})$ for sufficiently small values of $h$,
    \item[(C2)] $f(a+\delta) = f(a) + f'(a)\varphi(\delta) + \dfrac{1}{2} f''(a) (\varphi(\delta),\varphi(\delta)) +\dots + \dfrac{1}{p!} f^{(p)}(a) (\varphi(\delta),\varphi(\delta),\dots, \varphi(\delta)) + R_{p+1},$
\end{itemize}
where
$$ R_{p+1} = \dfrac{1}{p!} \int_0^1 f^{(p+1)} (a + \xi \varphi(\delta)) (\varphi(\delta),\varphi(\delta),\dots, \varphi(\delta)) (1-\xi)^p d \xi. $$
\end{lemma}

%The proof can be found in Section \ref{sec:proofs}.
Note that in (C2), the application of $\varphi$ in $\varphi(\delta)$ is meant element-wise. Moreover, $R_{p+1} = O(\Vert \delta\Vert^{p+1})$ by the assumption \eqref{eq:phicond}.

Next, we state the main theorem about the order of nonstandard multistep multistage methods.
\begin{theor}\label{th:order}
    Let us assume that the standard multistep multistage method \eqref{eq:SMMM} has an order of $p$. Moreover, assume that the starting procedure $\psi$ also approximates the corresponding values with order $p$. Then, the following are equivalent for the corresponding nonstandard method \eqref{eq:NSMMM}:
\begin{itemize}
    \item[(C1)] $\varphi(h) = h + O(h^{p+1})$ for sufficiently small values of $h$,
    \item[(C3)] method \eqref{eq:NSMMM} attains an order of $p$.
\end{itemize}    
\end{theor}
%The proof can be found in Section \ref{sec:proofs}.

\begin{rem}
    Let us observe that Theorem \ref{th:order} can also be proved without the use of Lemma \ref{lem:taylor}. However, the use of the nonstandard form of Taylor's theorem highlights the similarities between the standard and nonstandard methods, i.e., their errors can be rewritten as the same Taylor series, the only difference being that in the nonstandard case, the terms include $\varphi(\Delta t)$ instead of $\Delta t$. Therefore, the theory of Butcher trees (see e.g. \cite[Sec. 30, 31 and 53]{butcherbook}) can also be applied in the nonstandard setting.
\end{rem}

\begin{rem}
    Note that since the only difference between the standard multistep multistage method \eqref{eq:SMMM} and the nonstandard counterpart \eqref{eq:NSMMM} is the change of $\Delta t$ to $\varphi(\Delta t)$, the stability functions and the region of absolute stability also coincide. The only difference is that, while in the standard case, the variable in the stability function is $z = \lambda \Delta t$ (if the considered test problem is $x'(t) = \lambda x(t)$), in the nonstandard case it is $z=\lambda \varphi(\Delta t)$. Even if $\varphi$ is bounded (see \ref{sec:preserv}), since we have no restrictions on $\lambda$, $z$ can take any value on the complex plane. Therefore, all the classic results about the connection between the orders of general linear methods and the absolute stability  (see, e.g., \cite[Sec. 52, 53]{butcherbook}) also hold in the nonstandard context. 
\end{rem}

\subsection{Preservation of qualitative properties}\label{sec:preserv}
Let $\mathcal{P}$ be an $N \times N$ matrix with nonnegative elements. In this section, we show that the methods (under some conditions) preserve either the value of $\mathcal{P}(y_n)$, or its boundedness property for every choice of the stepsize.

\begin{theor}\label{th:preserv}
    Let $\mathcal{R}$ be a given relation $\mathcal{R} \in  \{ $ '<', '$\leq$', '=', '$\geq$', '>' $\}$. Let us assume the following:
    \begin{itemize}
        \item There exists a set $\mathcal{H}\subset \mathbb{R}^N$ and a constant vector $\mu\in \mathbb{R}^N$ such that if $\tilde{x} \in \mathcal{H}$, then $\mathcal{P}x(t) \, \mathcal{R} \, \mu$ for every $t \in [t_0, T]$.
        \item For the starting values of the method, $\mathcal{P} y_{[n]}^{(i)}\,\mathcal{R}\, \mu$ holds if $y_{[n]}^{(i)} \in \mathcal{H}$ for every $n=0, 1, \dots, k-1$, $i=1, \dots, s+1$.
        \item There exists a constant $B_{FE} \in \mathbb{R}^+$ in a way that if $y_n \in \mathcal{H}$ and $\mathcal{P} y_n \,\mathcal{R}\,\mu$, then for every $0 \leq \Delta t \leq B_{FE}$, $\mathcal{P}\left(y_n + \Delta t f(y_n)\right) \,\mathcal{R} \, \mu$ also holds. 
        \item Function $\varphi$ is bounded in a way that $\varphi(x) \leq \mathcal{C} B_{FE}$ holds for every $x\geq 0$, where $\mathcal{C}$ is the SSP coefficient \cite{const, sspbook} of method \eqref{eq:NSMMM}, defined as
        $$ \mathcal{C} = \min \left\{ \alpha_{[n-\ell]}^{i,j} / \beta_{[n-\ell]}^{i,j}: 1 \leq i \leq s+1, 1 \leq j \leq i-1, 1 \leq \ell \leq k, \beta_{[n-\ell]}^{i,j} \neq 0\right\}. $$
    \end{itemize}
    Under these assumptions, method \eqref{eq:NSMMM} preserves the property of the continuous system, namely, $\mathcal{P} y_n \,\mathcal{R} \, \mu$ holds for $n=1, 2, \dots, \mathcal{N}$.
\end{theor}

%The theorem is proved in Section \ref{sec:proofs}.

Here, we highlight two special cases that are essential in applications:
\begin{itemize}
    \item If $\mathcal{P}$ is the identity matrix, $\mathcal{R}$ is '$\geq$' and $\mu$ is the all-zero vector, then the previous result is about the preservation of non-negativity for all of the components of $y_n$. (By changing $\mathcal{R}$ and $\mu$, this can also mean boundedness from below or from above by any other value.)
    \item If $\mathcal{P}$ is the all-one matrix, $\mathcal{R}$ is '$=$' and $\mu$ is a vector having the same value for all of its elements, then the theorem is about the preservation of the sum of all the components of $y_n$. For example, in chemical applications, this can mean mass preservation, or in epidemiological settings, the preservation of the total population (if there are no vital dynamics present).
\end{itemize}

\begin{rem}\label{rem:initialRK}
    A natural question is the way the initial values $y_{[n]}^{(i)}$ for every $n=0, 1, \dots, k-1$, $i=1, \dots, s+1$ can be calculated in a way that $\mathcal{P} y_{[n]}^{(i)}\,\mathcal{R}\, \mu$ also holds. One possible way is the application of nonstandard Runge-Kutta methods \cite{dang, mosleh26}. These can be thought of as special cases of multistep multistage methods \eqref{eq:NSMMM} with $k=1$. Therefore, if we consider a standard Runge-Kutta method, and then change the timestep $\Delta t$ to $\varphi(\Delta t)$ (where $\varphi$ satisfies the conditions of Theorems \ref{th:order} and \ref{th:preserv}), then this method attains the high order of the standard counterpart, while also preserving the desired properties of the continuous system \eqref{eq:ODE}. 
\end{rem}

\begin{rem}
   In the theory of nonstandard methods, one usual question is the elementary stability of methods, i.e., the behavior of the solution as $t \rightarrow \infty$ (in the continuous case) and $n \rightarrow \infty$ (in the discrete case). By a proof similar to the one of Theorem \ref{th:preserv}, one can show that the nonstandard methods (under conditions that can be similarly constructed to the ones in Theorem \ref{th:preserv}) preserve the equilibria of the continuous model for any choice of the stepsize. However, the determination of the stability of these equilibria might involve similar observations in the case of standard multistep multistage methods, which exceeds the scope of this paper, but might be an interesting topic of further research.
\end{rem}

\section{Proofs of the statements in Section \ref{sec:main}.}\label{sec:proofs}

In this section, we detail the proofs of the theorems of Section \ref{sec:main}.
\begin{proof}[Proof of Theorem \ref{th:Lax_NS}]
Similarly to \cite{butcherbook}, we prove the statement in two parts. First, we show that stability and consistency are necessary for convergence (Lemmas \ref{lemma1} and \ref{lemma2}), and then the sufficiency is proven (Lemmas \ref{lemma3}, \ref{lemma4} and \ref{lemma5}).

\begin{lemma}\label{lemma1}
    If a nonstandard general linear method is convergent, then it is stable.
\end{lemma}

\begin{proof}[Proof of Lemma \ref{lemma1}]
    The proof is the same as the proof of the standard case (see \cite[Th. 513A]{butcherbook}), since both of the methods produce the same numerical scheme for the differential equation
    $$ x'(t)=0, \qquad x(0) = 0.  $$
    Consequently, the same argument using proof by contradiction can be used.
\end{proof}

\begin{lemma}\label{lemma2}
    If a covariant nonstandard general linear method is convergent, then it is consistent.
\end{lemma}

\begin{proof}[Proof of Lemma \ref{lemma2}]
Covariance implies that $Vu=u$ and $Uu=\mathbb{1}$ - for the standard case, see \cite[Sec. 511]{butcherbook}, the nonstandard one is the same since \eqref{eq:phicond} holds. Therefore, the only thing that should be shown is that there exists a vector $v$ such that $B \mathbb{1} + Vv = u+v$.

Consider the differential equation
$$ x'(t)=1, \qquad x(0)=0, $$
with starting process $\psi(\Delta t)=0$ and $\hat{t}=1$. The approximation given by the nonstandard general linear method with timestep $\Delta t = \frac{1}{n}$ is
$$ g_{[k]} = \varphi\left( \dfrac{1}{n} \right) B \mathbb{1} + V g_{[k-1]}, \qquad k = 1, 2, \dots, n. $$
Therefore, the error vector, after $n$-many steps, has the form
$$ g_{[n]} - u = \varphi  \left( \dfrac{1}{n} \right) \left( I + V + V^2 + \dots + V^{n-1} \right) B \mathbb{1} - u. $$
%Let us use the fact that $\varphi \left( \frac{1}{n} \right) = \frac{1}{n} + O\left( \frac{1}{n^2} \right)$, meaning that
%$$ g_{[n]} - u = \dfrac{1}{n} \left( I + V + V^2 + \dots + V^{n-1} \right) B \mathbb{1} - u + O\left( \frac{1}{n^2} \right). $$
%Since, by covariance, $\dfrac{1}{n}\left( I + V + V^2 + \dots + V^{n-1} \right) u = \dfrac{1}{n} n u = u$ holds, then
%$$ g_{[n]} - u = \dfrac{1}{n} \left( I + V + V^2 + \dots + V^{n-1} \right) (B \mathbb{1} - u) + O\left( \frac{1}{n} \right) u. $$
Since, by covariance, $Vu=u$ and also $\varphi \left( \frac{1}{n} \right) = \frac{1}{n} + O\left( \frac{1}{n^2} \right)$,
$$ \varphi \left( \frac{1}{n} \right) \left( I + V + V^2 + \dots + V^{n-1} \right) u = \varphi \left( \frac{1}{n} \right) n u =  \left( \frac{1}{n} + O\left( \frac{1}{n^2} \right) \right) n u = u + O\left( \frac{1}{n} \right) u. $$
Consequently,
$$ g_{[n]} - u = \varphi\left(\dfrac{1}{n}\right) \left( I + V + V^2 + \dots + V^{n-1} \right) (B \mathbb{1} - u) + O\left( \frac{1}{n} \right) u. $$
Since $V$ has bounded powers, it can be formulated as
$$ V = \mathcal{T}^{-1} \begin{pmatrix}
    I & 0 \\ 0 & W
\end{pmatrix} \mathcal{T},$$
where $I$ is and $\tilde{r} \times \tilde{r}$ matrix ($\tilde{r}\leq r$) and $W$ has bounded powers, and $1 \notin \sigma(W)$. Therefore,
$$ g_{[n]} - u = \mathcal{T}^{-1} \begin{pmatrix}
    I & 0 \\ 0 & \varphi\left(\frac{1}{n}\right) (I-W)^{-1}(I-W^n)
\end{pmatrix} \mathcal{T} (B \mathbb{1} - u) + O\left( \frac{1}{n} \right) u. $$
If $n \rightarrow \infty$, then the right-hand side tends to
$$ \mathcal{T}^{-1} \begin{pmatrix}
    I & 0 \\ 0 & 0
\end{pmatrix} \mathcal{T} (B \mathbb{1} - u), $$
which is the same as the limit in the standard case (see \cite[Th. 514A]{butcherbook}), so the last steps of the proof are the same.
\end{proof}

To prove the sufficiency of consistency and stability, we follow the same line of thought as \cite[Th. 515D]{butcherbook}. Namely, we bound the global error of the method in three steps.

\begin{lemma}\label{lemma3}
    Assume that conditions (P1)--(P3) hold. Moreover, also assume the following:
    \begin{itemize}
        \item $\varphi(\Delta t) \leq \tau_0$, chosen in a way that $\tau_0 L \Vert A \Vert_{\infty} < 1$. For sufficiently small values of $\Delta t$, this is true. Also note that the boundedness of function $\varphi$ is also needed for the preservation of the qualitative properties of the continuous model, see Section \ref{sec:preserv}. 
        \item $\tilde{x} \in S$, where $S \subset \mathbb{R}^N$ is the same as defined in Section \ref{sec:prelim}.
    \end{itemize}
    Let us also introduce the following notations:
    \begin{itemize}
        \item Let $\varepsilon \in \mathbb{R}^{s+1}$ be a vector satisfying
        $$ \sum_{j=1}^{s+1} \left( \delta_{ij} - \tau_0 L |a_{ij}| \right) \varepsilon_j = \dfrac{1}{2} c_i^2 + \sum_{j=1}^{s+1} |a_{ij} c_j|, $$
        where $\delta_{ij}$ is the Kronecker-delta. 
        \item Let $\hat{g}_{[n-1]}^i = u_i x(t_{n-1}) + v_i \varphi(\Delta t) x'(t_{n-1})$ (the 'exact solution') for every $n = 1, 2, \dots, \mathcal{N}$, $i=1, 2, \dots, r$.
        \item Let $\widehat{G}^i = x(t_{n-1} + \Delta t c_i)$ (the 'exact stage values') for $i=1,2, \dots, s$, where $c = A \mathbb{1} + Uv$.
        \item Let $\widetilde{G}^i$ be the value of $G^i$ that would be calculated if we used $\hat{g}_{[n-1]}$ as an input vector $g_{[n-1]}$.
    \end{itemize}
    Under these conditions, the following inequalities hold:
    \begin{equation}\label{eq:lemma3ineq1}
        \left\Vert \widehat{G}^i - \varphi(\Delta t) \sum_{j=1}^{s+1} a_{ij} f(\widehat{G}^j) - \sum_{j=1}^r U_{ij} \widehat{g}_{[n-1]}^j\right\Vert \leq (\Delta t)^2 L^2 M \left( \dfrac{1}{2} c_i^2 + \sum_{j=1}^{s+1} |a_{ij} c_j| \right) + O\left((\Delta t)^2\right),
    \end{equation}
    \begin{equation}\label{eq:lemma3ineq2}
    \begin{aligned}    
        &\left\Vert \widehat{g}_{[n]}^i - \varphi(\Delta t) \sum_{j=1}^{s+1} b_{ij} f(\widehat{G}^j) - \sum_{j=1}^r V_{ij} \widehat{g}_{[n-1]}^j\right\Vert \\
        & \hspace{50pt} \leq (\Delta t)^2 L^2 M \left( \dfrac{1}{2} |u_i| + |v_i| + \sum_{j=1}^{s+1} |b_{ij} c_j| \right) + O\left((\Delta t)^2\right),
    \end{aligned}
    \end{equation}
    \begin{equation}\label{eq:lemma3ineq3}
    \begin{aligned}
        & \left\Vert \widehat{g}_{[n]}^i - \varphi(\Delta t) \sum_{j=1}^{s+1} b_{ij} f(\widetilde{G}^j) - \sum_{j=1}^r V_{ij} \widehat{g}_{[n-1]}^j\right\Vert \\
        & \hspace{50pt} \leq (\Delta t)^2 L^2 M \left( \dfrac{1}{2} |u_i| + |v_i| + \sum_{j=1}^{s+1} |b_{ij} c_j| + \tau_0 L \sum_{j=1}^{s+1} |b_j| \varepsilon_j \right) + O\left((\Delta t)^2\right).
    \end{aligned}
    \end{equation}
\end{lemma}

\begin{proof}[Proof of Lemma \ref{lemma3}]
    The proofs of inequalities \eqref{eq:lemma3ineq1} and \eqref{eq:lemma3ineq2} are almost the same as the proof of the standard case (see \cite[Lemma 515A]{butcherbook}), where we only have to use the fact that \eqref{eq:phicond} holds. For the last inequality, let us bound $\left\Vert\widetilde{G} - \widehat{G}\right\Vert$: consider
    $$ \left\Vert \widetilde{G}^i - \widehat{G}^i - \varphi(\Delta t) \sum_{j=1}^{s+1} a_{ij} \left(f(\widetilde{G}^j) - f(\widehat{G}^j) \right)\right\Vert \leq $$
    $$\leq \left\Vert \widetilde{G}^i - \varphi(\Delta t) \sum_{j=1}^{s+1} a_{ij} f(\widetilde{G}^j) -  \sum_{j=1}^r U_{ij} \widehat{g}_{[n-1]}^j - \left( \widehat{G}^i - \varphi(\Delta t) \sum_{j=1}^{s+1} a_{ij} f(\widehat{G}^j) - \sum_{j=1}^r U_{ij} \widehat{g}_{[n-1]}^j  \right) \right\Vert \leq $$
    $$ \leq (\Delta t)^2 L^2 M \left( \dfrac{1}{2} c_i^2 + \sum_{j=1}^{s+1} |a_{ij} c_j| \right) + O\left((\Delta t)^2\right), $$
    where we used the fact that the first terms in the second line are zero (by the definition of $\widetilde{G}^i$), and then inequality \eqref{eq:lemma3ineq1}. To simplify the argument, let us introduce the notation $\eta_i = \left\Vert \widetilde{G}^i - \widehat{G}^i \right\Vert$. Now, by the definition of vector $\varepsilon$,
    $$ \eta_i \leq L \tau_0 \sum_{j=1}^{s+1} |a_{ij}| \eta_j + (\Delta t)^2 L^2 M \sum_{j=1}^{s+1} \left( \delta_{ij} - L \tau_0 |a_{ij}| \right) \varepsilon_j + O\left((\Delta t)^2\right). $$
    By rearranging the terms,
    $$ (I - L \tau_0 \Vert A\Vert_{\infty}) \eta \leq (\Delta t)^2 L^2 M  (I - L \tau_0 \Vert A\Vert_{\infty}) \varepsilon + O\left((\Delta t)^2\right). $$
    Let us multiply both sides by $(I - L \tau_0 \Vert A\Vert_{\infty})^{-1}$ (it exists by the definition of $\tau_0$):
    $$ \eta \leq (\Delta t)^2 L^2 M \varepsilon  + O\left((\Delta t)^2\right), $$
    meaning that $ \left\Vert \widetilde{G}^i - \widehat{G}^i \right\Vert \leq (\Delta t)^2 L^2 M \varepsilon_i  + O\left((\Delta t)^2\right) $. Consequently,
    \begin{equation}\label{eq:lemma3proof}
    \begin{aligned}
        &\left\Vert \varphi(\Delta t) \sum_{j=1}^{s+1} b_{ij} \left( f\left( \widetilde{G}^j \right) - f\left( \widehat{G}^j \right) \right) \right\Vert \leq \tau_0 L \sum_{j=1}^{s+1} |b_{ij}| \left\Vert \widetilde{G}^i - \widehat{G}^i \right\Vert  \\
        & \hspace{120pt} \leq \tau_0 (\Delta t)^2 L^3 M \sum_{j=1}^{s+1} |b_{ij}| \varepsilon_j + O\left((\Delta t)^2\right).
    \end{aligned}
    \end{equation}
    Therefore, from inequalities \eqref{eq:lemma3ineq2} and \eqref{eq:lemma3proof}, we get inequality \eqref{eq:lemma3ineq3}.
\end{proof}
Now we give a bound for the local truncation error.
\begin{lemma}\label{lemma4}
    Under the conditions of Lemma \ref{lemma3}, the local truncation error is given by
    $$ \widehat{g}_{[n]}^i - g_{[n]}^i = \sum_{j=1}^r V_{ij} \left( \widehat{g}_{[n]}^j - g_{[n]}^j\right) + K_{[n]}^i, \qquad i=1, 2, \dots, r, $$
    where
    $$ \Vert K_{[n]} \Vert \leq \Delta t \zeta \max_{i=1}^r \left\Vert \widehat{g}_{[n]}^i - g_{[n]}^i\right\Vert + (\Delta t)^2 \nu + O\left((\Delta t)^2\right), $$
    and $\zeta$ and $\nu$ are given by
    $$ \zeta = L \max_{i=1}^r \sum_{j=1}^{s+1} |b_{ij}| \overline{\varepsilon_j}, \qquad \nu = L^2 M \max_{i=1}^{s+1} \left( \dfrac{1}{2} |u_i| + |v_i| + \sum_{j=1}^{s+1} |b_{ij} c_j| + \tau_0 L \sum_{j=1}^{s+1} |b_{ij}| \varepsilon_j \right), $$
    and $\overline{\varepsilon}$ is defined as
    $$ \sum_{j=1}^{s+1} \left( \delta_{ij} - \tau_0 L |a_{ij}| \right) \overline{\varepsilon}_j = \sum_{j=1}^{s+1} |U_{ij}|, \qquad i=1, 2, \dots, s+1. $$
\end{lemma}
\begin{proof}[Proof of Lemma \ref{lemma4}]
    Since
    $$ g_{[n]}^i - \varphi(\Delta t) \sum_{j=1}^{s+1} b_{ij} f(G^j) - \sum_{j=1}^r V_{ij} g_{[n-1]}^j = 0 $$
    holds, and by inequality \eqref{eq:lemma3ineq3}, we get
    \begin{align*}
        & \left\Vert \widehat{g}_{[n]}^i - g_{[n]}^i - \sum_{j=1}^r V_{ij} \left( \widehat{g}_{[n]}^j - g_{[n]}^j \right) \right\Vert  \\
        & \hspace{20pt}\leq \varphi(\Delta t) \sum_{j=1}^{s+1} |b_{ij}| \left\Vert f(\widetilde{G}^j) - f(G^j) \right\Vert \\
        & \hspace{40pt} + (\Delta t)^2 L^2 M^2  \left( \dfrac{1}{2} |u_i| + |v_i| + \sum_{j=1}^{s+1} |b_{ij} c_j| + \tau_0 L \sum_{j=1}^{s+1} |b_{ij}| \varepsilon_j \right) + O\left((\Delta t)^2\right) 
    \end{align*}
    \begin{equation}\label{eq:lemma4proof}
    \begin{aligned}
    & \leq \Delta t L \sum_{j=1}^{s+1} |b_{ij}| \left\Vert \widetilde{G}^j - G^j \right\Vert \\
    & \hspace{20pt} + (\Delta t)^2 L^2 M^2  \left( \dfrac{1}{2} |u_i| + |v_i| + \sum_{j=1}^{s+1} |b_{ij} c_j| + \tau_0 L \sum_{j=1}^{s+1} |b_{ij}| \varepsilon_j \right) + O\left((\Delta t)^2\right).  
      \end{aligned}
    \end{equation}
    Since
    $$ \left\Vert \widetilde{G}^j - G^j - \sum_{k=1}^r U_{jk} \left( \widehat{g}_{[n-1]}^k - g_{[n-1]}^k \right)  \right\Vert \leq \varphi(\Delta t) L \sum_{k=1}^{s+1} |a_{jk}|\left\Vert \widetilde{G}^k - G^k \right\Vert, $$
    we can get an estimate for $\left\Vert \widetilde{G}^j - G^j  \right\Vert$:
    $$ \left\Vert \widetilde{G}^j - G^j  \right\Vert - \varphi(\Delta t) L \sum_{k=1}^{s+1} |a_{jk}|\left\Vert \widetilde{G}^k - G^k \right\Vert \leq \sum_{k=1}^r |U_{jk}| \left\Vert \widehat{g}_{[n-1]}^k - g_{[n-1]}^k \right\Vert. $$
    To make the notations easier, let us introduce $\mu_j := \left\Vert \widetilde{G}^j - G^j  \right\Vert$. Therefore, the previous bound can be rearranged as
    $$ (I - \tau_0 L \Vert A \Vert_{\infty}) \mu \leq |U| \left\Vert \widehat{g}_{[n-1]} - g_{[n-1]} \right\Vert, $$
    $$ \mu \leq (I - \tau_0 L \Vert A \Vert_{\infty})^{-1} |U| \left\Vert \widehat{g}_{[n-1]} - g_{[n-1]} \right\Vert, $$
    which means that (by the definition of $\overline{\varepsilon}$)
    $$ \left\Vert \widetilde{G}^j - G^j  \right\Vert \leq \overline{\varepsilon}_j \max_{k=1}^r \left\Vert \widehat{g}_{[n-1]}^k - g_{[n-1]}^k \right\Vert. $$
    By substituting this bound into \eqref{eq:lemma4proof}, we get the statement.
\end{proof}

\begin{lemma}\label{lemma5}
    Let us introduce the following notation for the accumulated errors we get at different points on interval $[t_0,\hat{t}]$ after taking $n$-many steps with stepsize $\Delta t = (\hat{t} - t_0)/n$:
    $$ E_{[i]} = \begin{pmatrix}
        \widehat{g}_{[i]}^1 - g_{[i]}^1\\[5pt]
        \widehat{g}_{[i]}^2 - g_{[i]}^2\\[5pt]
        \vdots \\[5pt]
        \widehat{g}_{[i]}^r - g_{[i]}^r
    \end{pmatrix}, \qquad i=0,1, 2, \dots n. $$
    Then, the following bound holds for the error at the final step:
    $$ \Vert E_{[n]} \Vert \leq \begin{cases}
        \exp \left( \zeta C (\hat{t} - t_0)\right) \Vert E_{[0]} \Vert + \frac{\nu \Delta t}{\zeta} \left( \exp(\zeta C(\hat{t}-t_0))-1 \right) + O\left((\Delta t)^2\right), \quad \text{if }\; \zeta >0,\\
        \Vert E_{[0]} \Vert + \nu C (\hat{t} - t_0) \Delta t + O\left((\Delta t)^2\right), \hspace{135pt} \text{if }\; \zeta=0, 
        \end{cases} $$
        where $C= \sup_{k \in \mathbb{N}} \Vert V^k \Vert_{\infty}$ and the norm of $E_{[n]}$ is thought of as the maximum of the norms of all of its elements.
\end{lemma}

\begin{proof}[Proof of Lemma \ref{lemma5}]
    Let us rewrite the result of Lemma \ref{lemma4} as
    $$ E_{[i]} = (V^i \otimes I) E_{[i-1]} + K_{[i-1]}. $$
    Similarly to the proof in the standard case (see \cite[Lemma 515C]{butcherbook}), we get the bound
    $$ \Vert E_{[i]} \Vert \leq C \Vert E_{[0]} \Vert + \sum_{j=0}^{i-1} C \Vert K_{[i-j]}\Vert.$$
    By substituting the bounds of $K_{[i-j]}$ from Lemma \ref{lemma4}, the inequality can be rewritten as
    $$ \Vert E_{[i]} \Vert \leq C \Vert E_{[0]} \Vert + \Delta t \zeta C \sum_{j=0}^{i-1} \Vert E_{[j]} \Vert + C i \nu (\Delta t)^2 + O\left((\Delta t)^2\right),$$
    where $\zeta$ and $\nu$ are the same as in Lemma \ref{lemma4}. Let us define the values of $\xi_j$ as $\xi_0 = C \Vert E_{[0]} \Vert$ and
    $$ \xi_i = \xi_0 + \Delta t \zeta C \sum_{j=0}^{i-1} \xi_j + C i \nu (\Delta t)^2 + O\left((\Delta t)^2\right). $$
    Consequently,
    $$ \xi_i - \xi_{i-1} = \Delta t \zeta C \xi_{i-1} + C i \nu (\Delta t)^2 + O\left((\Delta t)^2\right), $$
    from which we get the formula
    $$ \xi_i = \xi_0 (1+\zeta \Delta t C)^i + \dfrac{\nu}{\zeta} \Delta t \left( (1+\zeta \Delta t C)^i - 1 \right) + O\left((\Delta t)^2\right), $$
    or, if $\zeta=0$,
    $$ \xi = \xi_0 + C i \nu (\Delta t)^2 + O\left((\Delta t)^2\right). $$
    If $i=n$, then we get the statement of the lemma.
\end{proof}
It is easy to see from Lemma \ref{lemma5} that if the method is consistent and stable (which were used during the proofs of the previous Lemmas), the method is convergent, which completes the proof of the theorem.
\end{proof}

\begin{proof}[Proof of Lemma \ref{lem:taylor}.]
    First, let us assume that (C1) is true. We prove the validity of (C2) by induction. The statement is trivially true if $p=0$. Let us assume that it is true for $p=n-1$. The key observation here is that
    $$ R_{n-1} = \dfrac{1}{n!} f^{(n)} (a) (\varphi(\delta), \dots, \varphi(\delta)) + R_n. $$
    (This is also true in the $n=1$ case.) The reason for this is that by partial integration, 
    \begin{align*}
        & \int_0^1 f^{(n)}(a + \xi \varphi(\delta)) (\varphi(\delta), \dots \varphi(\delta)) (1-\xi)^{n-1} d\xi =\\
        & \hspace{20pt} \dfrac{1}{n} f^{(n)}(a)(\varphi(\delta), \dots, \varphi(\delta)) + \dfrac{1}{n} \int_0^1 f^{(n+1)}(a + \xi \varphi(\delta)) (\varphi(\delta), \dots, \varphi(\delta)) d \xi.
    \end{align*}
    Consequently,
    $$ R_{n-1} = \dfrac{1}{(n-1)!} \left( \dfrac{1}{n} f^{(n)}(a)(\varphi(\delta), \dots, \varphi(\delta)) + \dfrac{1}{n} \int_0^1 f^{(n+1)}(a + \xi \varphi(\delta)) (\varphi(\delta), \dots, \varphi(\delta)) d \xi \right) = $$
    $$ = \dfrac{1}{n!} f^{(n)} (a) (\varphi(\delta), \dots, \varphi(\delta)) + R_n. $$
    This proves the validity of (C2).

    Now, let us assume that (C2) is true. Let us substitute the standard Taylor series of $\varphi$ around $0$ into formula (C2):
    \begin{align*}
        & f(a+\delta) = f(a) + f'(a)\left( \delta + \sum_{\ell=1}^p \dfrac{\varphi^{(\ell)}(0)}{\ell!} (\delta, \dots, \delta) + O(\Vert \delta \Vert^{p+1}) \right) + \dots +  \\
    & \dfrac{1}{p!} f^{(p)}(a) \left(\delta + \sum_{\ell=1}^p \dfrac{\varphi^{(\ell)}(0)}{\ell!} (\delta, \dots, \delta) + O(\Vert \delta \Vert^{p+1}) , \dots, \delta + \sum_{\ell=1}^p \dfrac{\varphi^{(\ell)}(0)}{\ell!} (\delta, \dots, \delta) + O(\Vert \delta \Vert^{p+1})\right) \\
    & \hspace{300pt} + R_{p+1}.
    \end{align*}
    This should coincide with the standard Taylor series of $f$ (with error $O(\Vert \delta \Vert^{p+1})$), which can only hold if (C1) is true.
\end{proof}

\begin{proof}[Proof of Theorem \ref{th:order}]
    Assume that for the error vector $\widetilde{E}_{[i]}$ (defined in Lemma \ref{lemma5}) of the standard method \eqref{eq:SMMM}, we have the following bound using its (standard) Taylor series:
    $$ \Vert \widetilde{E}_{[i]} \Vert \leq \sum_{k=1}^{\infty} d_k (\Delta t)^k.$$
    (For the exact forms of the coefficients $d_k$, see, for example, \cite[Th. 532A]{butcherbook}.) Since the order of the standard method is $p$, $d_k =0$ should hold for $k=1, 2, \dots, p-1$. If condition (C1) is true, then by Lemma \ref{lem:taylor}, (C2) is also true, meaning that we can construct a similar series for the error of the nonstandard method $E_{[i]}$ by the same steps:
    $$ \Vert E_{[i]} \Vert \leq \sum_{k=1}^{p} d_k (\varphi(\Delta t))^k + O((\Delta t)^{p+1}).$$
    Since (C1) holds, we get order $p$ for the nonstandard method too.

    If we assume that (C3) is true, i.e., the nonstandard method has an order of $p$, then it means that $\Vert E_{[i]} \Vert \leq O((\Delta t)^p)$. If we consider the standard Taylor series of the error of the nonstandard method, it will contain at least one term in the form $\varphi(\Delta t) d_0$, where $d_0$ does not depend on $\Delta t$. Therefore, for this method to have an order of $p$, (C1) should also hold.
\end{proof}

\begin{proof}[Proof of Theorem \ref{th:preserv}]
    We prove the statement by induction. Since the first step and an arbitrary step are the same, we only present the proof of an arbitrary step. Moreover, to simplify the notation, let us fix $\mathcal{R}$ as $\leq$ (all the other cases can be proven similarly). 

    Assume that $\mathcal{P} y_j \leq \mu$ holds for $j=1, \dots, n-1$ (and this is also true for all the inner stages), and now we would like to show that $\mathcal{P} y_n \leq \mu$ is also true ($n \leq \mathcal{N}$). Note that the nonstandard multistep multistage method \eqref{eq:NSMMM} can be written in the form 
    \begin{equation*}%\label{eq:NSMMM_alpha}
\begin{aligned}
    y_{[n-1]}^{(1)} = & \; y_{n-1},\\%[10pt]
   y_{[n-1]}^{(i)} = & \; \sum_{\ell=2}^k \left(\sum_{j=1}^{s} \alpha_{[n-\ell]}^{i,j} \left( y_{[n-\ell]}^{(j)} + \dfrac{\beta_{[n-\ell]}^{i,j}}{\alpha_{[n-\ell]}^{i,j}} \; \varphi(\Delta t) \; f\left(y_{[n-\ell]}^{(j)} \right)\right) \right) + \\%[10pt]
  & + \sum_{j=1}^{i-1} \alpha_{[n-1]}^{i,j} \left(y_{[n-1]}^{(j)} + \dfrac{\beta_{[n-1]}^{i,j}}{\alpha_{[n-1]}^{i,j}} \; \varphi(\Delta t) \; f\left(y_{[n-1]}^{(j)}\right)\right) \quad (i=2,\; 3, \; \dots,\; s+1),\\%[10pt]
  y_n \quad \;\; = & \; y_{[n-1]}^{(s+1)}.
  \end{aligned}
\end{equation*}
Note that the expressions multiplied by the $\alpha_{[n-\ell]}^{i,j}$ terms can be thought of as a forward Euler step with stepsizes $\frac{\beta_{[n-\ell]}^{i,j}}{\alpha_{[n-\ell]}^{i,j}} \; \varphi(\Delta t)$. The form of the bound of $\varphi$ guarantees that these forward Euler steps will preserve the desired properties. 

We prove the statement for every inner stage $y_{[n-1]}^{(i)}$ by induction. It is obvious that $\mathcal{P} y_{[n-1]}^{(1)} \leq \mu$ holds. Let us assume that the statement is true for $i=1, 2, \dots, m-1$. Then,
%\newpage
\begin{align*}
    & \mathcal{P} y_{[n-1]}^{(m)} = \sum_{\ell=2}^k \left(\sum_{j=1}^{s} \alpha_{[n-\ell]}^{m,j} \mathcal{P} \left( y_{[n-\ell]}^{(j)} + \dfrac{\beta_{[n-\ell]}^{m,j}}{\alpha_{[n-\ell]}^{m,j}} \; \varphi(\Delta t) \; f\left(y_{[n-\ell]}^{(j)} \right)\right) \right) \\
    & \hspace{120pt} + \sum_{j=1}^{m-1} \alpha_{[n-1]}^{m,j} \mathcal{P} \left(y_{[n-1]}^{(j)} + \dfrac{\beta_{[n-1]}^{m,j}}{\alpha_{[n-1]}^{m,j}} \; \varphi(\Delta t) \; f\left(y_{[n-1]}^{(j)}\right)\right) \leq
\end{align*}
$$ \leq \sum_{\ell=2}^k \left(\sum_{j=1}^{s} \alpha_{[n-\ell]}^{m,j} \mu \right) + \sum_{j=1}^{m-1} \alpha_{[n-1]}^{m,j} \mu = \mu \sum_{\ell=1}^k \left(\sum_{j=1}^{s} \alpha_{[n-\ell]}^{m,j} \right) = \mu, $$
where we used the fact that for a multistep multistage method to be consistent, $\sum_{\ell=1}^k \sum_{j=1}^{s} \alpha_{[n-\ell]}^{i,j} = 1$ should hold for every $i= 1, 2, \dots, s+1$ (see \cite{butcherbook}, or just consider the differential equation $x'(t)=0$, $x(0)=a$, $a \neq 0$). Therefore, if $m=s+1$, then $\mathcal{P} y_n \leq \mu$ is also true.
\end{proof}

\section{Numerical experiments}\label{sec:num}
In the following pages, we examine the performance of the nonstandard multistep multistage methods defined earlier. Before the specific applications, we list the general steps of the construction of a given method.
\begin{enumerate}
    \item Assume that the original differential equation \eqref{eq:ODE} attains some of the qualitative properties mentioned in Section \ref{sec:preserv}. The numerical method aims to preserve these properties.
    \item Apply the forward Euler method to the given equation with $\Delta t$ and calculate such a bound $B_{FE}$ for which the properties considered in Step 1 are preserved if $\Delta t \leq B_{FE}$ holds.
    \item Choose function $\varphi(x)$ in a way that $\varphi(x)\leq \mathcal{C} B_{FE}$. Moreover, to attain order $p$, the first $p$-many terms of the Taylor series around zero of $\varphi$ and $g(x)=x$ should coincide.
    \item For a method with $k$-many steps and $s+1$-many inner stages, the values $y_{[n]}^{(i)}$ ($n=0, \dots, k-1$, $i=1, \dots, s+1$) should be calculated. For this, a nonstandard Runge-Kutta method can be used with an appropriate function $\varphi(x)$ (see Remark \ref{rem:initialRK}.). 
\end{enumerate}

\begin{rem}\label{rem:phi}
    Up until this point, the choice of function $\varphi$ was not discussed. One possible choice is $\varphi_4(x) = \dfrac{\mathcal{C} B_{FE} x}{\left( \left( \mathcal{C} B_{FE} \right)^4 + x^4\right)^{1/4}}$. It is easy to see that $\varphi(x) = x + O(x^5)$ holds for small values of $x$, meaning that its application allows orders up to $4$, while also being bounded by $\mathcal{C} B_{FE}$. In general, the choice $\varphi_p(x) = \dfrac{\mathcal{C} B_{FE} x}{\left( \left( \mathcal{C} B_{FE} \right)^p + x^p\right)^{1/p}}$ works for order $p$. As it turned out in \cite{takacs26} (and as we will also see in the following numerical experiments), this choice is also recommended in the case of lower-order (i.e., second- or third-order) methods, since in some cases this choice allows convergence of high order, while still having the same computational cost as the choices $\varphi_2$ and $\varphi_3$. A list of other possible choices for $\varphi$ can be found in \cite{takacs26}. We should also note here that in \cite{dang}, Dang and Hoang demonstrated a different construction of functions fulfilling these conditions up to an arbitrary order.
\end{rem}

\subsection{One dimensional logistic equation}\label{sec:logistic}
Let us consider the one-dimensional ($N=1$) logistic growth equation with constant $c \in \mathbb{R}^+$ (and $t_0 = 0$)
\begin{equation}\label{eq:logistic}
   \begin{cases}
       x'(t) & = x(t) \left( c - x(t) \right),\\
       x(0) & = \tilde{x}.
   \end{cases}
\end{equation}
The solution of the equation is $x(t) = \dfrac{c e^{c t}\tilde{y}}{\tilde{y} (e^{ct}-1)+c}$, with equilibrium points at $x_1^*=0$ and $x_2^*=c$, and invariant sets $(-\infty, 0)$, $(0, c)$ and $(c, \infty)$, making the solution bounded by the values of the equilibrium points.

In accordance with the steps outlined at the beginning of this section, we first observe the behavior of the forward Euler method.

\begin{prop}(\cite[Prop. 1.]{takacs26})\label{prop:FE_logistic}
    The forward Euler method applied to equation \eqref{eq:logistic} preserves the boundedness of the solution in the $\tilde{x}\geq 0$ case, if $\Delta t \leq \min \left\{\dfrac{1}{c}, \dfrac{1}{\tilde{x}} \right\}$. If $\tilde{x}<0$, then the method preserves the properties unconditionally.
\end{prop}

Consequently, the choice $B_{FE} = \min \left\{ \dfrac{1}{c}, \dfrac{1}{\tilde{x}} \right\}$ is a proper one, meaning that in these cases the sets $(-\infty,0)$, $(0,c)$ and $(c,\infty)$ should be invariant for any choice of $\Delta t$. Since the exact solution is known, the starting procedure can be the evaluation of the exact solution at the given points. Later, in Section \ref{sec:SEIR}, we will use high-order nonstandard Runge-Kutta methods to approximate these values.

In the following, we run several numerical tests analyzing the performance of the aforementioned numerical methods. For this, we use the strongly stable forms of multistep multistage methods \cite{const} - the coefficients of the methods applied are listed in Appendix A.

\subsubsection{The non-stiff case of $c=10$}
Let us first consider equation \eqref{eq:logistic} with $c=10$. First, we test the orders of different nonstandard multistep multistage methods: we run the nonstandard versions of multistep multistage methods GLp2q2s3k3, GLp3q3s3k3, GLp3q2s3k2, GLp4q3s3k3, and GLp4q4s3k3 (denoted by the name of the standard version with an N added to the beginning) with final time $T=0.1$, initial condition $\tilde{x}=5$ and timesteps $\Delta t =0.01/2^k, \; k=0, 1,  \dots, 9$. We also plot the errors for the nonstandard Runge-Kutta and nonstandard multistep methods. The values of the errors (with the calculated slopes) can be found in Table \ref{tab:nonstiff_errors} in Appendix B. As we can see on the left panels of Figure \ref{fig:nonstiff_orders}, the nonstandard multistep multistage methods attain the same orders as their standard counterparts. Note that their errors are always between the errors of the Runge-Kutta and the multistep methods, sometimes even performing better than the Runge-Kutta ones (see the second-order methods on the upper left panel). As it will turn out later, the reason multistep multistage methods could be favored instead of the Runge-Kutta ones is their behavior in the case of larger timesteps; see the following experiments. One more thing to note is that while the Runge-Kutta NSSPMS(4,3) and multistep method NSSPRK(3,3) can have orders higher than $3$, this behavior is absent in the case of third-order multistep multistage methods NGLp3q3s3k3 and NGLp3q2s3k2.

%\FloatBarrier
\begin{figure}[!htbp]
    \centering
    \includegraphics[width=0.45\linewidth]{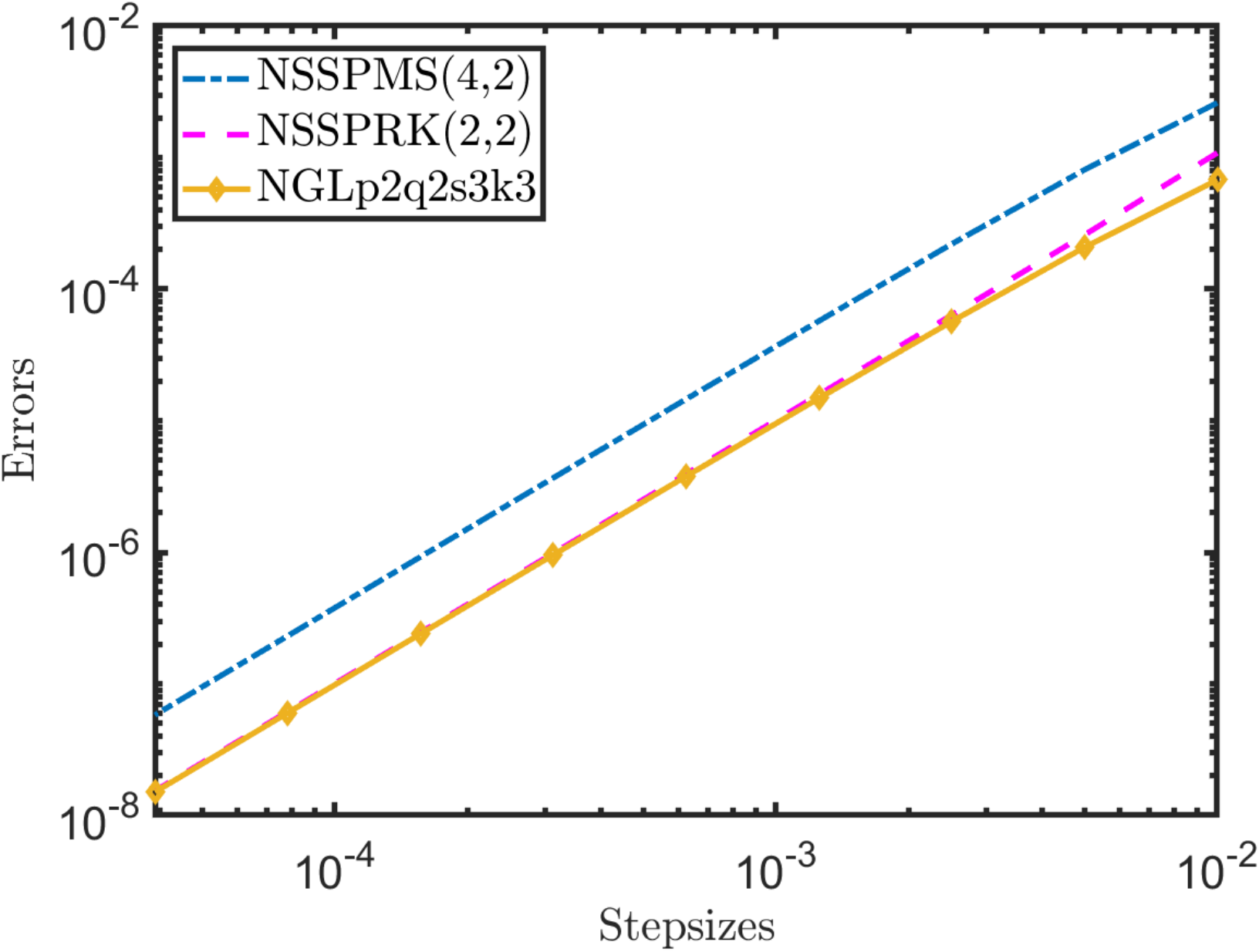}
    \includegraphics[width=0.45\linewidth]{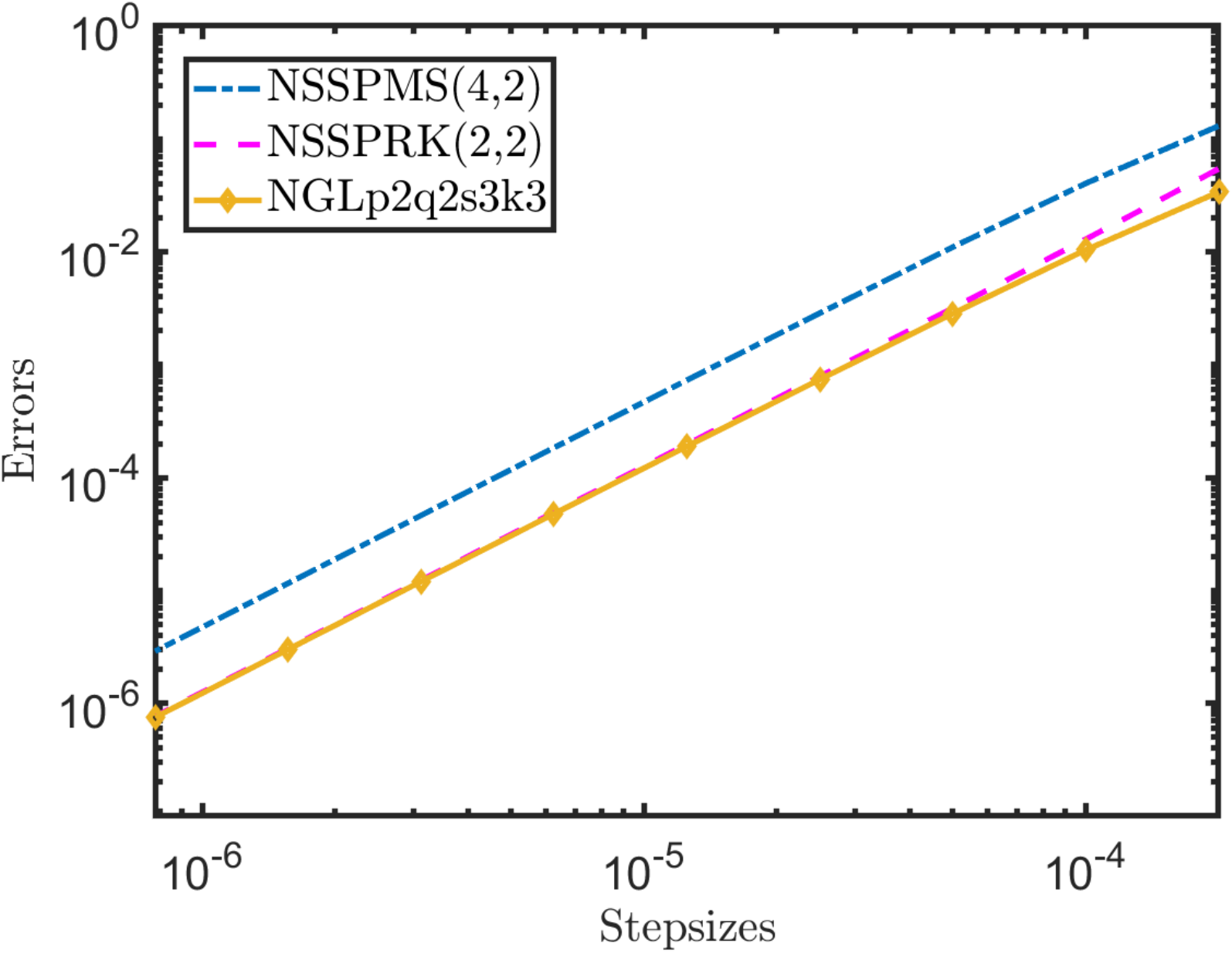}
    \includegraphics[width=0.45\linewidth]{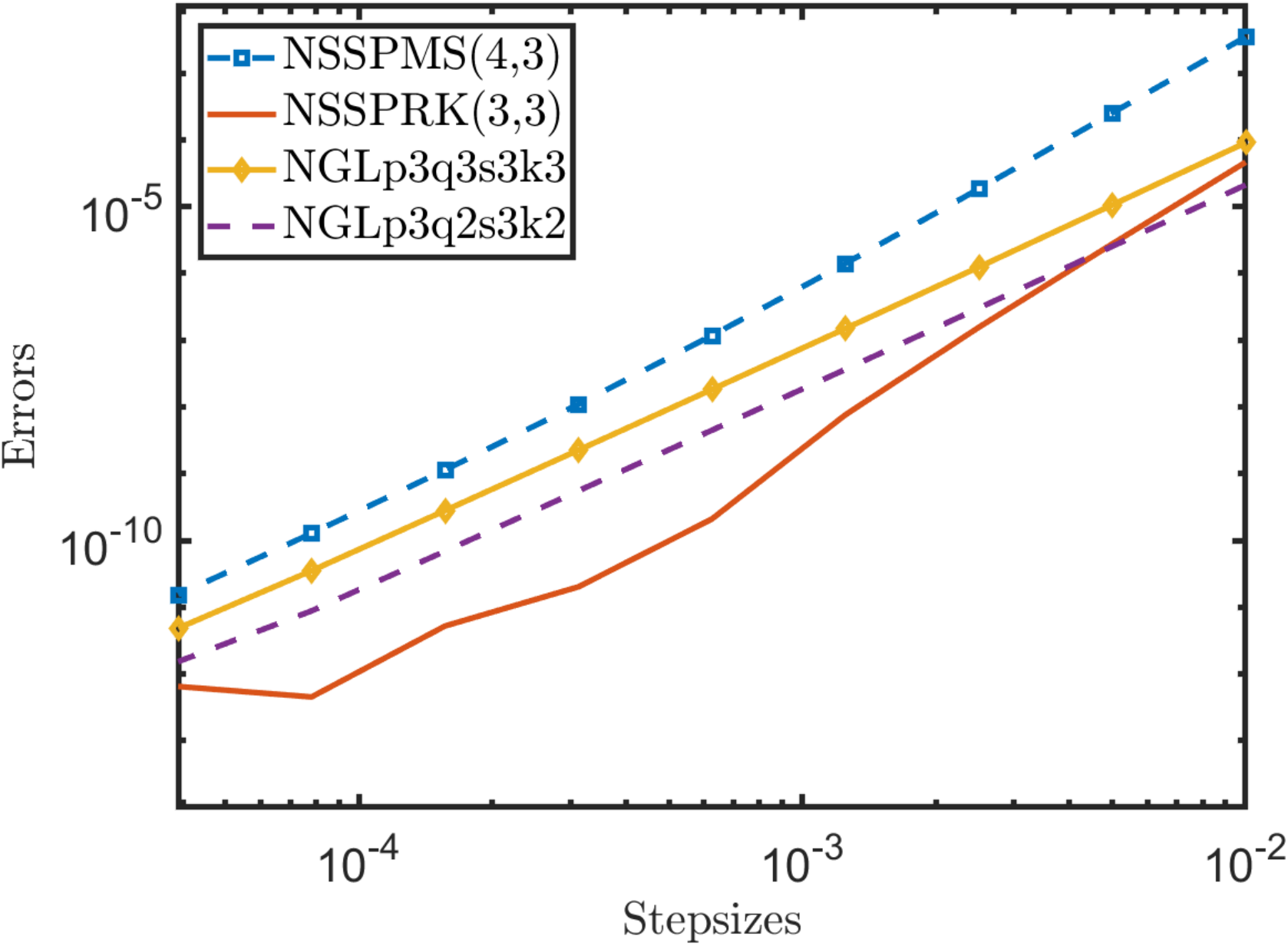}
    \includegraphics[width=0.45\linewidth]{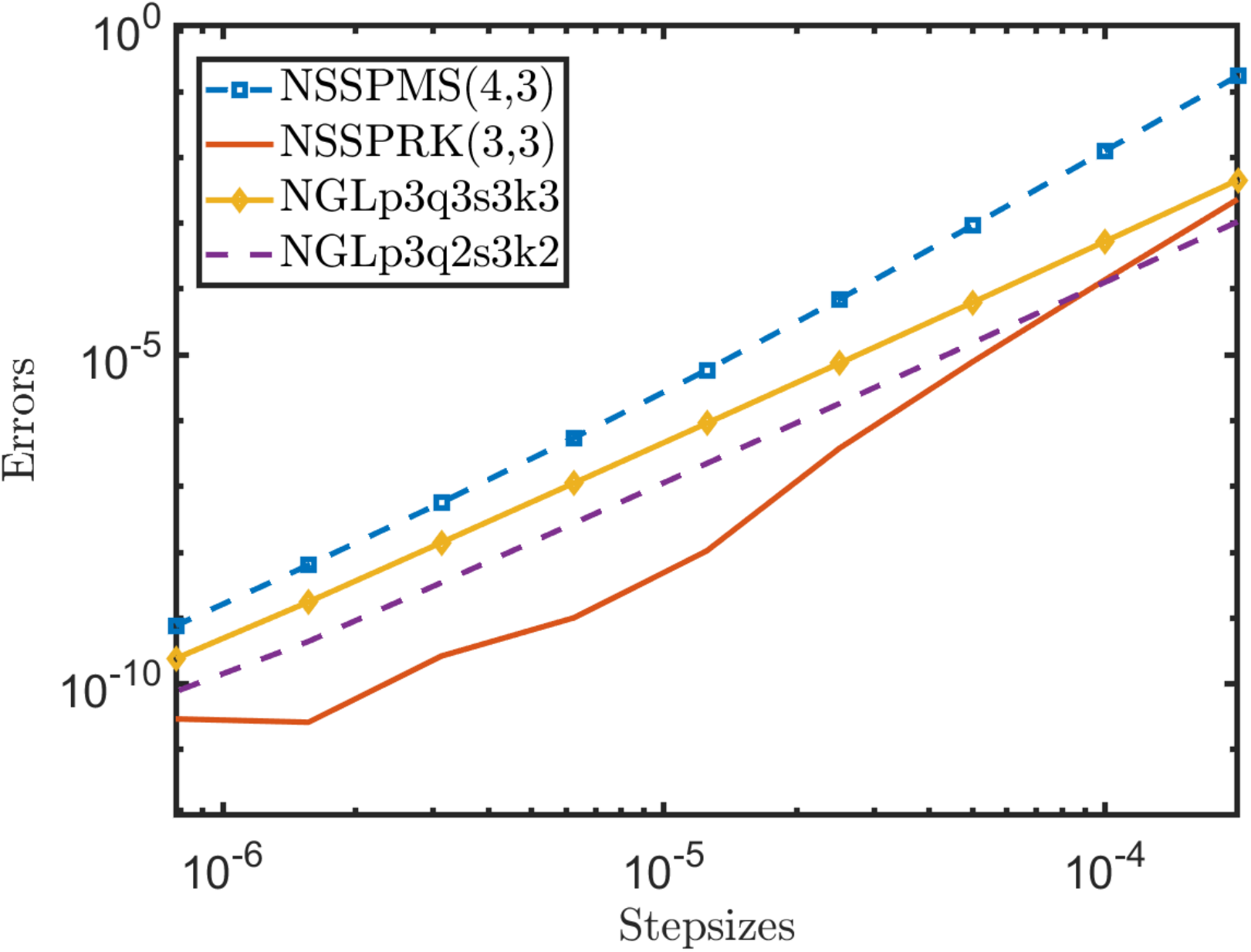}
    \includegraphics[width=0.45\linewidth]{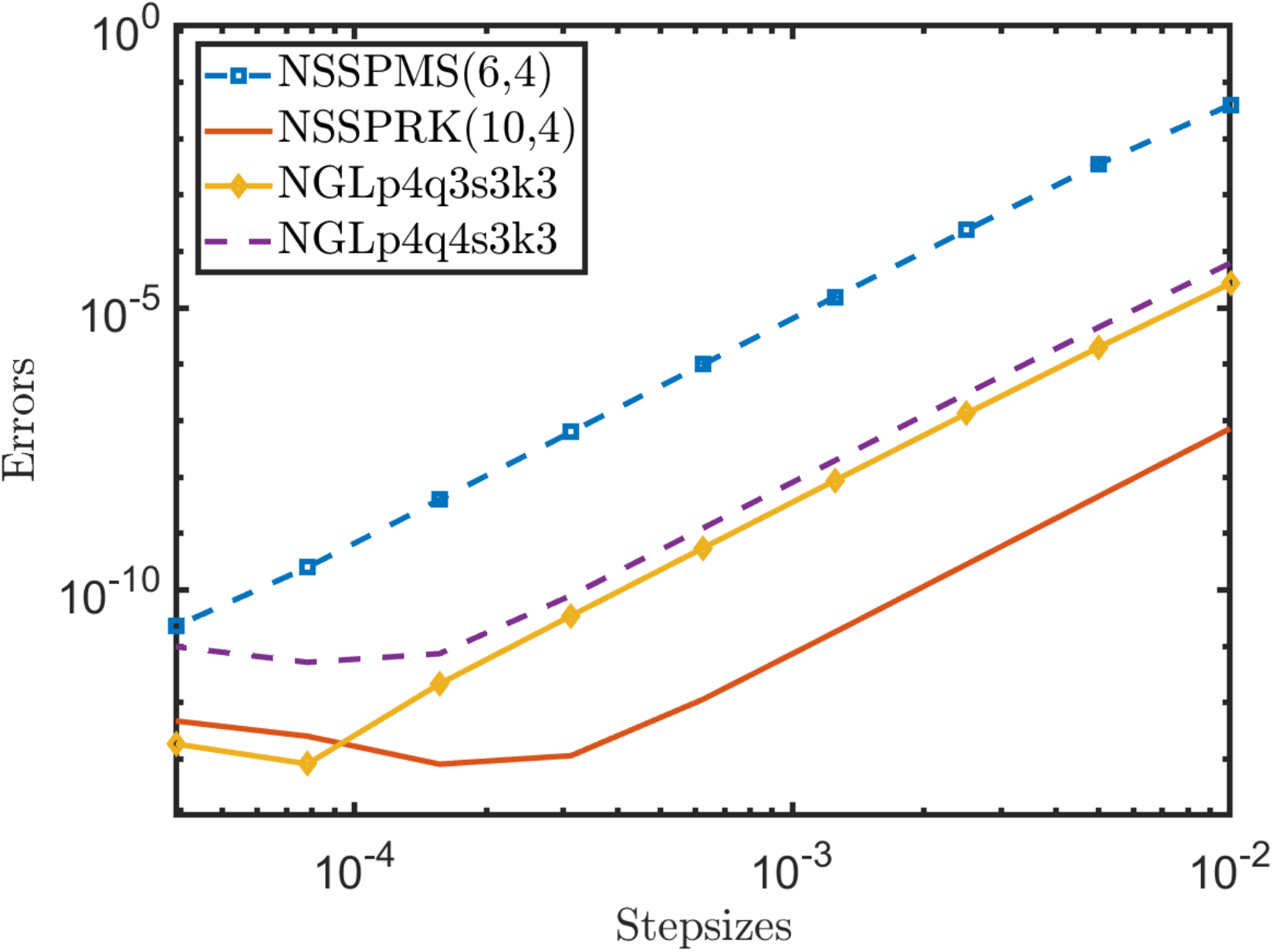}
    \includegraphics[width=0.45\linewidth]{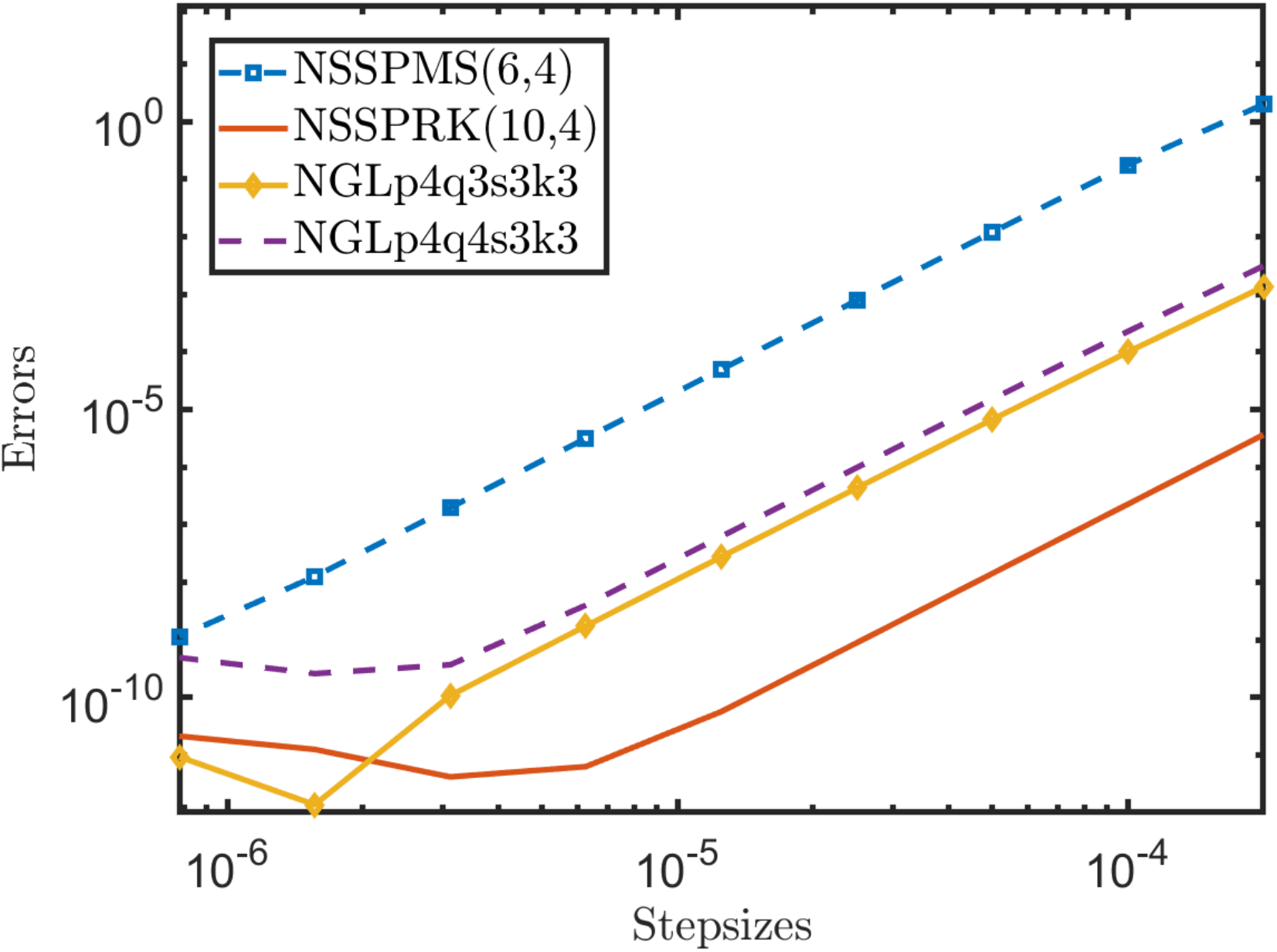}
    \caption{The errors of the different nonstandard methods for different values of $\Delta t$ with $c=10$ (left panels) and $c=500$ (right panels). The methods are grouped according to their expected orders.}
    \label{fig:nonstiff_orders}
\end{figure}
%\FloatBarrier

%\lipsum

%\iffalse
Next, we examine how these methods preserve the qualitative properties of the continuous model. We run the different schemes with $T=10$, $\tilde{x}=15$, and two different timesteps $\Delta t_1$ and $\Delta t_2=T/500$ (the value of $\Delta t_1$ varies across the different methods, see the corresponding figures). The behavior of the second-, third-, and fourth-order methods can be seen in Figures \ref{fig:nonstiff_compare2}, \ref{fig:nonstiff_compare3}, and \ref{fig:nonstiff_compare4}, respectively. In the case of second-order methods, the nonstandard multistep multistage method performs the best, while all nonstandard methods remain bounded. In this case, the standard multistep multistage method is also close to the exact solution for small values of $n$, but oscillations appear around $t=6$, and the scheme does not preserve the boundedness properties. 

%\FloatBarrier
\begin{figure}[!htbp]
    \centering
    \includegraphics[width=0.45\linewidth]{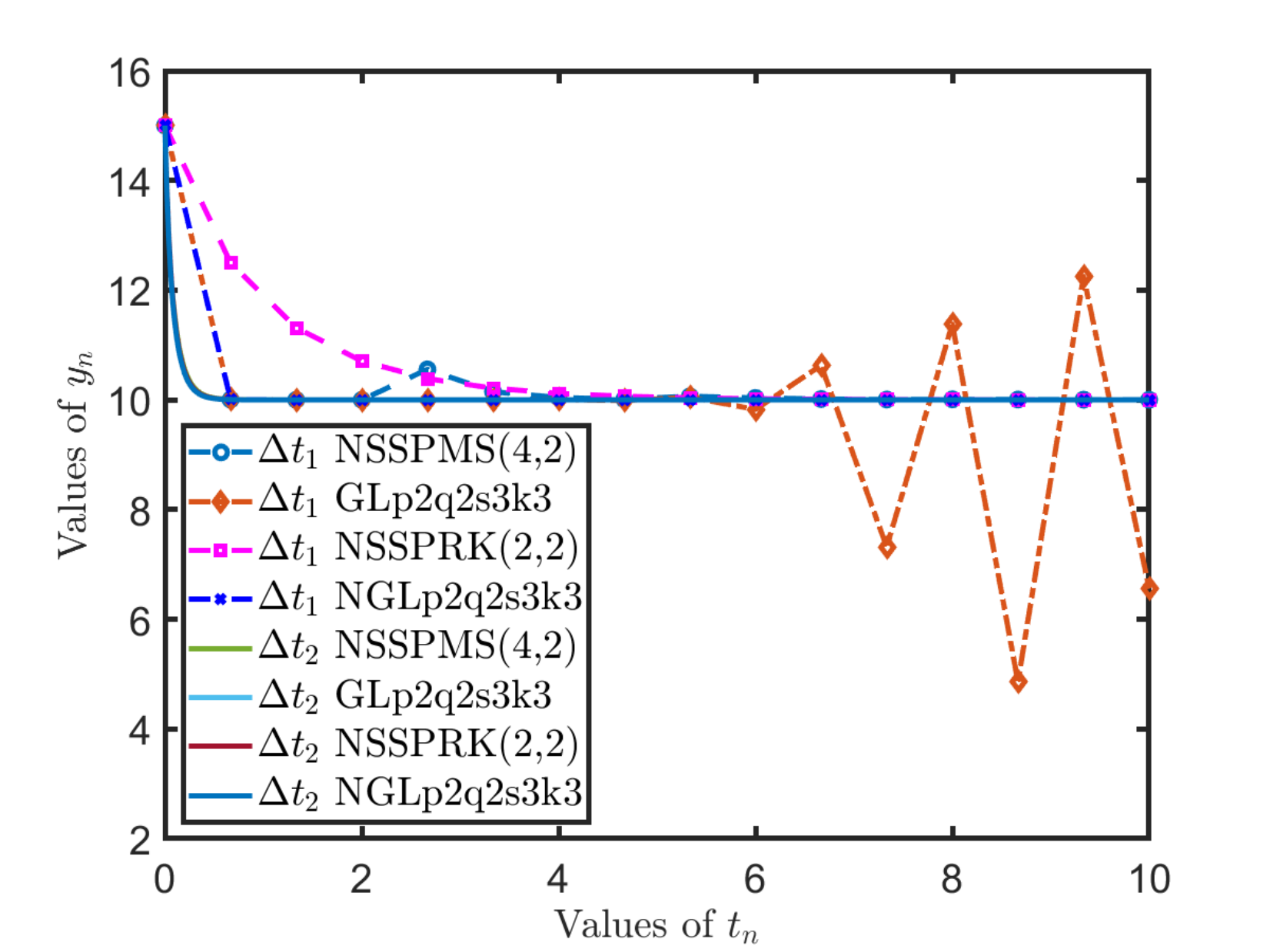}
    \caption{The plots of the different second-order methods solving the logistic equation with $\Delta t_1 = T/15$ and $\Delta t_2 = T/500$ for $c=10$. For $\Delta t_2$, all of the plots are very close to each other, making only one of them visible.}
    \label{fig:nonstiff_compare2}
\end{figure}

% \begin{figure}%[!h]
%     \centering
%     % \includegraphics[width=0.4\linewidth]{nonstiff_logistic_compare_order3_justnonstandard_GLdiff_c10_A.png}
%     % \includegraphics[width=0.408\linewidth]{nonstiff_logistic_compare_order3_justnonstandard_GLdiff_c10_B.png}
%     % \includegraphics[width=0.4\linewidth]{nonstiff_logistic_compare_order3_justmmm_both.png}
%     % \includegraphics[width=0.4\linewidth]{nonstiff_logistic_compare_order3_justmmm_juststable.png}
%     \includegraphics[width=0.4\linewidth]{Fig3A.pdf}
%     \includegraphics[width=0.4\linewidth]{Fig3B.pdf}
%     \includegraphics[width=0.4\linewidth]{Fig3C.pdf}
%     \includegraphics[width=0.4\linewidth]{Fig3D.pdf}
%     \caption{The plots of the different third-order methods with $\Delta t_2 = T/500$ and $\Delta t_1 = T/10$ (upper left), $\Delta t_1 = T/30$ (lower left) and $\Delta t_1 = T/15$ (lower right) for $c=10$. For $\Delta t_2$, all of the plots are very close to each other, making only one of them visible. In the upper right panel, the difference between the two third-order nonstandard multistep multistage methods from the upper left panel is plotted: since their difference is very small, they are indistinguishable on the different plots.}
%     \label{fig:nonstiff_compare3}
% \end{figure}
%\FloatBarrier

A similar behavior is present in the case of third-order methods in Figure \ref{fig:nonstiff_compare3}. In the upper left panel, we plot the nonstandard multistep, nonstandard Runge-Kutta, and nonstandard multistep multistage methods with $\Delta t_1 = T/10$. It is apparent that the multistep multistage methods are the ones closest to the exact solution. Since these are indistinguishable from each other, the difference between the methods is also plotted in the upper right panel. We also compare these nonstandard methods with standard ones on the lower panels. On the lower left one, the standard and nonstandard multistep multistage methods are plotted with $\Delta t_1 = T/30$. In this case, method GLp3q3s3k3 starts to behave in an unbounded way, while GLp3q2s3k2 exhibits the desired properties. If we increase the timestep to $\Delta t_1 = T/10$, method GLp3q2s3k2 also loses the boundedness property. Note that in both cases the nonstandard methods behave as expected. It is also worth mentioning that the reason the standard methods are not present in the upper left panel is that for $\Delta t_1 = T/10$, both of them are highly unstable.

\begin{figure}[!htbp]
    \centering
    \includegraphics[width=0.4\linewidth]{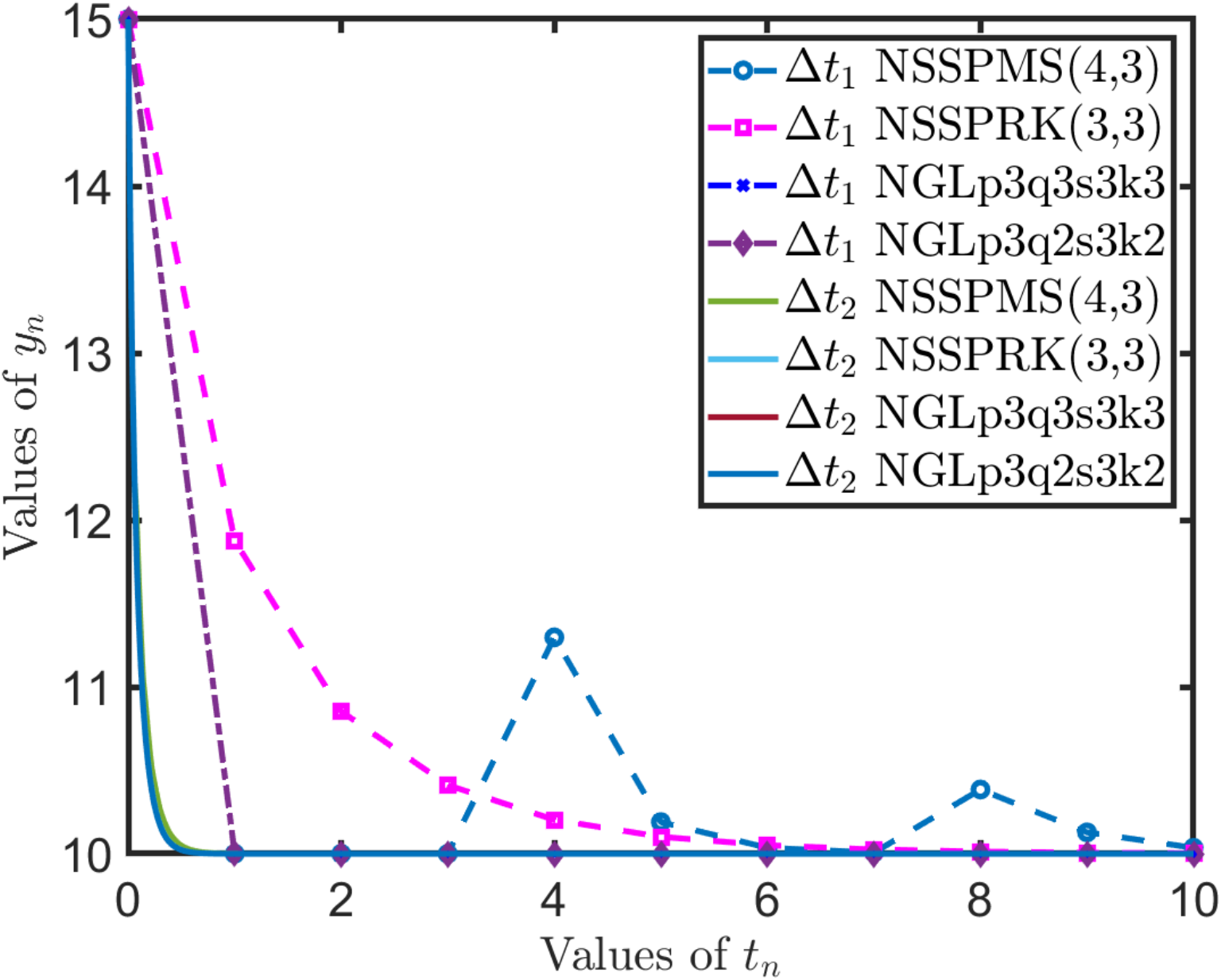}
    \includegraphics[width=0.4\linewidth]{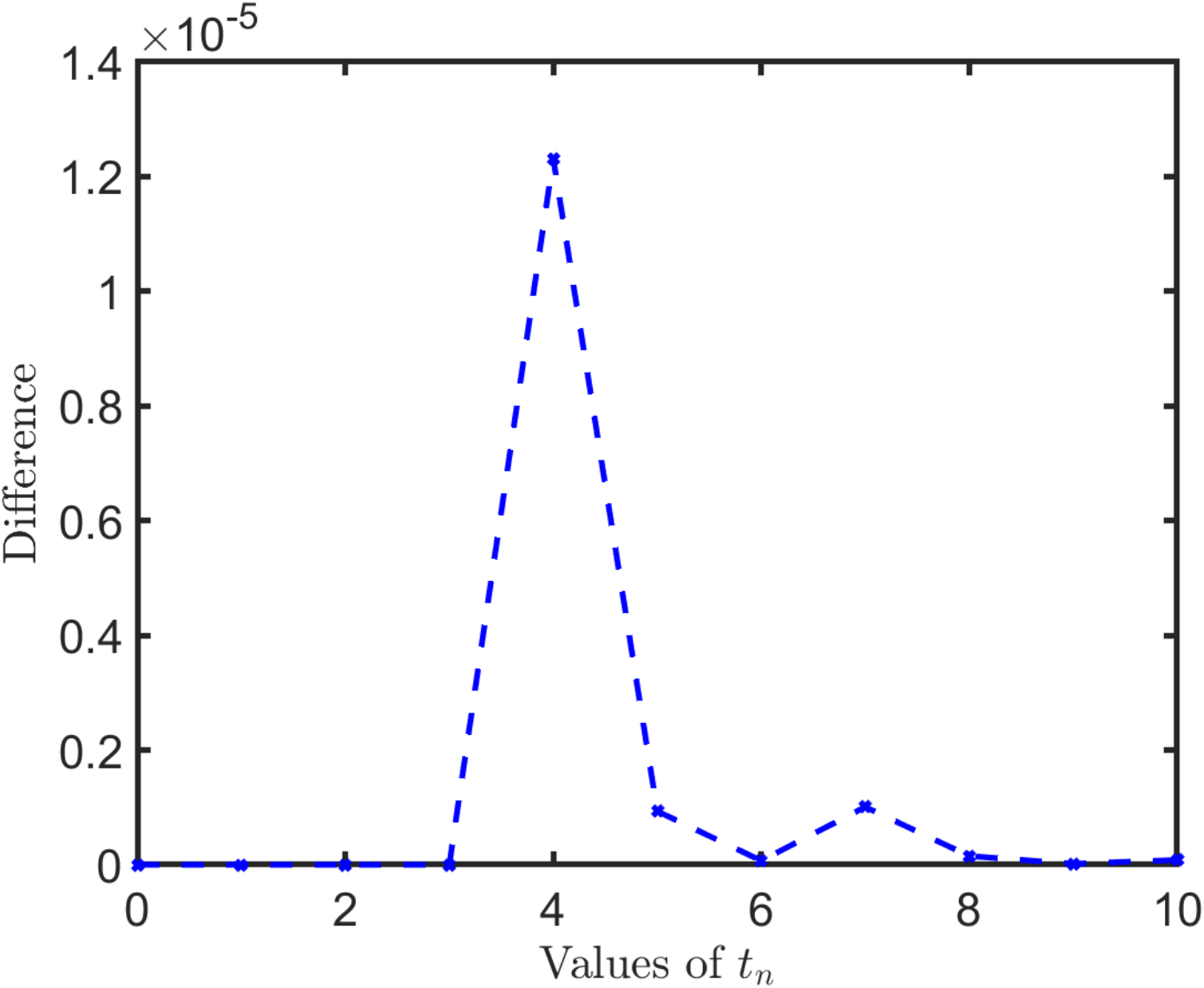}
    \includegraphics[width=0.4\linewidth]{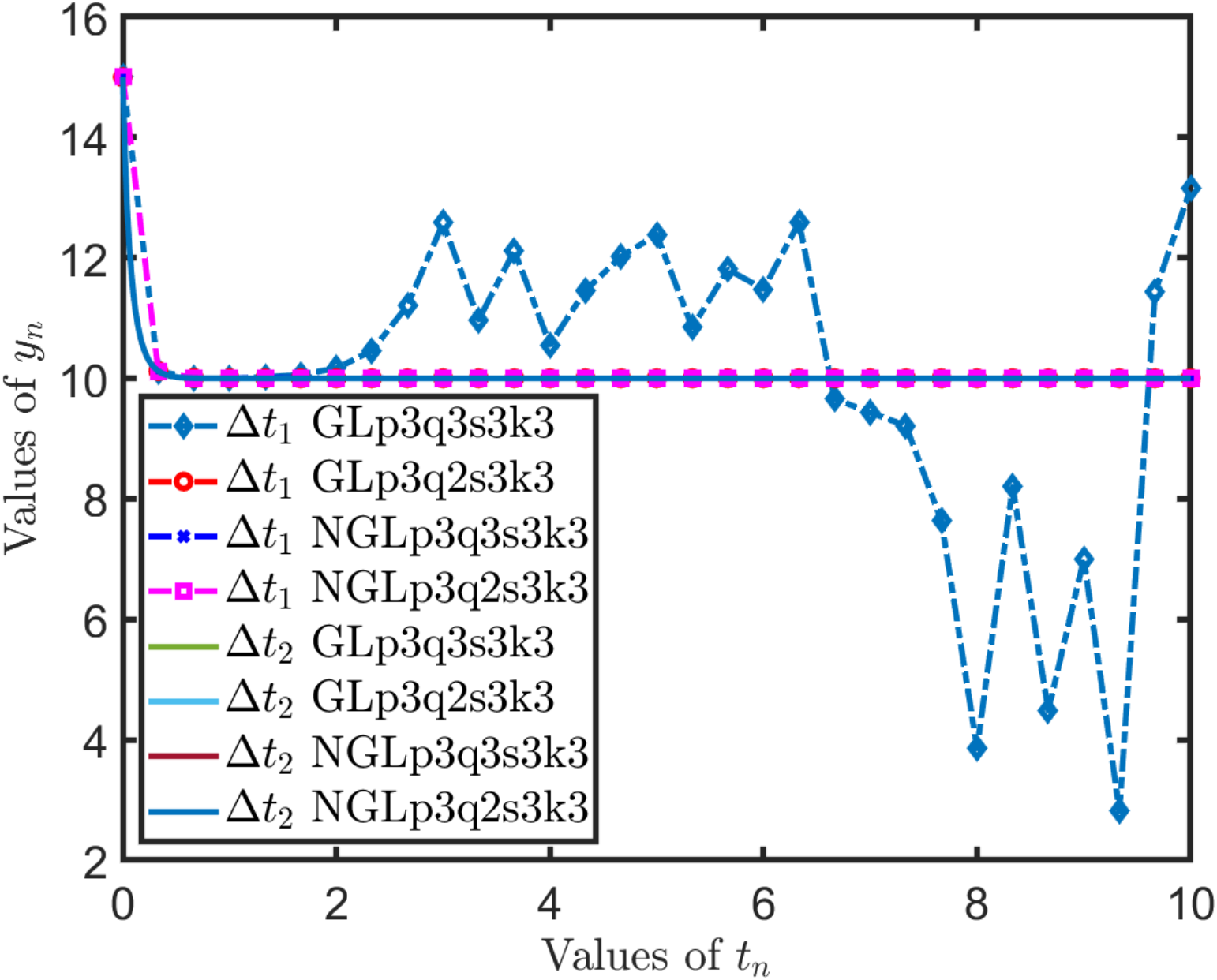}
    \includegraphics[width=0.4\linewidth]{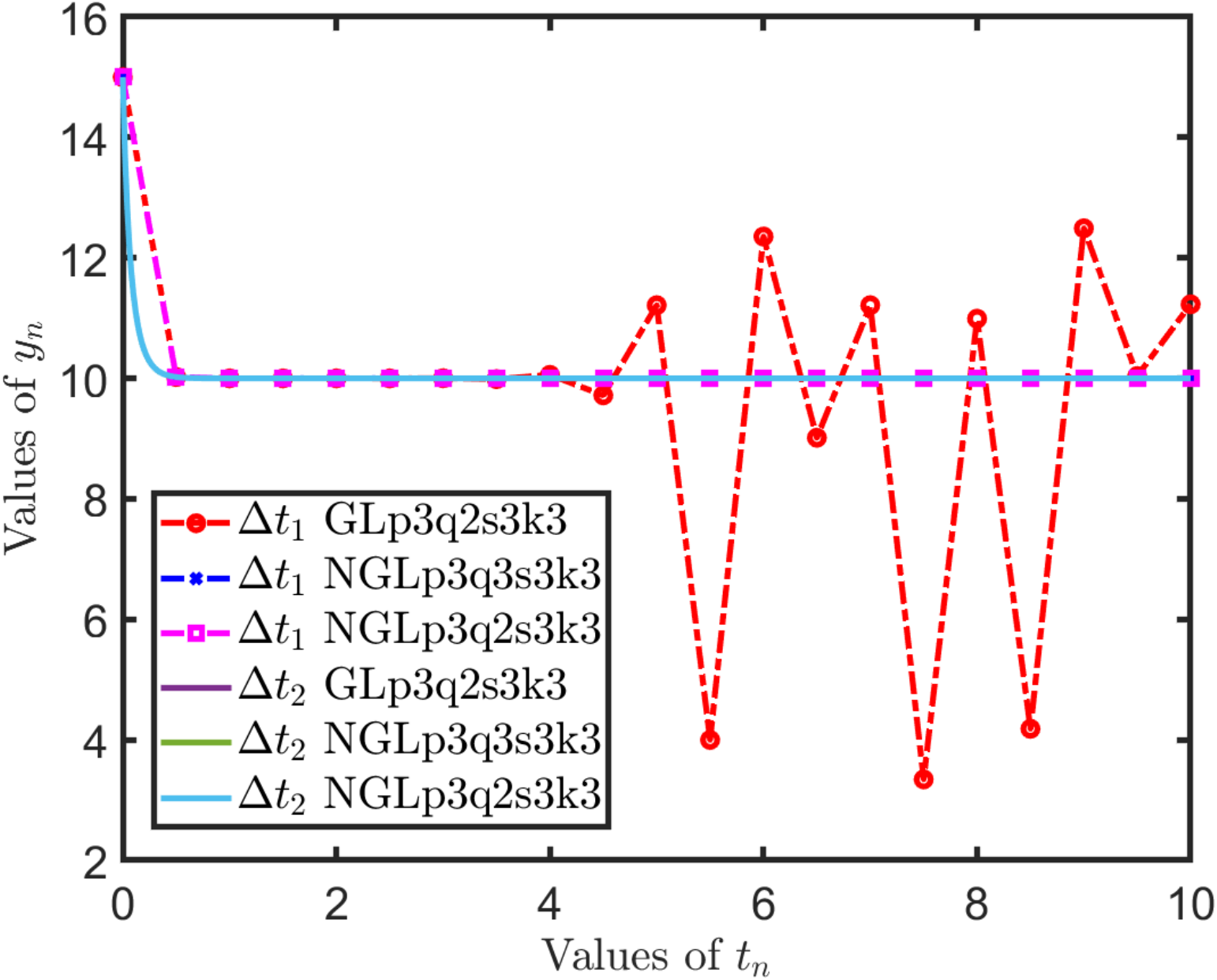}
    \caption{The plots of the different third-order methods with $\Delta t_2 = T/500$ and $\Delta t_1 = T/10$ (upper left), $\Delta t_1 = T/30$ (lower left) and $\Delta t_1 = T/15$ (lower right) for $c=10$. For $\Delta t_2$, all of the plots are very close to each other, making only one of them visible. In the upper right panel, the difference between the two third-order nonstandard multistep multistage methods from the upper left panel is plotted: since their difference is very small, they are indistinguishable on the different plots.}
    \label{fig:nonstiff_compare3}
\end{figure}

In Figure \ref{fig:nonstiff_compare4}, the results of the fourth-order methods are akin to the performance of the third-order ones: the nonstandard methods behave as expected, where the multistep multistage ones are the closest to the exact solution (although the Runge-Kutta method performs well too). Again, the difference between methods NGLp4q3s3k3 and NGLp4q4s3k3 cannot be seen on the figure, so we also plot the difference in the upper right panel. The nonstandard methods are also compared to the standard ones in the lower panel: the standard ones behave in an unstable way, but the errors in the case of GLp4q4s3k3 are smaller; the nonstandard ones remain bounded.

\begin{figure}[!htbp]
    \centering
    \includegraphics[height=0.35\linewidth]{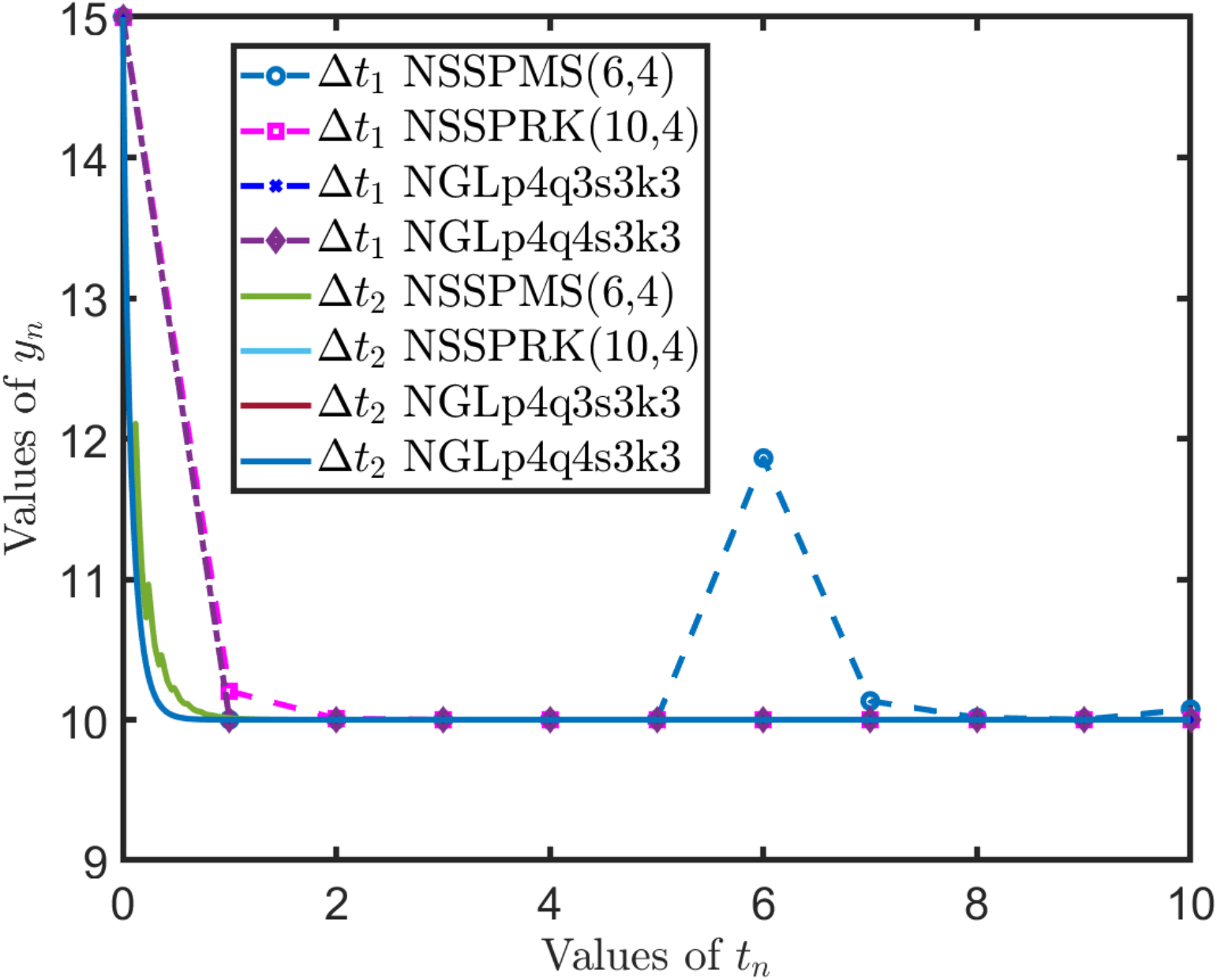}
    \includegraphics[height=0.35\linewidth]{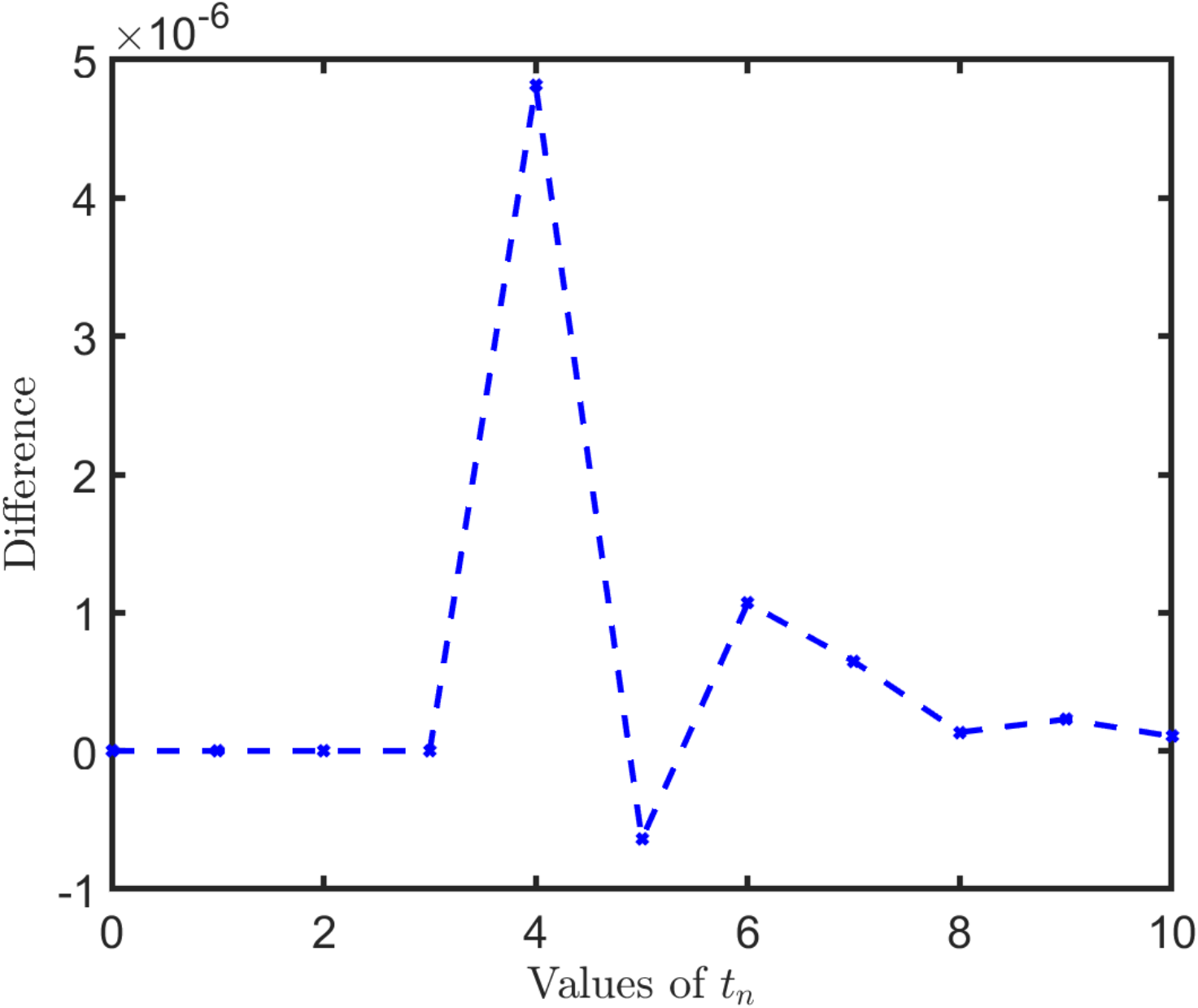}
    \includegraphics[height=0.35\linewidth]{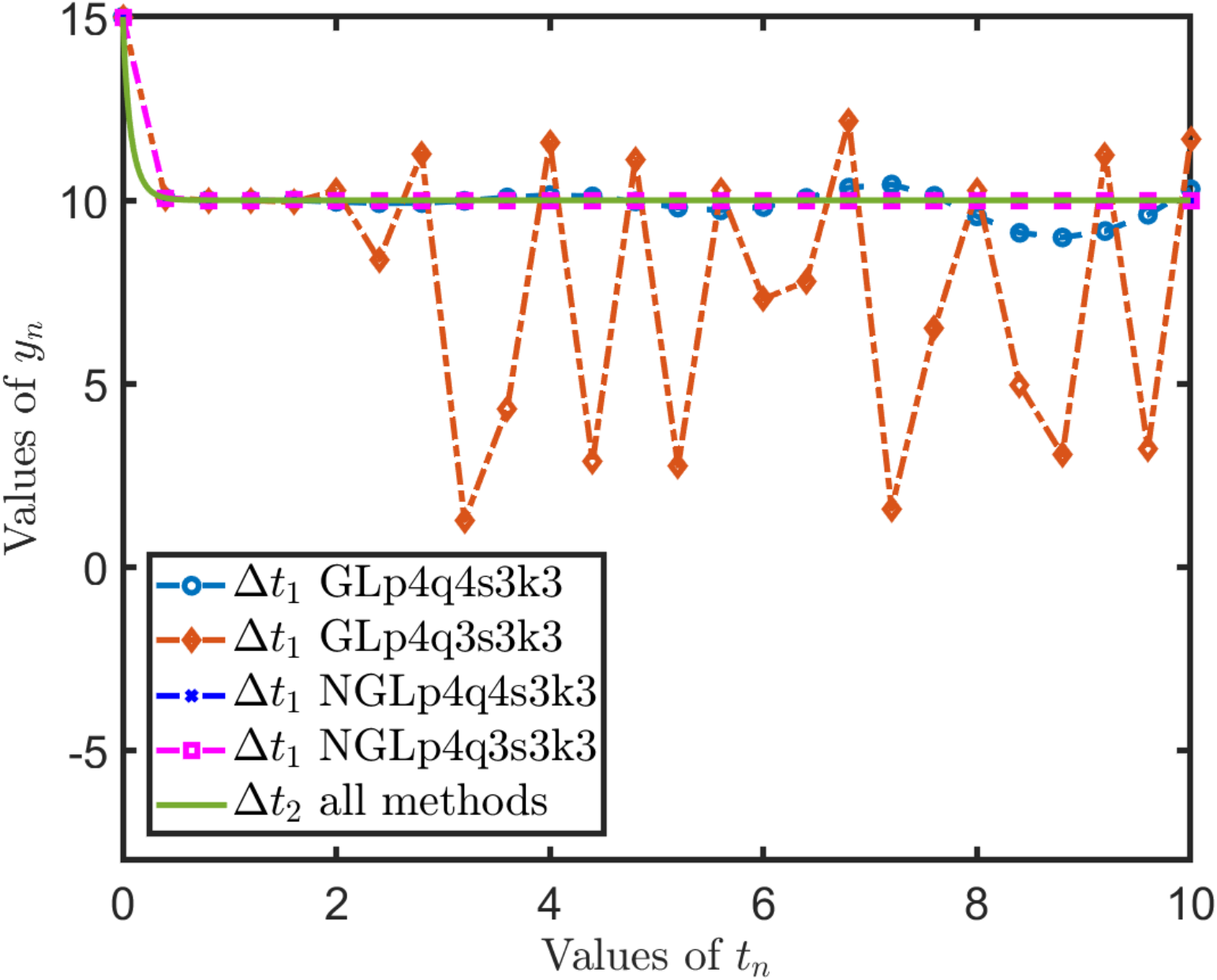}
    \caption{The plots of the different fourth-order methods with $\Delta t_2 = T/500$ and $\Delta t_1 = T/10$ (upper left), and $\Delta t_1 = T/25$ (lower). For $\Delta t_2$, all of the plots are very close to each other, making only one of them visible. In the upper right panel, the difference between the two fourth-order nonstandard multistep multistage methods from the upper left panel is plotted: since their difference is very small, they are indistinguishable on the different plots.}
    \label{fig:nonstiff_compare4}
\end{figure}

Finally, we observe the necessity of the bound given for $\varphi$ in Theorem \ref{th:preserv}. Namely, it was stated that the nonstandard multistep multistage method preserves the given property if $\varphi(x)\leq \mathcal{B}:=\mathcal{C} B_{FE}$ holds $\forall x \geq 0$. There might exist another, higher bound $\mathcal{B}^*\geq \mathcal{B}$ such that the property is preserved even when only the weaker assumption $\varphi(x)\leq \mathcal{B}^*$ holds. To test this behavior, we change the bound for $\varphi$ to an arbitrary positive constant, and run the methods with $1000$ values of $\tilde{x}$ from the interval $[10^{-3}, \; 5]$ and consider $1000$ possible values of timesteps in the interval $[0.5,\; 3]$. We say that the method behaves as expected for a fixed value of $\tilde{x}$ and with this given bound for $\varphi$, if the given property is preserved for every possible timestep on the interval $t \in [0,10]$. By a bisection method, we determine the value $\mathcal{B}^*$ where the the boundedness property is lost, i.e., the scheme does not preserve boundedness if $\varphi(x)>\mathcal{B}^*$ is allowed. The values of $\mathcal{B}^*$ for different values of $\tilde{x}$ can be seen in Figure \ref{fig:nonstiff_bound}. The bound seems to be strict for most methods in the case $0 < \tilde{x} < 10$, while still being relatively close to the real bound when $\tilde{x}>10$. The plots of Figure \ref{fig:nonstiff_bound} might indicate that a bigger bound $\mathcal{B}$, namely, $\mathcal{B}:=\mathcal{C} \dfrac{1}{c}$ seems to be sufficient too for these methods. However, in general this is not the case: it turned out in \cite{takacs26} that for some nonstandard multistep methods (more precisely, methods NSSPMS(4,2), NSSPMS(4,3) and NSSPMS(6,4)) this bound is not sufficient, and only $\mathcal{B}:=\mathcal{C} \min\left\{\frac{1}{c}, \frac{1}{\tilde{x}}\right\}$ would suffice. This is also apparent in the stiff case; see the next section for details.

\begin{figure}[!htbp]
    \centering
    \includegraphics[height=0.35\linewidth]{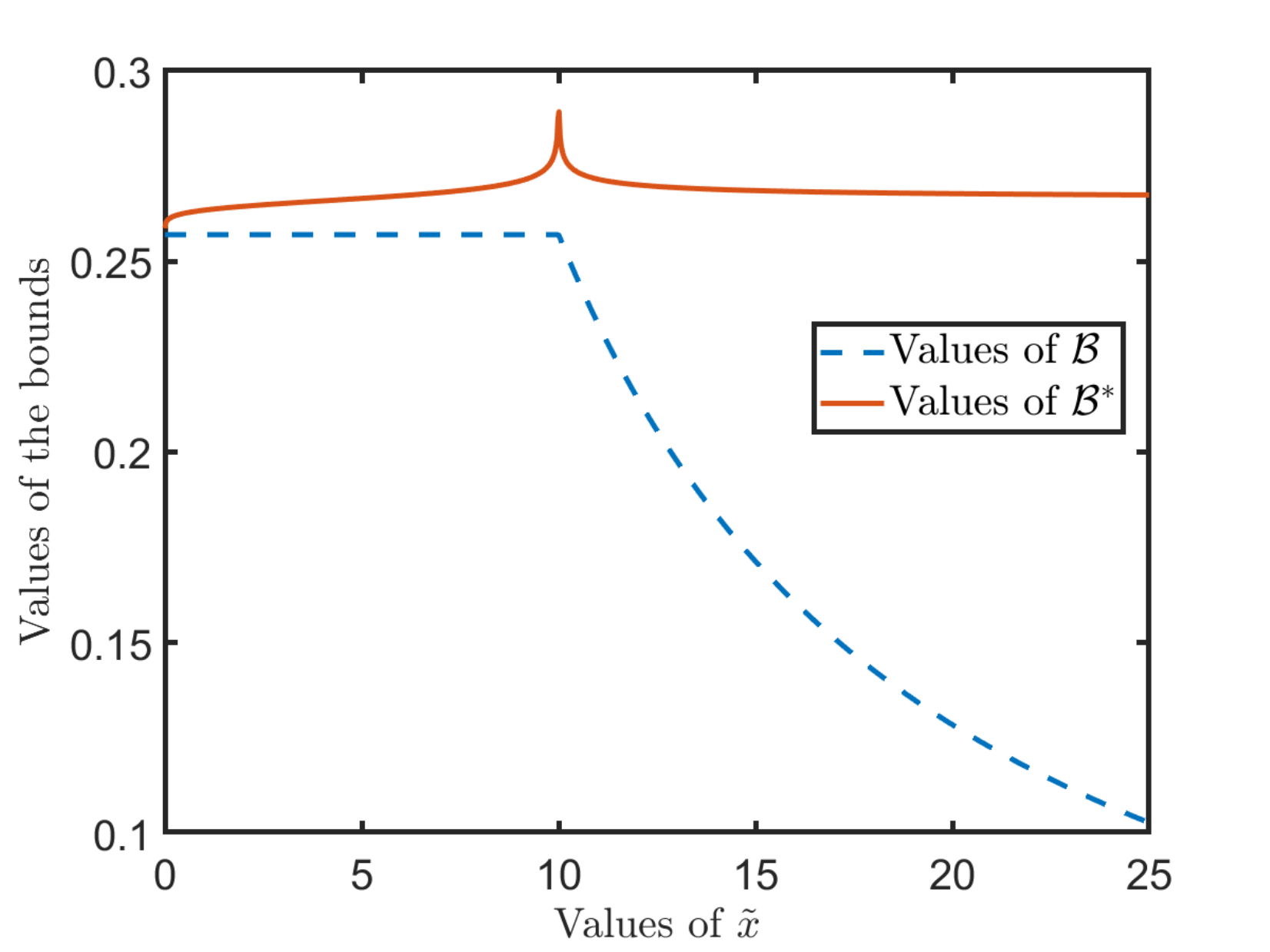}
    \includegraphics[height=0.35\linewidth]{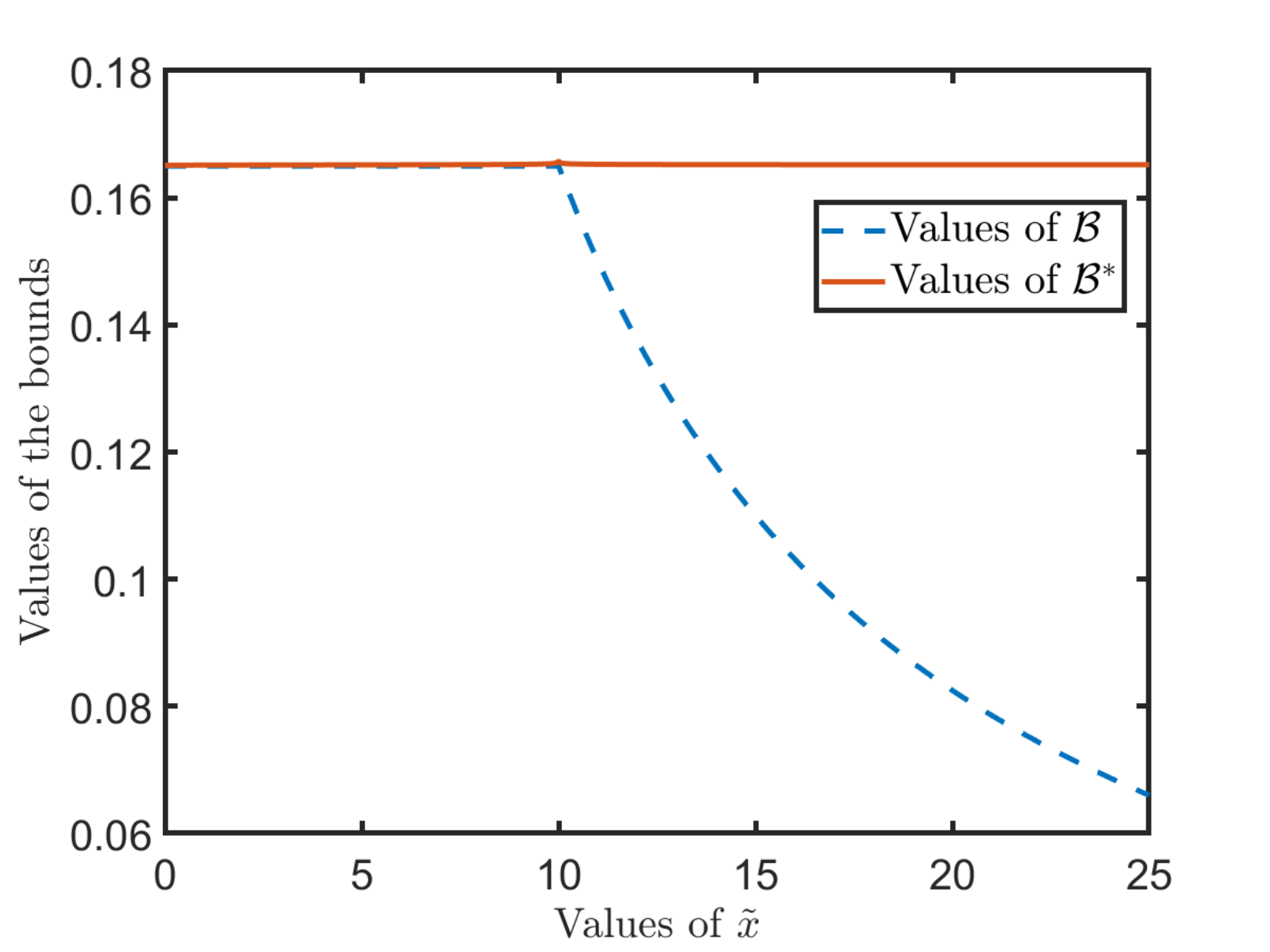}
    \includegraphics[height=0.35\linewidth]{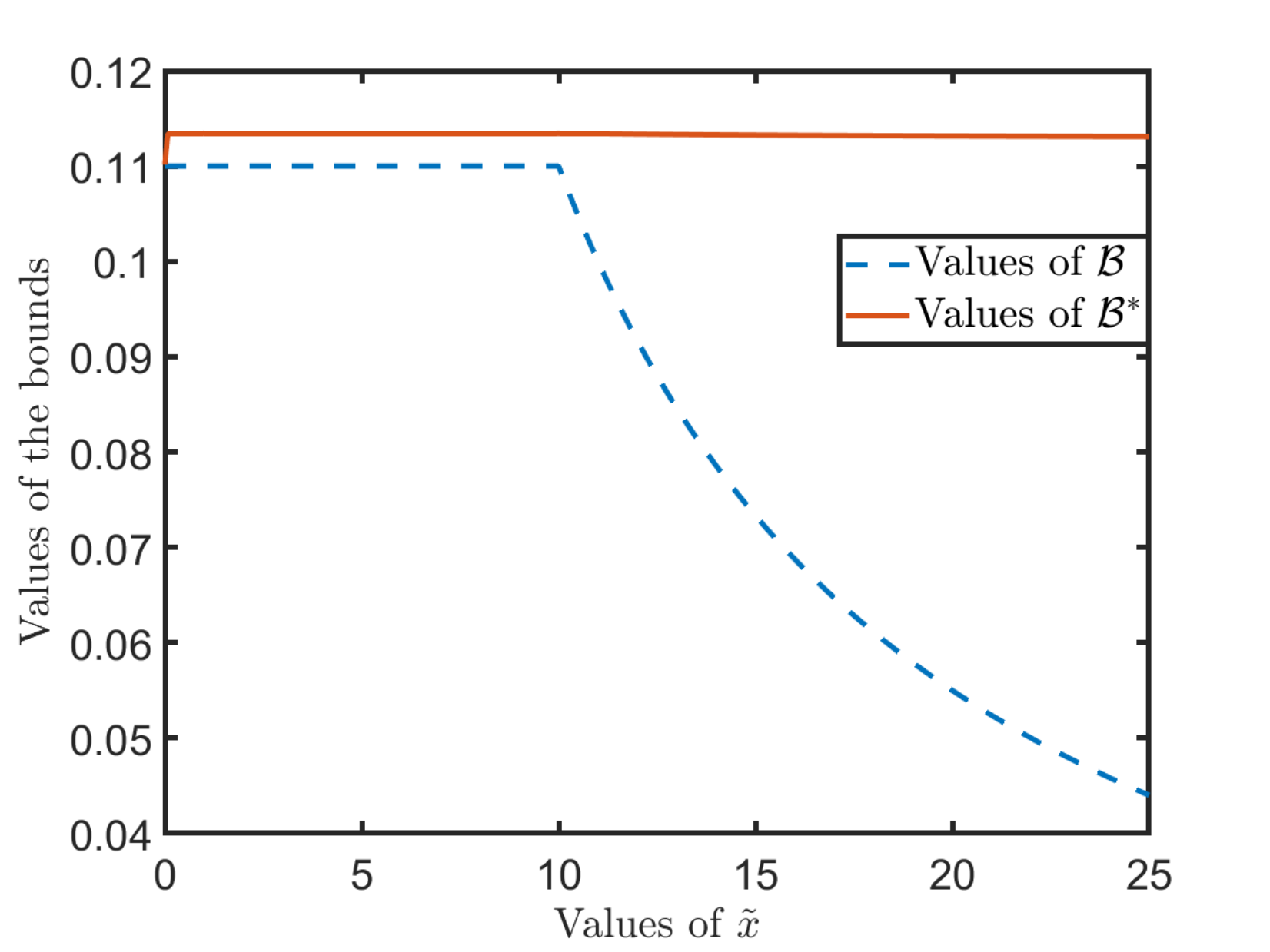}
    \includegraphics[height=0.35\linewidth]{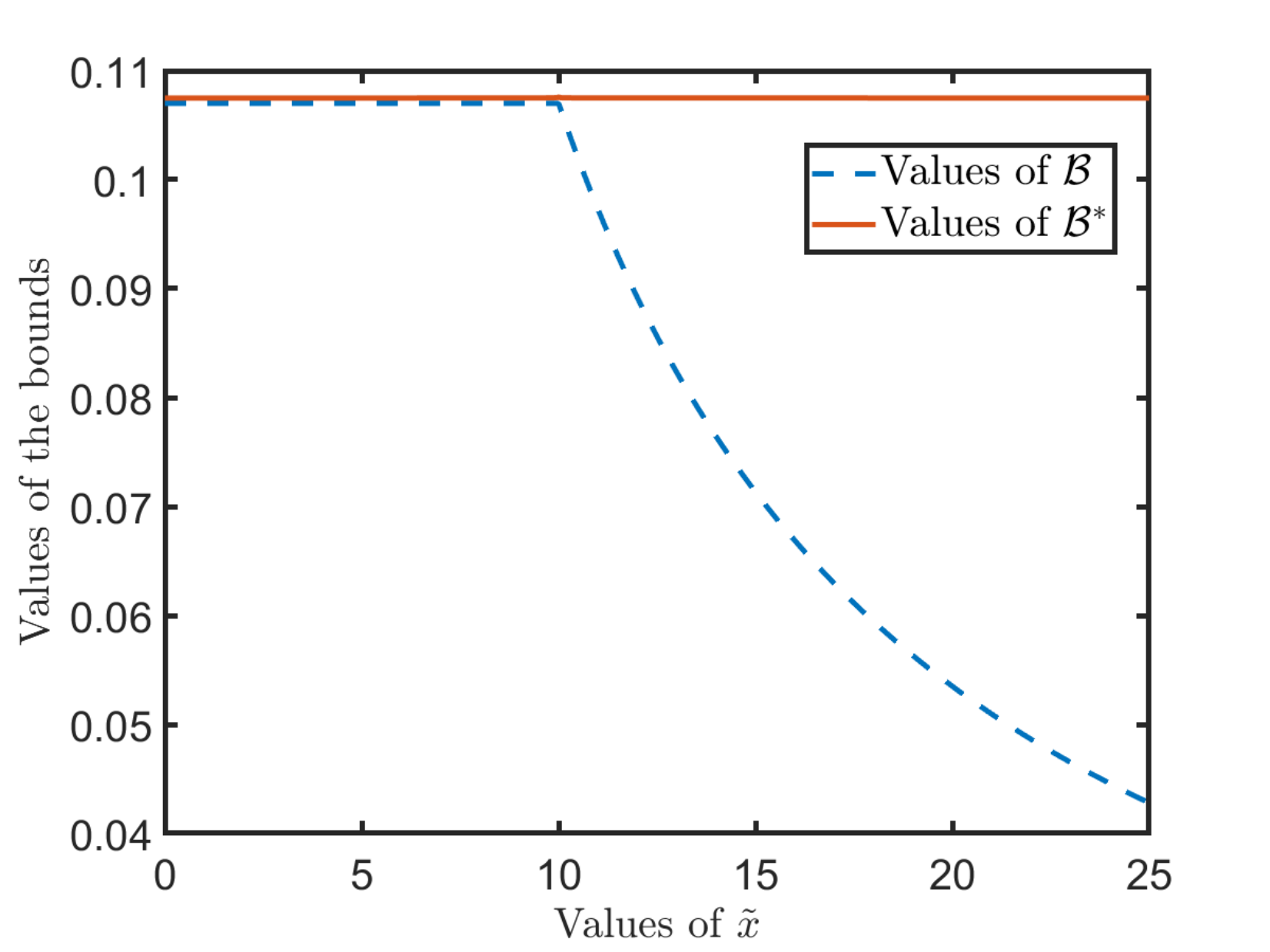}
    \includegraphics[height=0.35\linewidth]{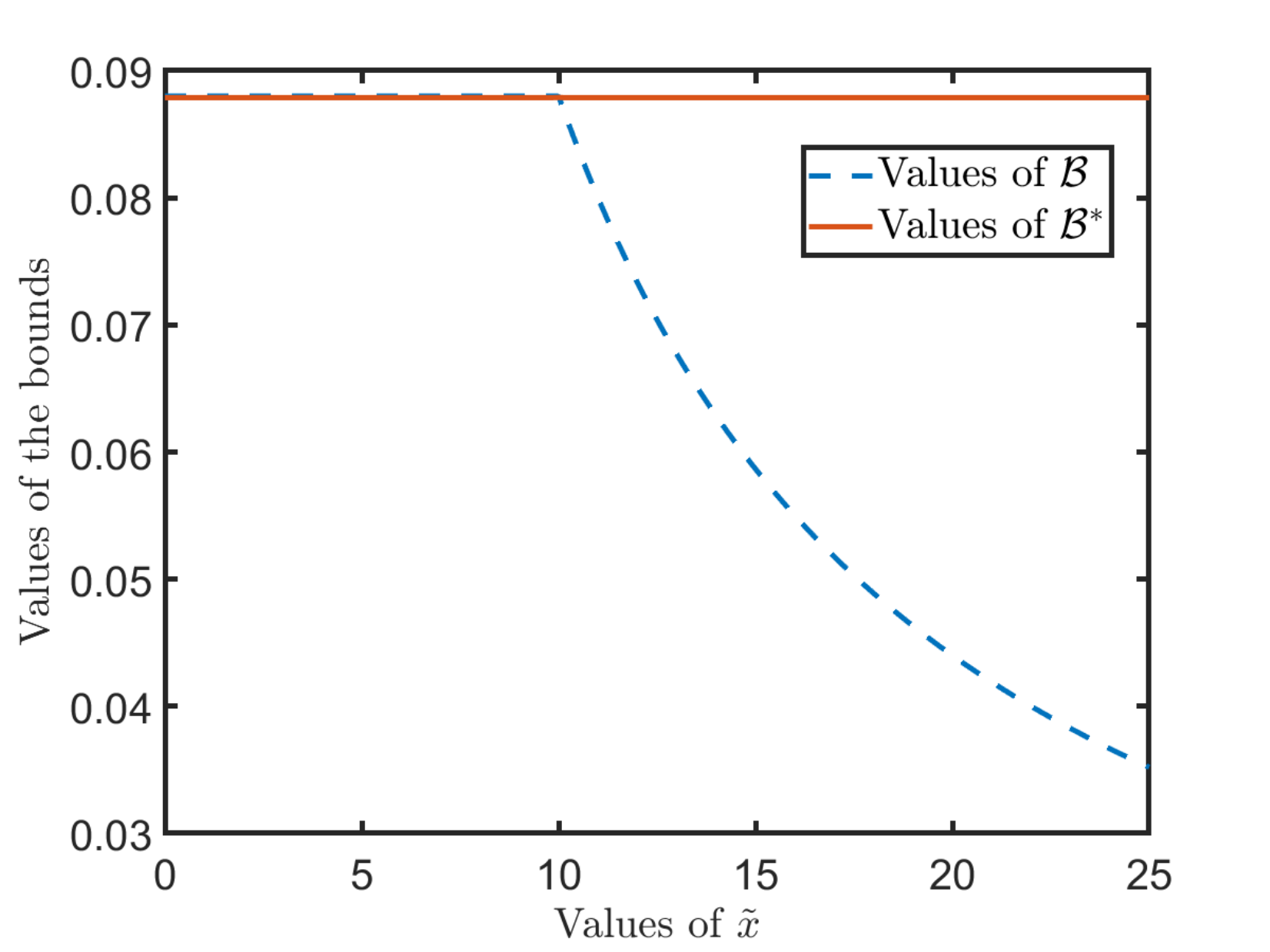}
    \caption{The sufficient bound $\mathcal{B}:=\mathcal{C}B_{FE}$ and the ``real'' bounds $\mathcal{B}^{*}$ for methods NGLp2q2s3k3 (top left), NGLp3q2s3k3 (top right), NGLp3q3s3k2 (middle left), NGLp4q3s3k3 (middle right), and NGLp4q4s3k3 (bottom) for the boundedness property with $c=10$.}
    \label{fig:nonstiff_bound}
\end{figure}

\subsubsection{The stiff case $c=500$}
Let us now consider the stiff case $c=500$. We run the same test about the orders of the methods as in the previous section, but now with $T=1/500$, initial condition $\tilde{x}=250$ and stepsizes $\Delta t =2\cdot 10^{-4}/2^k, \; k=0, 1,  \dots, 9$. (The reason for the seemingly small value of $T$ is that for bigger values of $t$, the change of the solution is smaller than machine precision.) In the right panels of Figure \ref{fig:nonstiff_orders}, the errors of the different methods can be seen, while the exact values and the errors (calculated from the slopes) can be found in Table \ref{tab:stiff_errors} in Appendix B. The orders of the $c=10$ and $c=500$ cases are very similar, the only difference being that the stiff case results in larger errors.

We also compare the way the nonstandard multistep multistage methods preserve the qualitative properties compared to the other nonstandard methods. In Figure \ref{fig:stiff_compare}, the nonstandard methods are plotted with $c=500$, $T=10$, $\tilde{x}=750$, $\Delta t_1=T/10$ and $\Delta t_2 = T/500$. As we can see, the ones closest to the exact solution are the multistep multistage ones: the Runge-Kutta ones are relatively far from the exact one for smaller values of $t_n$, while the multistep methods attain some jumps for larger values of $t_n$. Note that the standard versions of these methods weren't plotted because they are highly unstable for such large timesteps. We also do not plot the differences between the nonstandard multistep multistage methods, since they are smaller than machine precision (which is a result of the property of the continuous solution).

\begin{figure}[!htbp]
    \centering
    \includegraphics[width=0.45\linewidth]{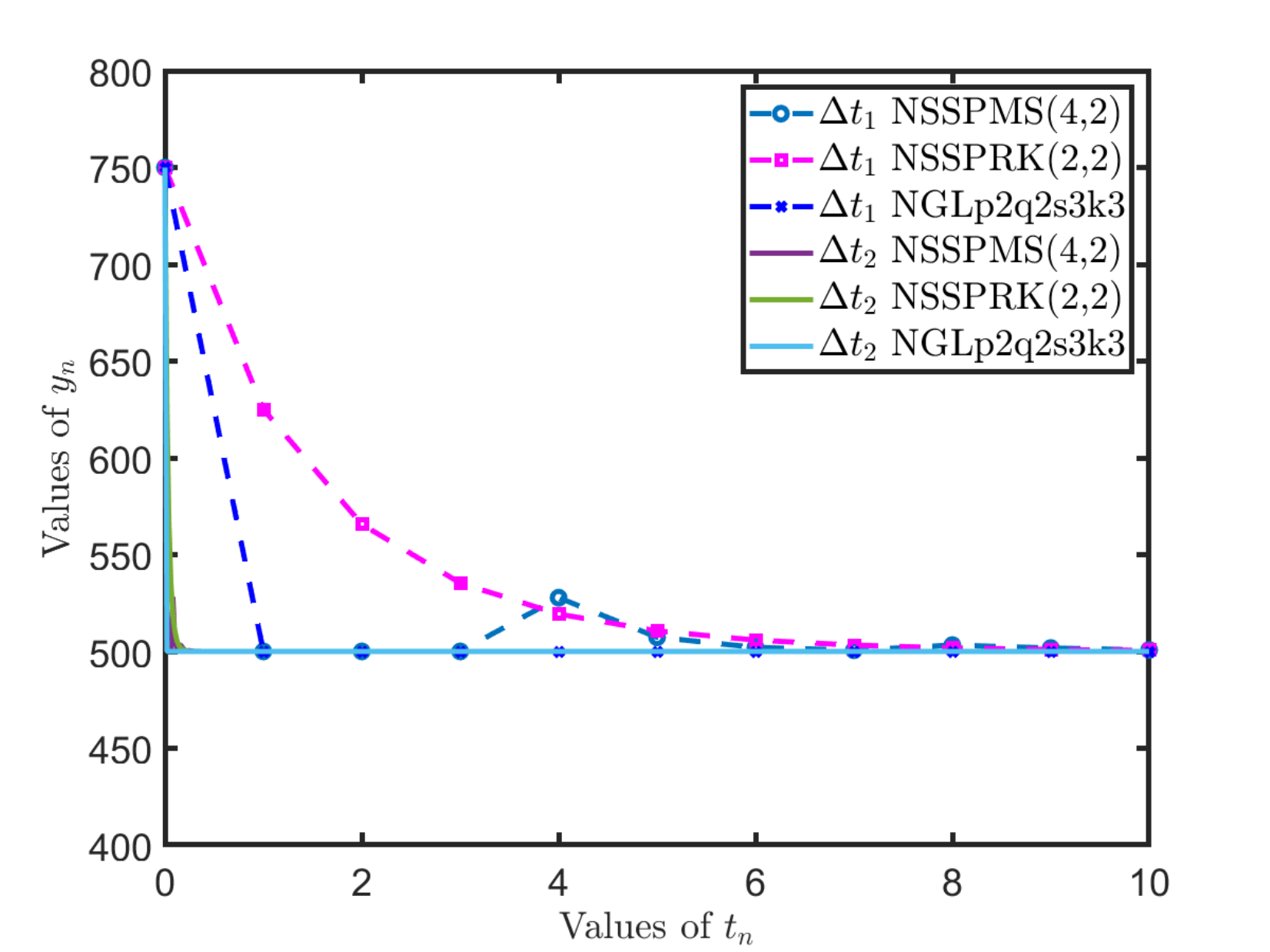}
    \includegraphics[width=0.45\linewidth]{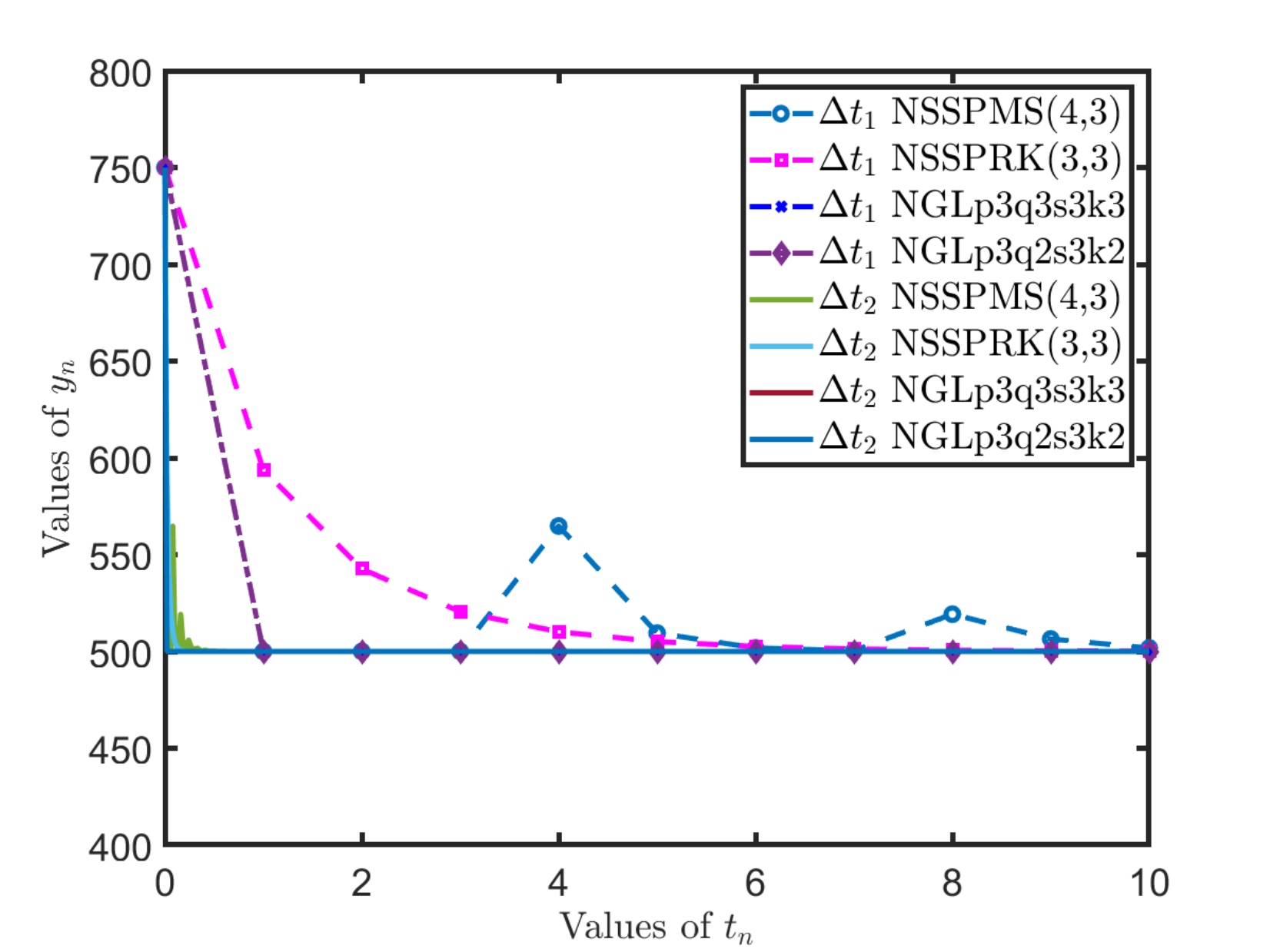}
    \includegraphics[width=0.45\linewidth]{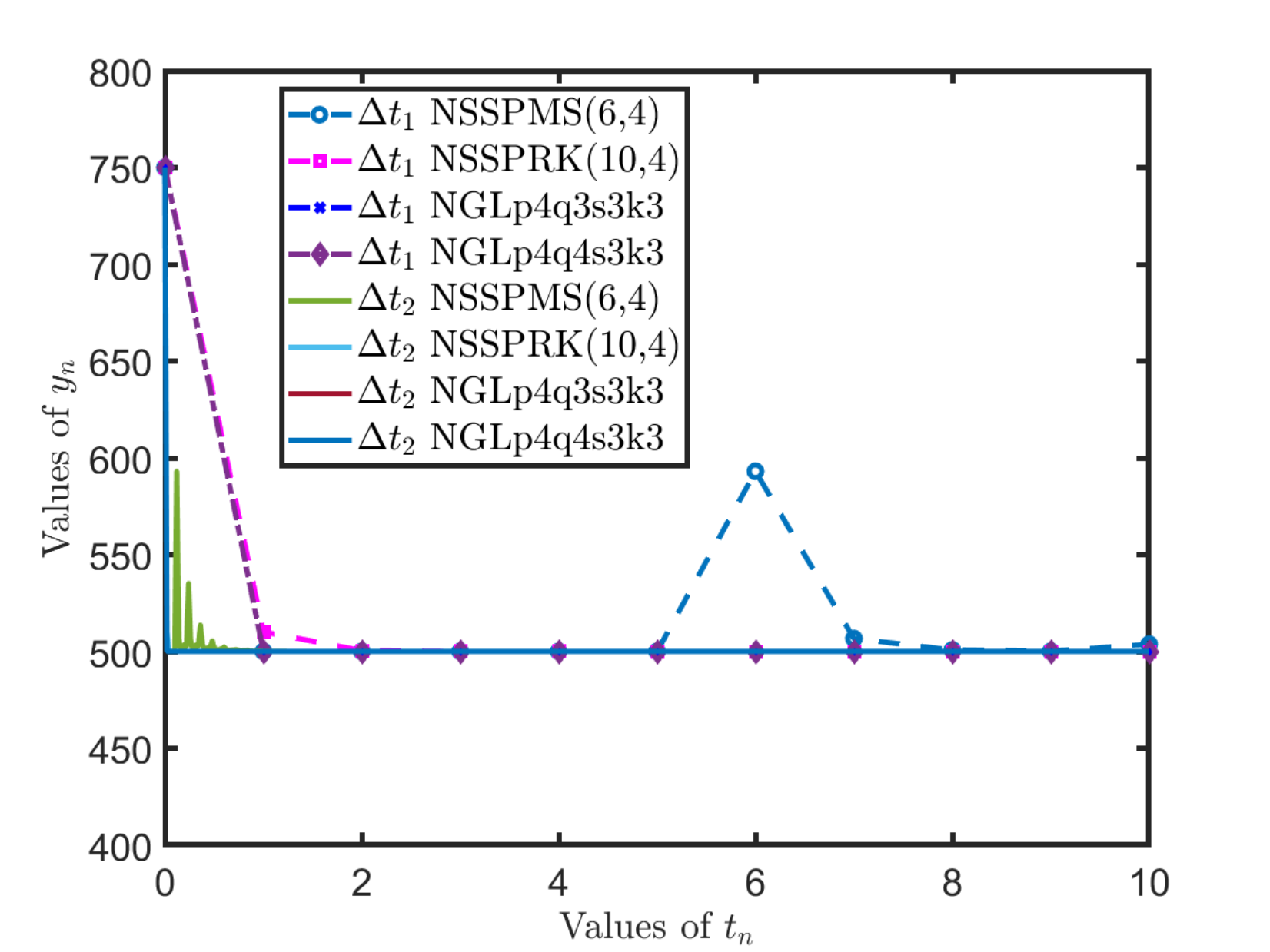}
    \caption{The plots of the different second-order (upper left), third-order (upper right), and fourth-order (bottom) methods with $\Delta t_2 = T/500$ and $\Delta t_1 = T/10$. For $\Delta t_2$, all of the plots are very close to each other, making only one of them visible. Also, the nonstandard multistep multistage methods are very close to each other too.}
    \label{fig:stiff_compare}
\end{figure}

Lastly, we test the necessity of bound $\mathcal{C} B_{FE}$ in the stiff case too: for this, 1000 values of $\tilde{x}$ from the interval $[10^{-3}, 1250]$ were used, and 1000 different timesteps from $[0.002, 0.012]$ in the case of methods NGLp2q2s3k3 and NGLp3q3s3k3, while 100-many values for methods NGLp3q2s3k2, NGLp4q3s3k3 and NGLp4q4s3k3 (the reason for the change was the high runtimes). As we can see in Figure \ref{fig:stiff_bound}, the sufficient bound is relatively close to the 'real', necessary one in most cases. Also note that the bound in the case of method NGLp3q3s3k2 highlights the reason the choice $\dfrac{1}{c}$ would not be sufficient, since the bound goes below these values for values of $\tilde{x}$ close to 1200. Similar behavior can also be seen in the case of nonstandard linear multistep methods, see \cite{takacs26}.

\begin{figure}[!htbp]
    \centering
    \includegraphics[height=0.35\linewidth]{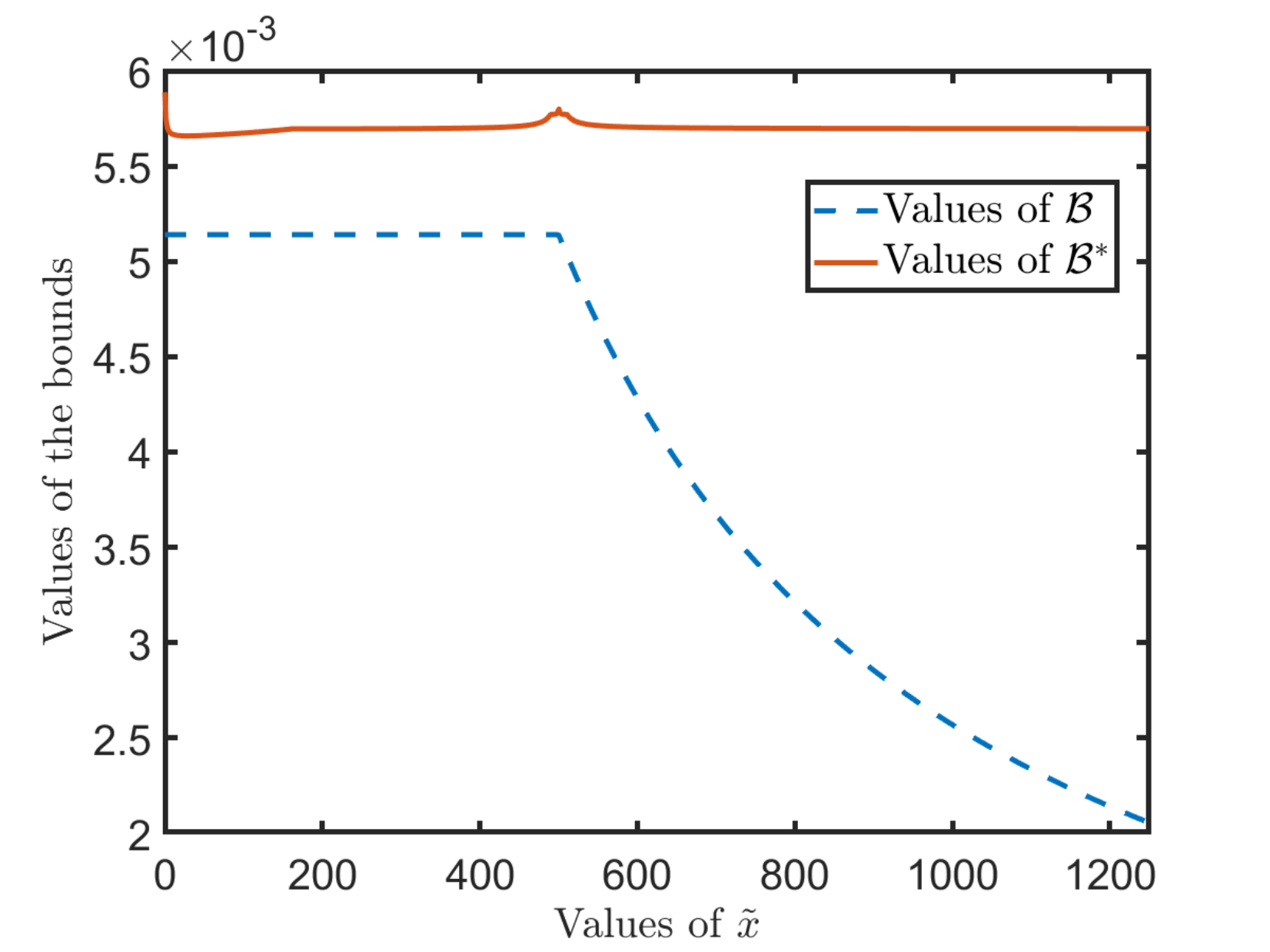}
    \includegraphics[height=0.35\linewidth]{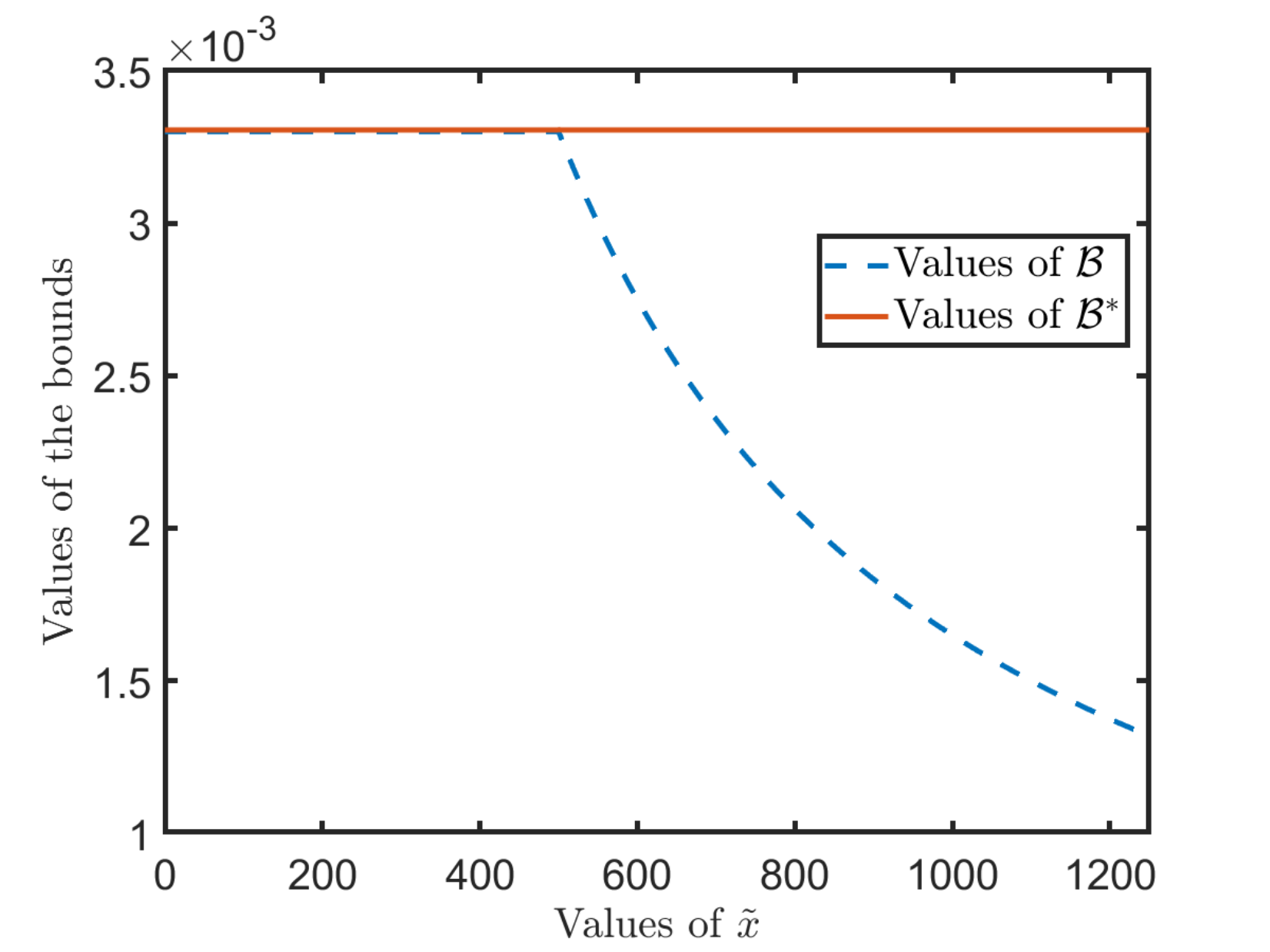}
    \includegraphics[height=0.35\linewidth]{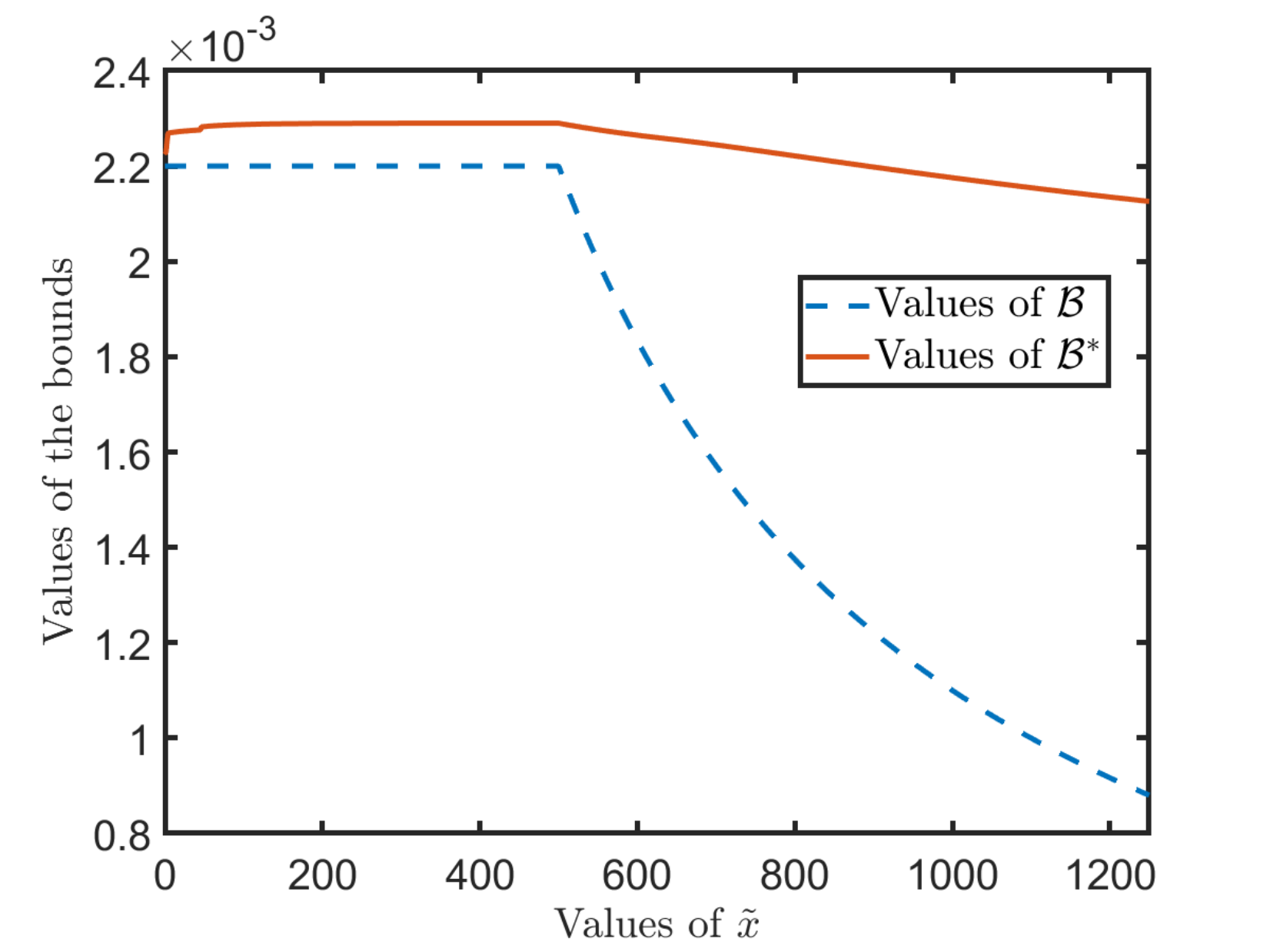}
    \includegraphics[height=0.35\linewidth]{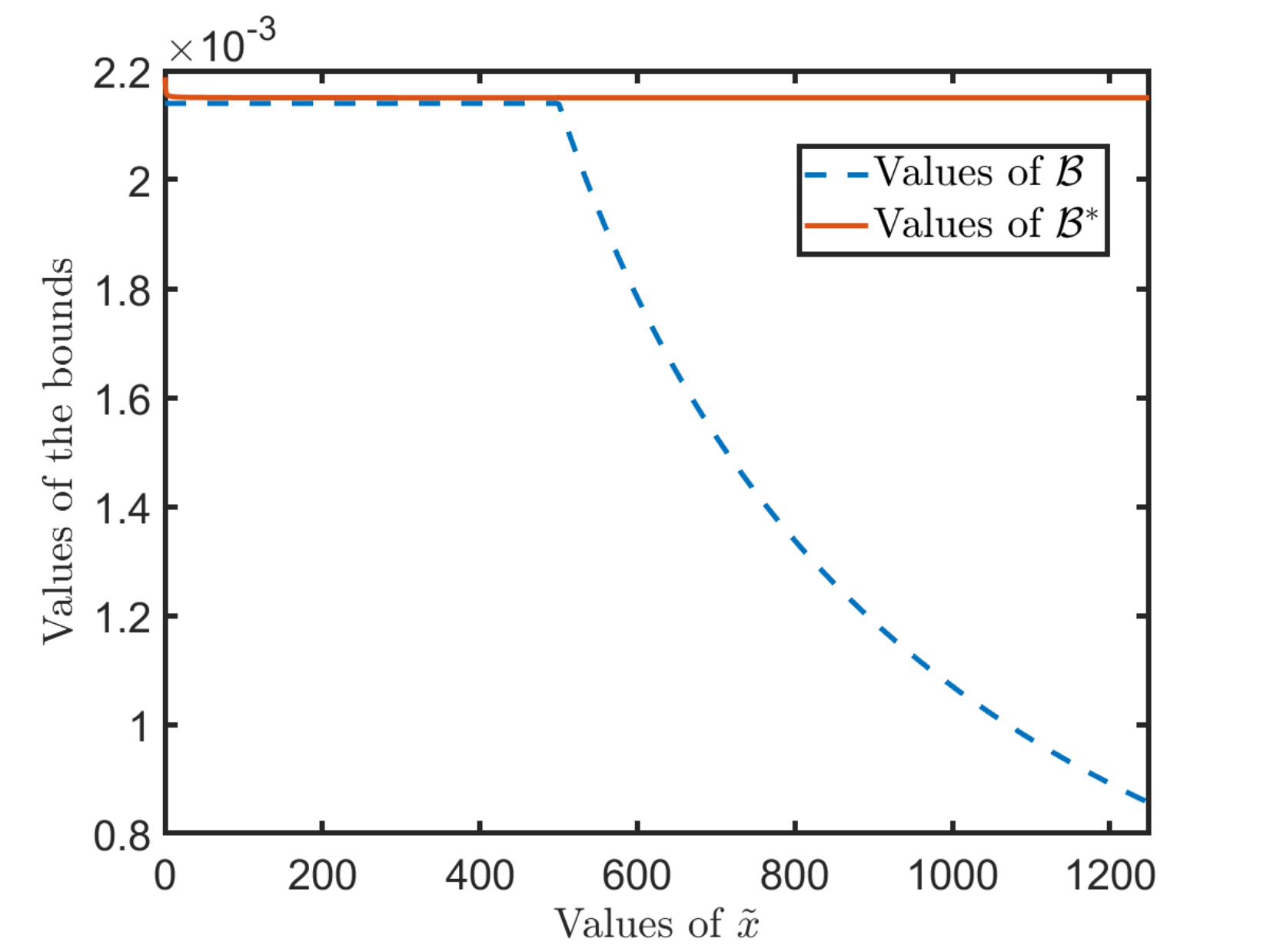}
    \includegraphics[height=0.35\linewidth]{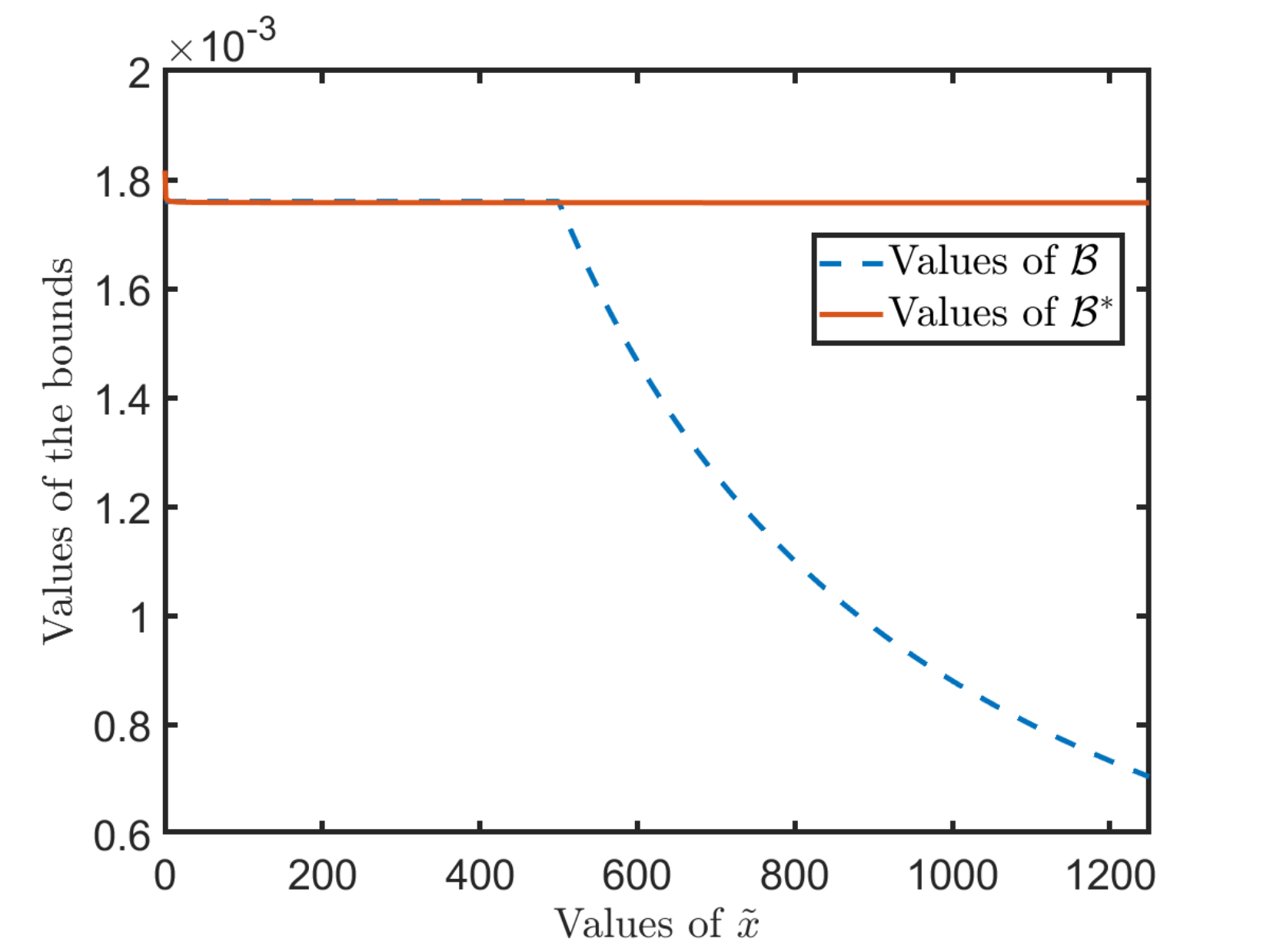}
    \caption{The sufficient bound $\mathcal{B}:=\mathcal{C}B_{FE}$ and the ``real'' bounds $\mathcal{B}^{*}$ for methods NGLp2q2s3k3 (top left), NGLp3q2s3k3 (top right), NGLp3q3s3k2 (middle left), NGLp4q3s3k3 (middle right), and NGLp4q4s3k3 (bottom) for the boundedness property with $c=500$.}
    \label{fig:stiff_bound}
\end{figure}

\subsection{An SEIR system}\label{sec:SEIR}
The following form of the Kermack-McKendrick model \cite{kermack} describes the propagation of an illness among a given population with a constant influx of healthy people $\Pi \in \mathbb{R}^+$ and a constant $\gamma \in \mathbb{R}^+$ describing the infectiousness of the process:
\begin{equation}\label{eq:sir}
\begin{aligned}
    S'(t) &= \Pi-\gamma S(t) I(t),\\
    E'(t) &= \gamma S(t) I(t) - E(t),\\
    I'(t) &= E(t) - I(t),\\
    R'(t) &= I(t),
\end{aligned} 
\end{equation}
where $S(t), E(t), I(t)$ and $R(t)$ describe the number of susceptible, exposed (infected but not yet infectious), ill (infected and infectious), and recovered people, respectively. It is easy to see that if $S(t_0), E(t_0), I(t_0), R(t_0)\geq 0$, then $S(t), E(t), I(t), R(t)\geq 0$ holds for every $t \in [t_0, T]$. Moreover, $S(t)+E(t)+I(t)+R(t) = \rev{ \Pi t + \mathcal{M}}$ is also true with constant $S(t_0) + E(t_0) + I(t_0) + R(t_0) = \mathcal{M} \in \mathbb{R}^+$ for every $t\in [t_0, T]$ (an interested reader might consult \cite{capasso}). From now on, we use $t_0=0$.

The value of $B_{FE}$ is given by the next proposition.

\begin{prop}(\cite[Prop. 2.]{takacs26})\label{prop:seir_bound}
    Let us apply the forward Euler method to equation \eqref{eq:sir} with stepsize $\Delta t_{FE}\leq \min\left\{ \rev{\dfrac{1}{\gamma \mathcal{M}}}, 1 \right\}$. Then, the forward Euler method applied to equation \eqref{eq:sir} preserves the non-negative property of the numerical solution, and if $\Pi = 0$, then $S^n+E^n+I^n+R^n=\mathcal{M}$ also holds.
\end{prop}
The proof of a more general case can be found in \cite{takacs24}. Consequently, the choice $B_{FE}=\min\left\{ \rev{\dfrac{1}{\gamma \mathcal{M}}}, 1 \right\}$ is appropriate.

In the case of the logistic equation \eqref{eq:logistic}, since the exact solution was known, we could use those values as the starting procedure of our method. However, in the case of system \eqref{eq:sir}, we cannot take this approach. Because of this, two different starting procedures are applied:
\begin{itemize}
    \item[$\psi_1$:] Let us approximate the inner stage values in the interval $[t_{n-1}, t_{n}]$ ($n=1, \dots, k$) by $n$-many steps of a nonstandard Runge-Kutta method of the same order as the considered multistep multistage method.
    \item[$\psi_2$:] Let us approximate the inner stage values in the interval $[t_{n-1}, t_{n}]$ ($n=1, \dots, k$) by $10 n$-many steps of a nonstandard Runge-Kutta method of the same order as the considered multistep multistage method.
\end{itemize}
Since the nonstandard Runge-Kutta methods are special cases of the previously discussed nonstandard multistep multistage methods, the previous theorems about their orders and the preservation of qualitative properties also hold.

In the next two sections, we analyze the cases $\Pi=0$ and $\Pi>0$ separately.

\subsubsection{The case $\Pi=0$}
First, let us assume that $\Pi=0$ holds. To observe the orders of the methods, we calculate a reference solution using a fourth-order (standard) Runge-Kutta method with a very small timestep $\Delta t=0.1 \cdot 2^{-9} \cdot 10^{-3}$. Then, we plot the maximum norm of the difference between the reference solution and a given nonstandard method calculated with timesteps $\Delta t=\rev{0.05/2^k, \; k=0, 1, \dots, 9}$ at $T=1$ with initial condition $(S_0, E_0, I_0, R_0)=(0.8, 0, 0.2, 0)$, $\gamma=5$, and starting process $\psi_1$. In Figure \ref{fig:seir_order}, it is evident that for second- and third-order methods, the multistep multistage methods produce the smallest errors, while in the fourth-order case they are worse than the Runge-Kutta method, but better than the multistep ones. The increase of the order in the case of the third-order methods is also apparent, justifying the use of $\varphi_4$ (see Remark \ref{rem:phi}). The corresponding values and errors can be found in Table \ref{tab:seir_errors} in Appendix B.

\begin{figure}[!htbp]
    \centering
    \includegraphics[width=0.45\linewidth]{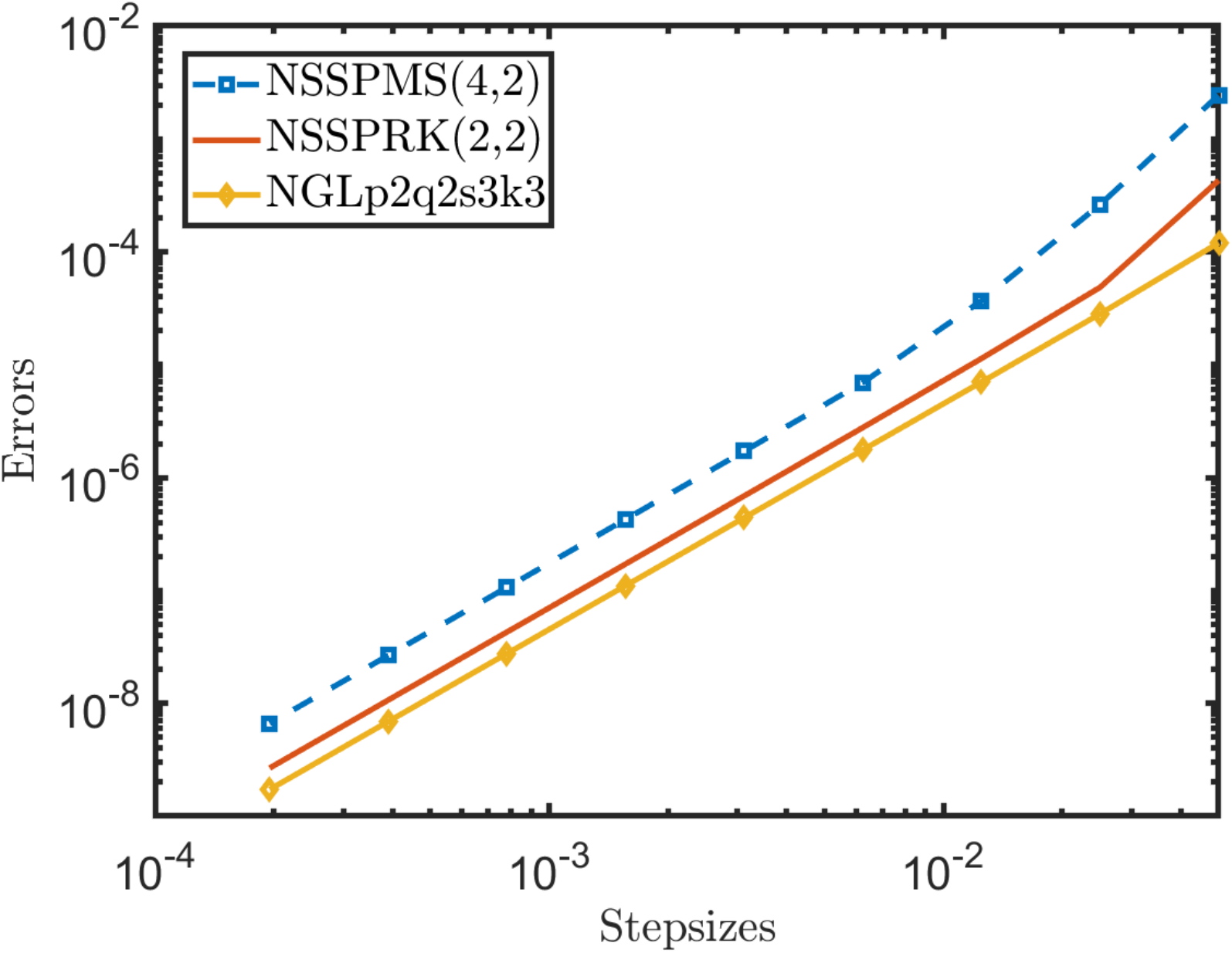}
    \includegraphics[width=0.45\linewidth]{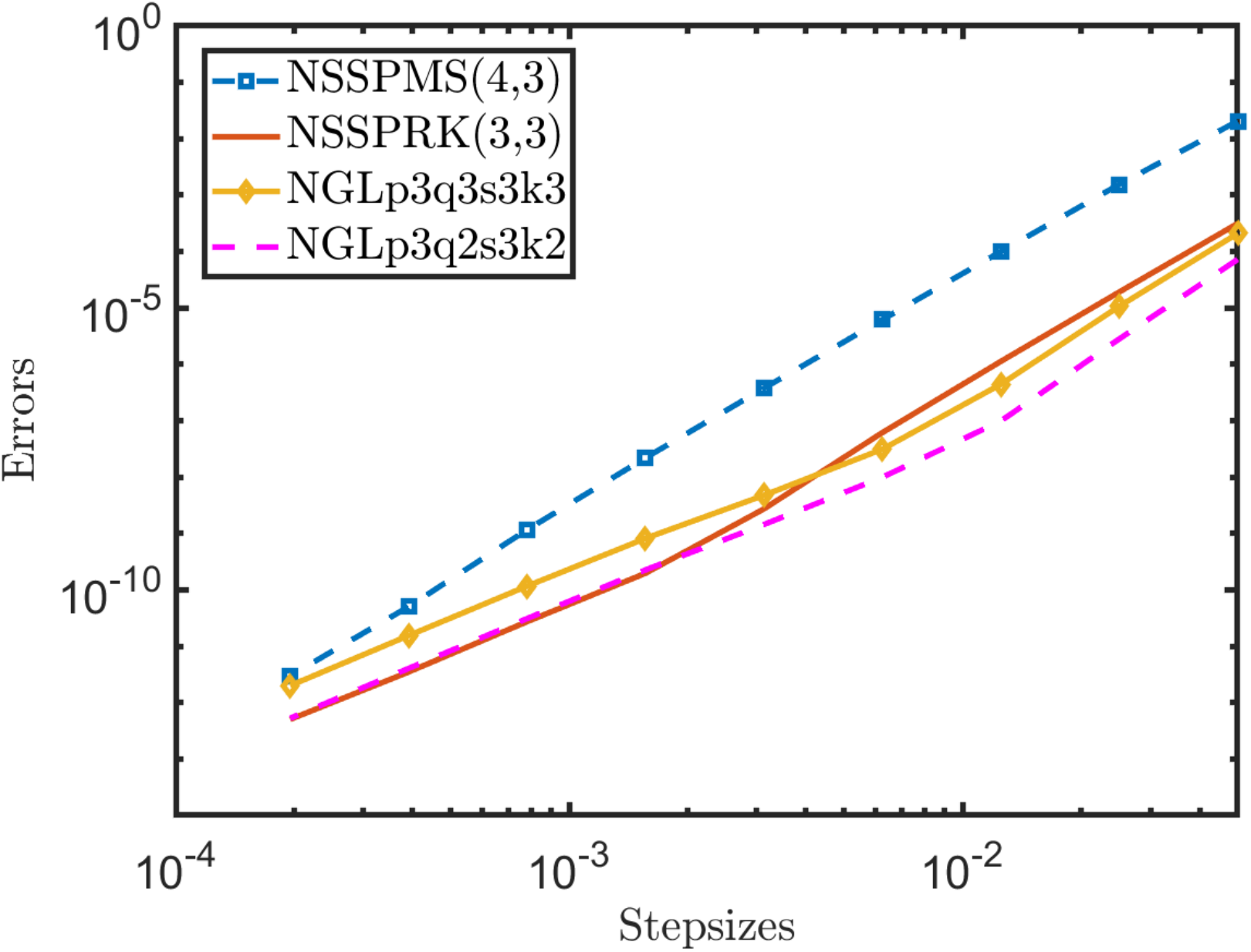}
    \includegraphics[width=0.45\linewidth]{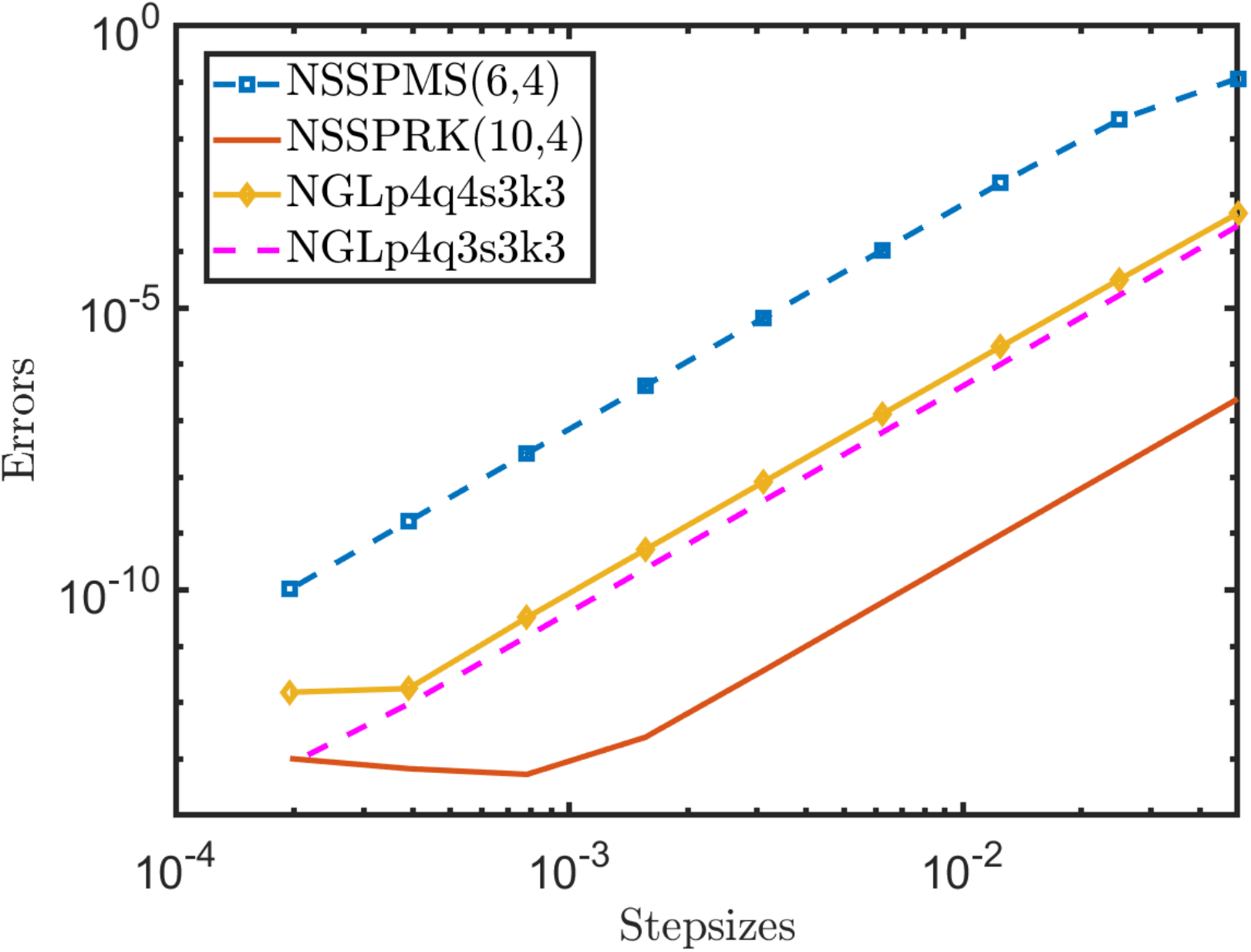}
    \caption{The errors of the different nonstandard methods for different values of $\Delta t$ applied to the SEIR model \eqref{eq:sir} with $\Pi=0$. The methods are grouped according to their expected orders.}
    \label{fig:seir_order}
\end{figure}

We also test the way the nonstandard multistep multistage methods preserve the properties of the continuous models. Second-order methods are run with $T=80$, initial condition $(S_0, E_0, I_0, R_0)=(0.95, 0, 0.05, 0)$, $\gamma=1.5$ and timesteps $\Delta t_1 = T/26$ and $\Delta t_2 = T/4000$. We only plot the curves of $I^n$, while the others attain similar behavior. As we can see in Figure \ref{fig:seir_compare2}, the standard method becomes unstable as $n$ increases, while the nonstandard multistep multistage methods behave as expected, although they produce large errors when $\psi_1$ is chosen, and perform much better when $\psi_2$ is applied. Both the Runge-Kutta and multistep methods are far from the exact solution, although they remain bounded.

\begin{figure}[!htbp]
    \centering
    \includegraphics[width=0.45\linewidth]{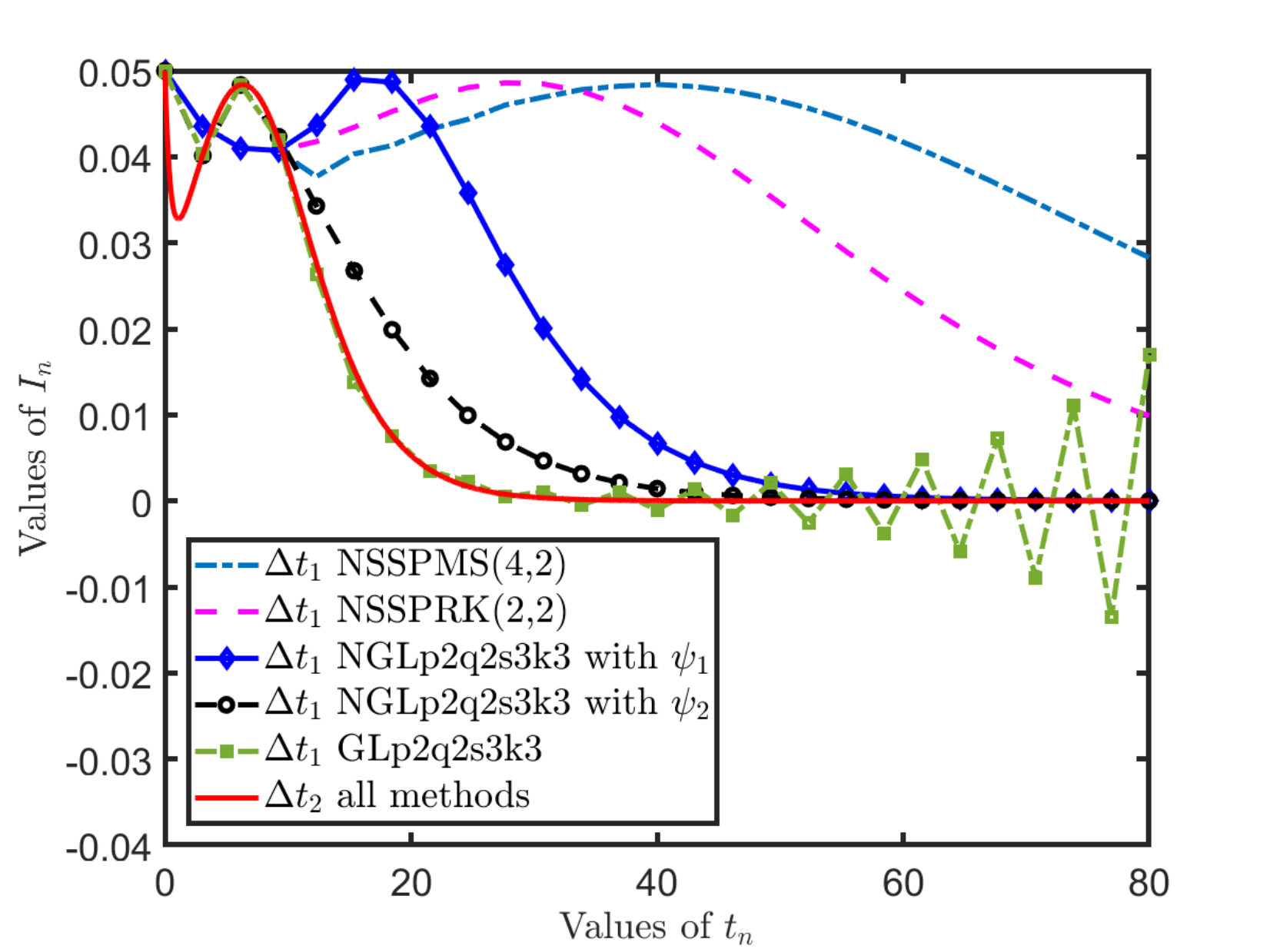}
    \caption{The plots of the different second-order methods solving the SEIR system with $\Delta t_1 = T/26$ and $\Delta t_2 = T/4000$. For $\Delta t_2$, all of the plots are very close to each other, making only one of them visible.}
    \label{fig:seir_compare2}
\end{figure}

The values of $I^n$ produced by third-order methods can be seen in Figure \ref{fig:seir_compare3}, where $T=40$, $(S_0, E_0, I_0, R_0)=(0.9, 0, 0.1, 0)$ and $\gamma=5$. In the upper panels, we used $\Delta t_1 = T/55$. On the left, all the nonstandard methods are plotted, while on the right, only the multistep multistage methods (along with standard counterparts) are present. In the lower panels, $\Delta t_1 = T/12$ is used, and out of the two standard multistep multistage methods, only GLp3q2s3k2 is present on the right. As we can see, as we increase the timestep, the standard methods start to oscillate, while the nonstandard ones remain bounded. Like in the case of second-order methods, the Runge-Kutta and multistep methods produce much bigger errors than the multistep multistage ones. Also, the use of $\psi_2$ should be favored over $\psi_1$.

\begin{figure}[!htbp]
    \centering
    \includegraphics[width=0.45\linewidth]{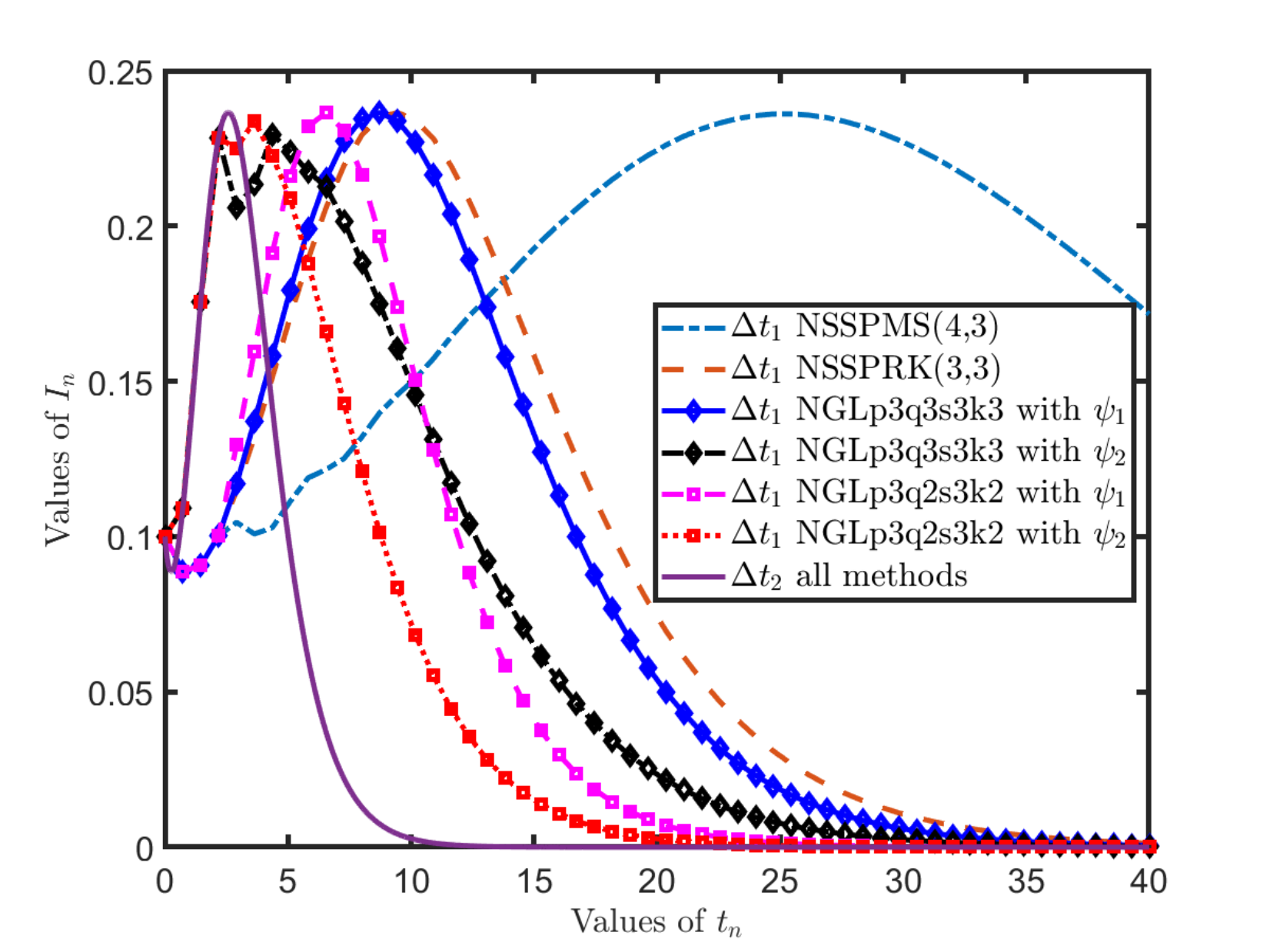}
    \includegraphics[width=0.45\linewidth]{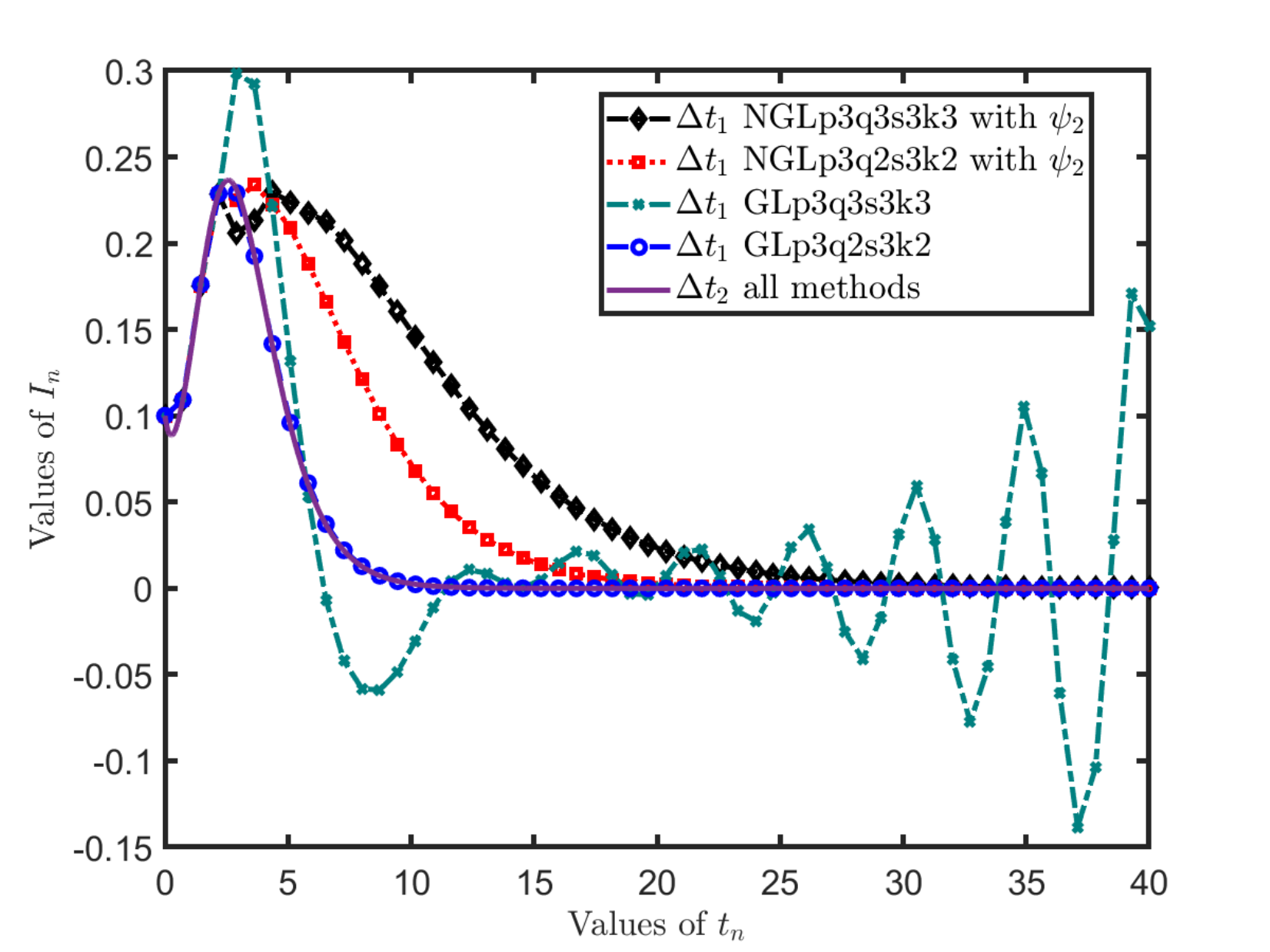}
    \includegraphics[width=0.45\linewidth]{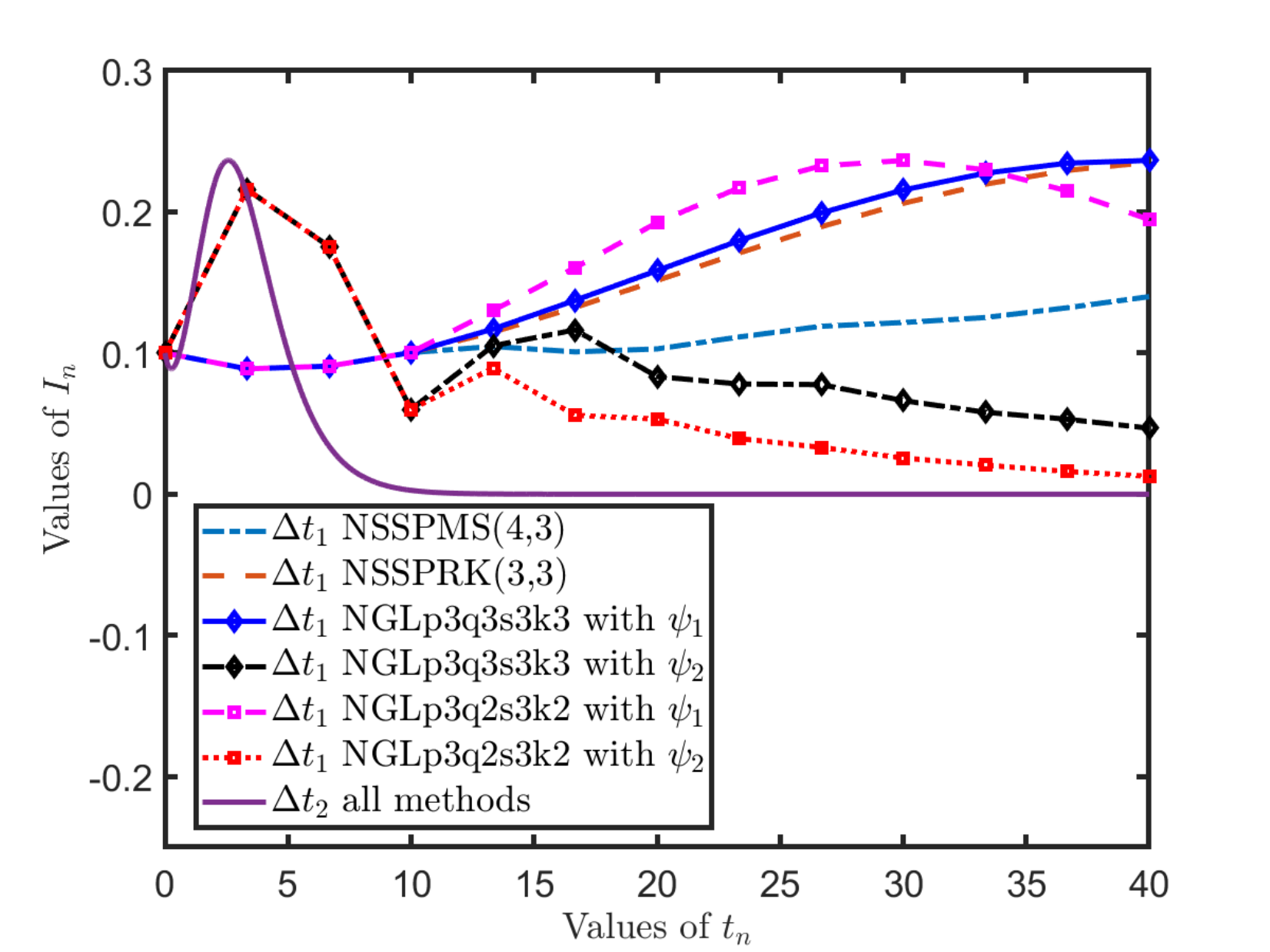}
    \includegraphics[width=0.45\linewidth]{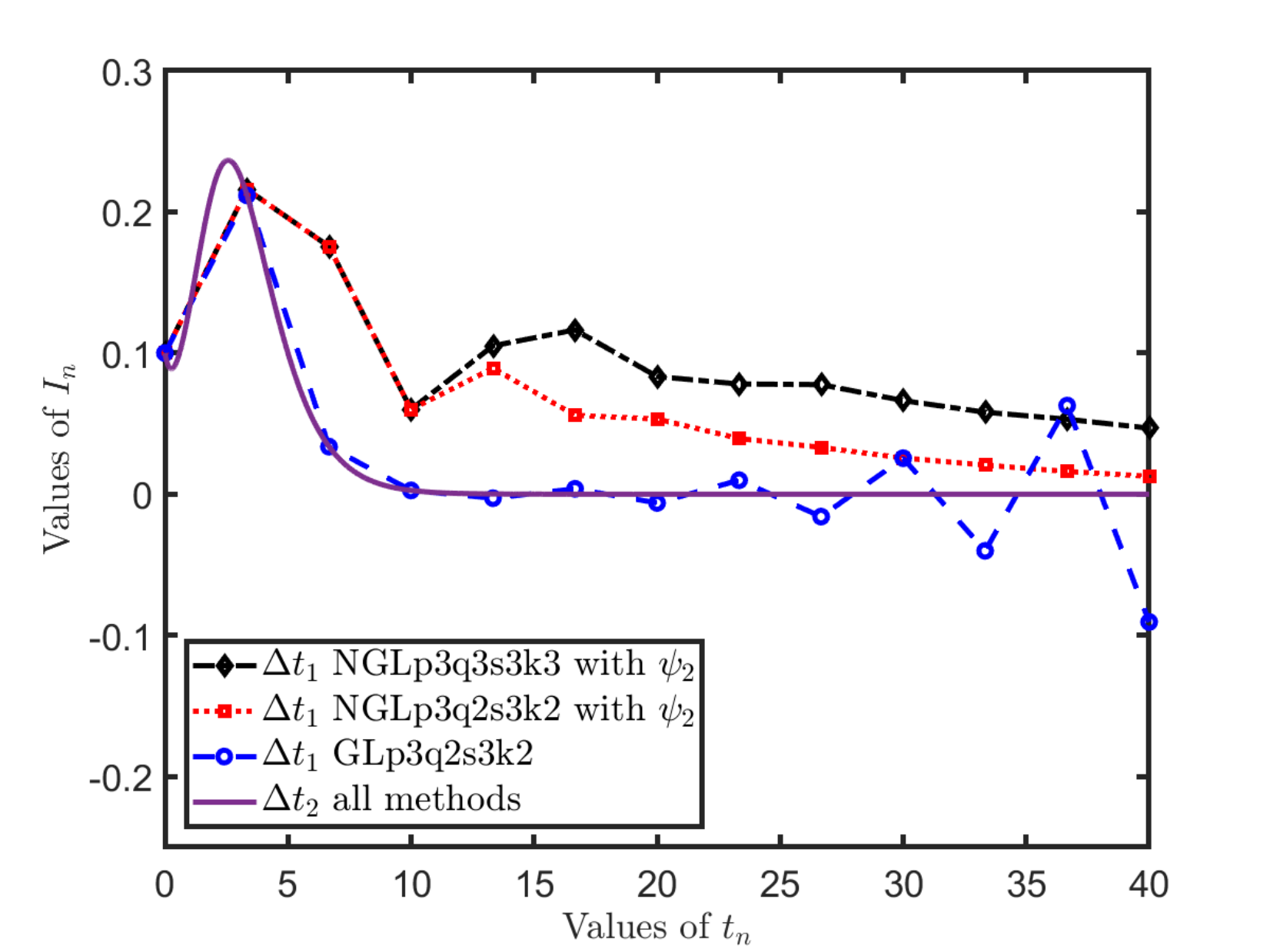}
    \caption{The plots of the different third-order methods solving the SEIR system with $\Delta t_2 = T/2000$ and $\Delta t_1 = T/55$ (upper panels) and $\Delta t_1 = T/12$ (lower panels). For $\Delta t_2$, all of the plots are very close to each other, making only one of them visible.}
    \label{fig:seir_compare3}
\end{figure}

Fourth-order methods can be seen in Figure \ref{fig:seir_compare4}, which are run with the same constants as the previous case, but with timesteps $\Delta t_1 = T/18$ and $\Delta t_2 = T/2000$. The methods behave similarly to the previous case. Here we only one value of $\Delta t_1$ since both of the standard multistep multistage methods become unstable around the same value of $\Delta t$.

\begin{figure}[!htbp]
    \centering
    \includegraphics[width=0.45\linewidth]{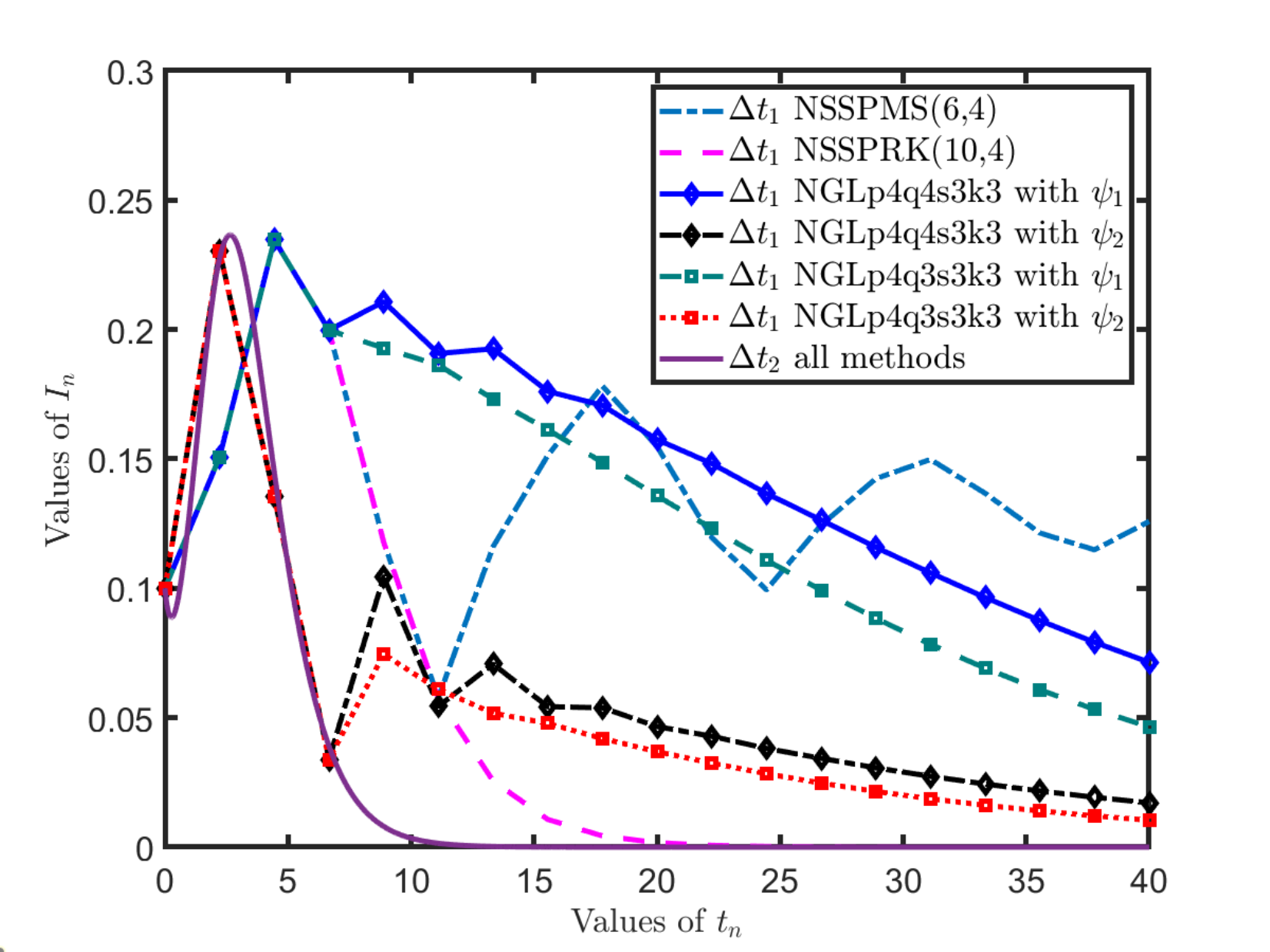}
    \includegraphics[width=0.45\linewidth]{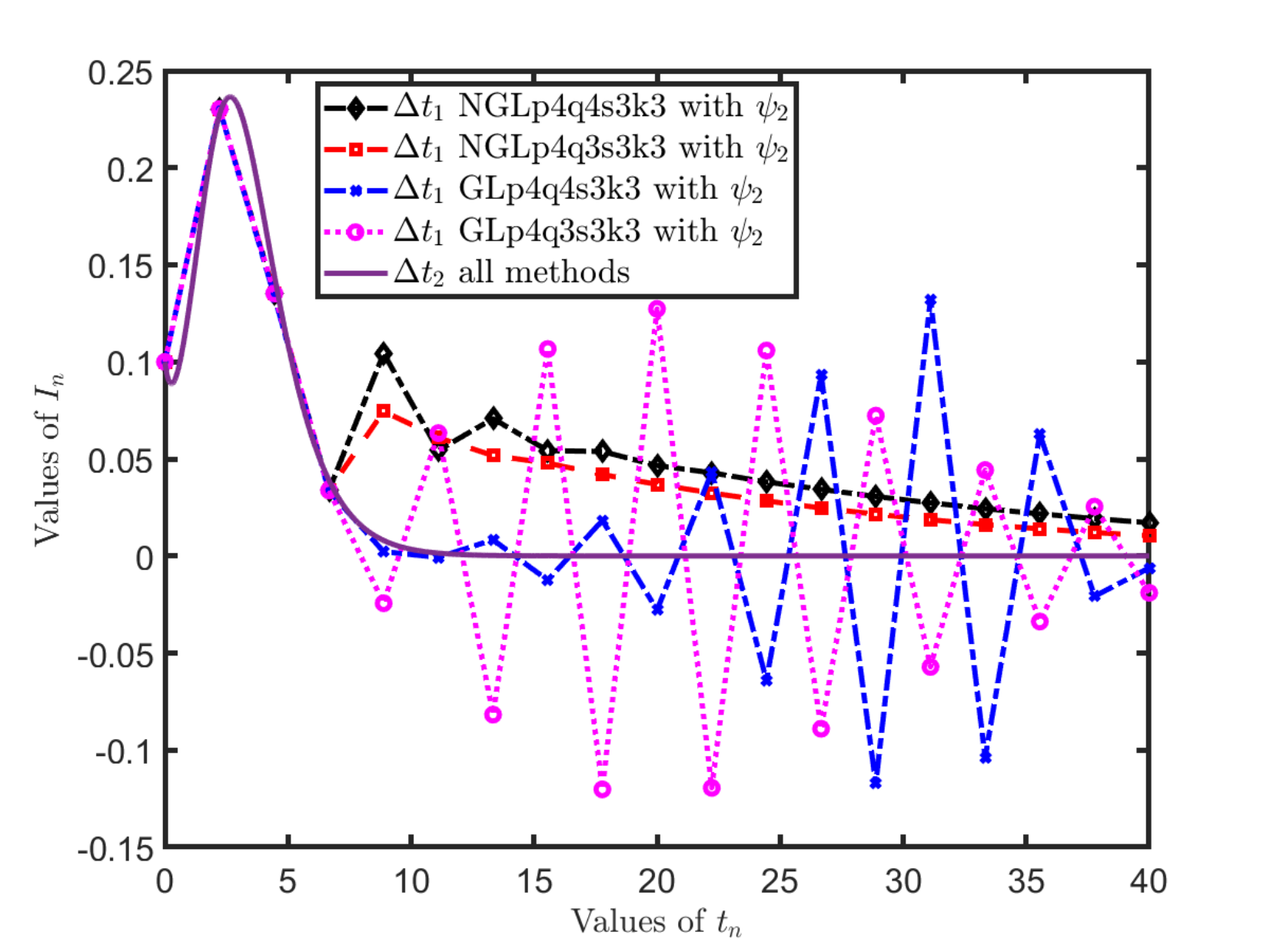}
    \caption{The plots of the different fourth-order methods solving the SEIR system with $\Delta t_1 = T/18$ and $\Delta t_2 = T/2000$. For $\Delta t_2$, all of the plots are very close to each other, making only one of them visible.}
    \label{fig:seir_compare4}
\end{figure}

Finally, we test how close the sufficient bound $\mathcal{C} B_{FE}$ is to the necessary one. Like in the case of the logistic equation, we change the bound for $\varphi$ to an arbitrary positive constant, and run the methods with $\gamma=5$, $T=100$, starting process $\psi_2$, initial condition $(S_0, E_0, I_0, R_0)=(1-I_0, 0, I_0, 0)$, where $I_0$ is chosen to be $1000$ different values from the interval $[10^{-3}, \; 1-10^{-3}]$ and consider $1000$ possible values of timesteps in the interval $[0.5,\; 10]$, except for method NGLp4q4s3k3 where we used only 100 values (because of the high computational runtime). We say that the method behaves as expected for a fixed value of $I_0$ and with this given bound for $\varphi$, if the given property is preserved for every possible timestep on the interval $t \in [0,100]$. By a bisection method, we determine this bound $B^*$ for the boundedness property (i.e., $S_n \geq 0$, $E_n \geq 0$, $I_n \geq 0$ and $R_n \geq 0$ for every $n$). The values of $B^*$ for different values of $\tilde{x}$ can be seen in Figure \ref{fig:seir_bound}. As we can see, the necessary bounds differ from method to method considerably, and in the case of third-order methods, they get relatively close to the sufficient bound given by Theorem \ref{th:preserv}.

\begin{figure}[!htbp]
    \centering
    \includegraphics[height=0.35\linewidth]{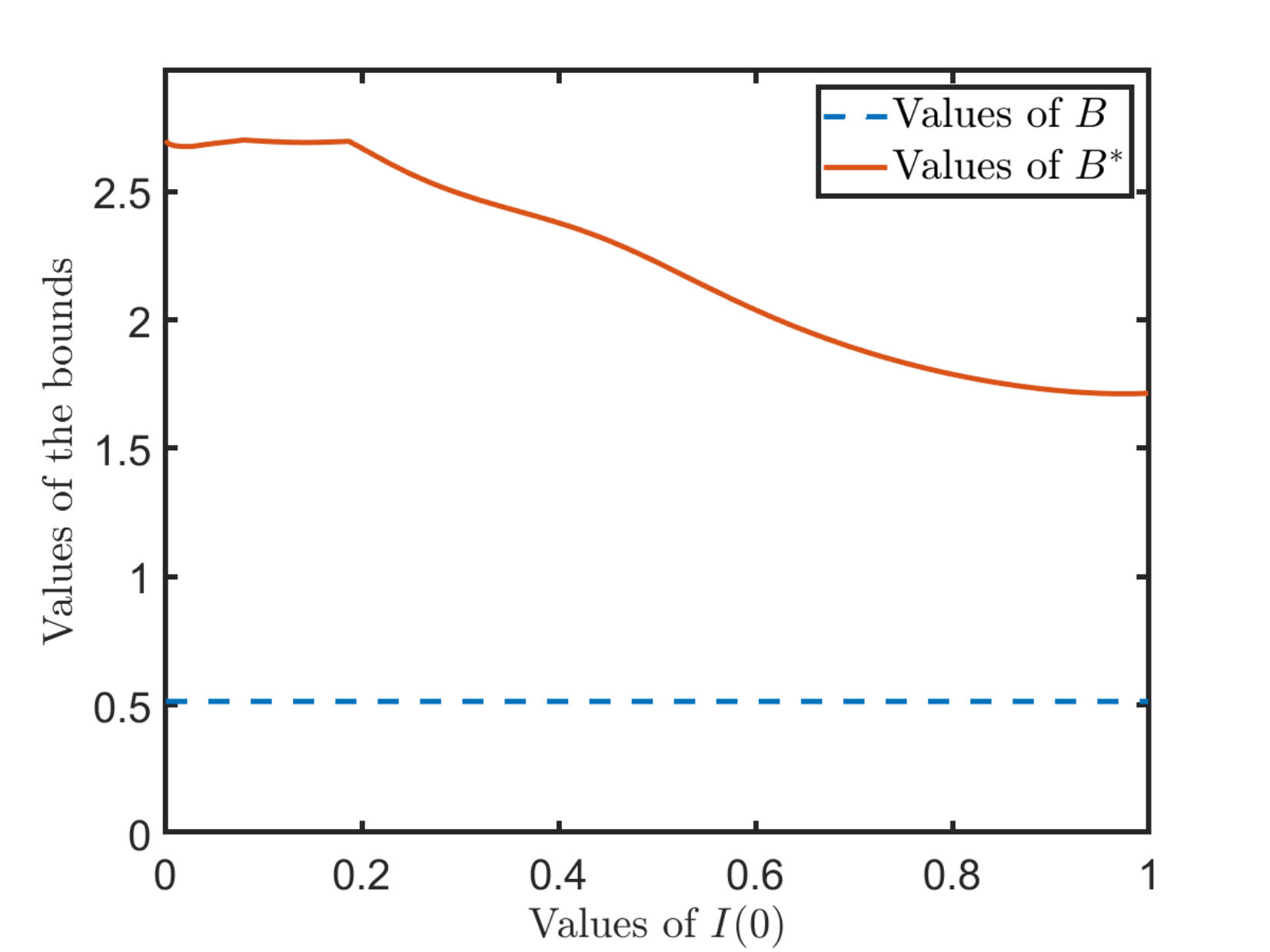}
    \includegraphics[height=0.35\linewidth]{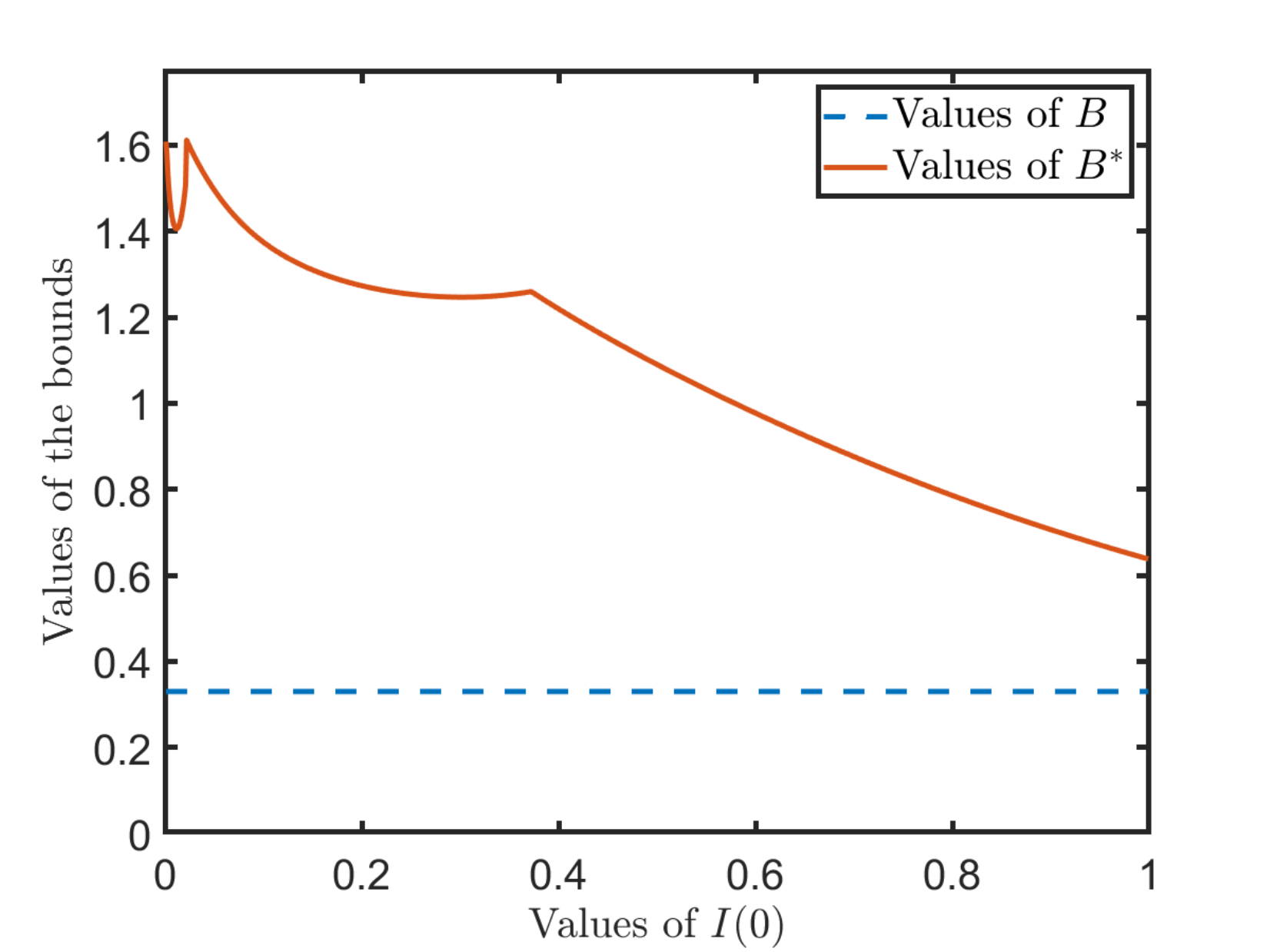}
    \includegraphics[height=0.35\linewidth]{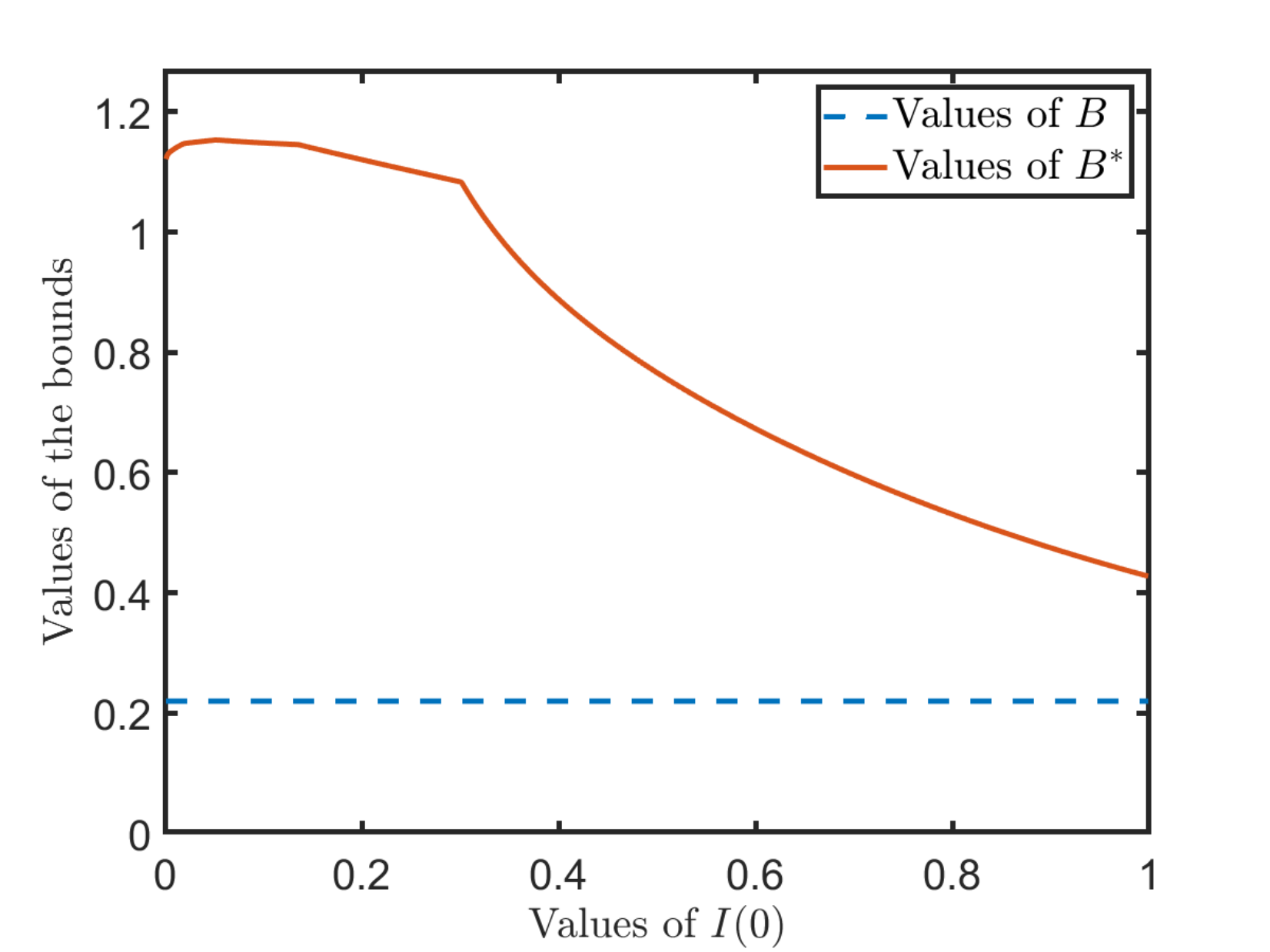}
    \includegraphics[height=0.35\linewidth]{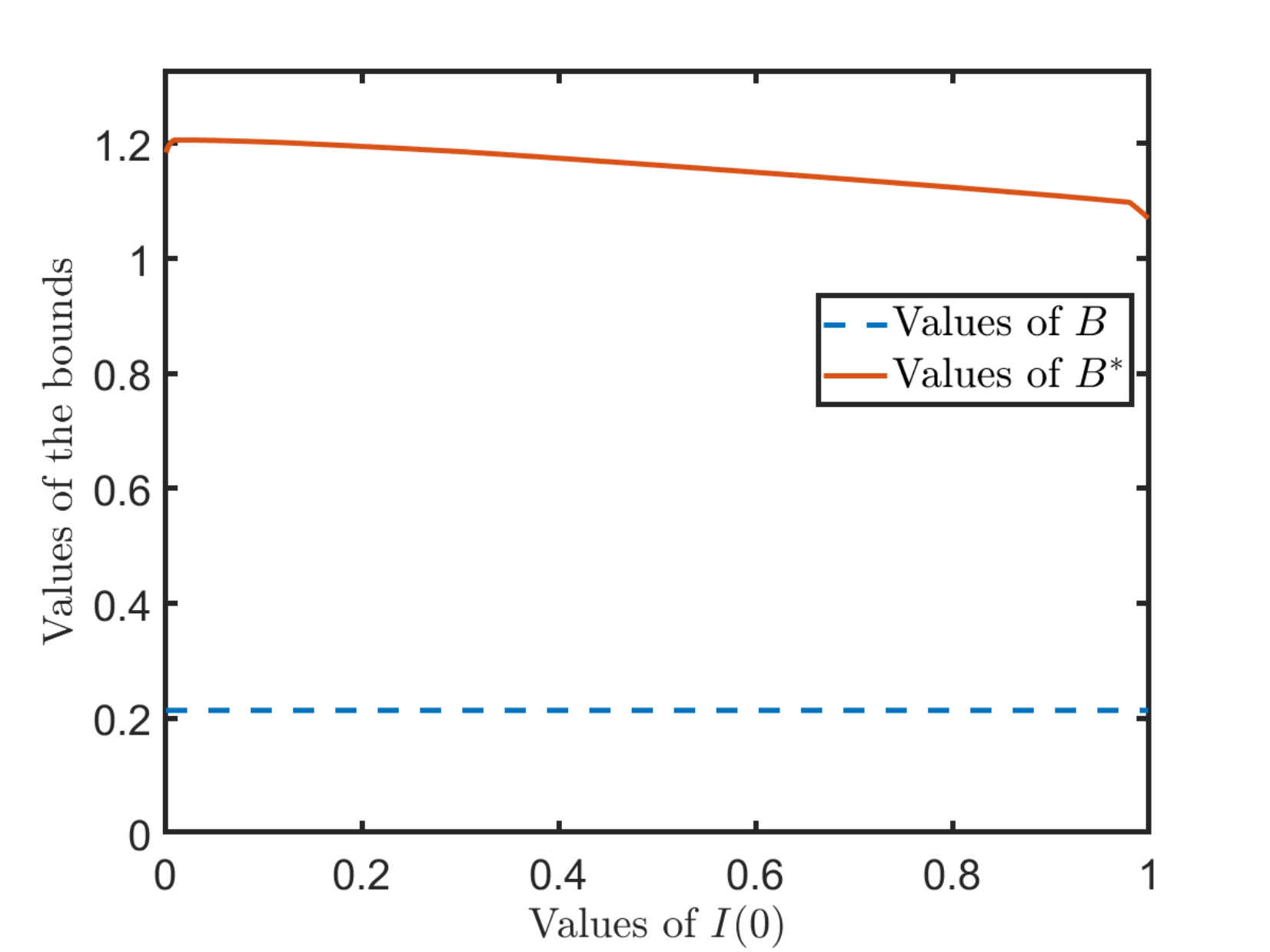}
    \includegraphics[height=0.35\linewidth]{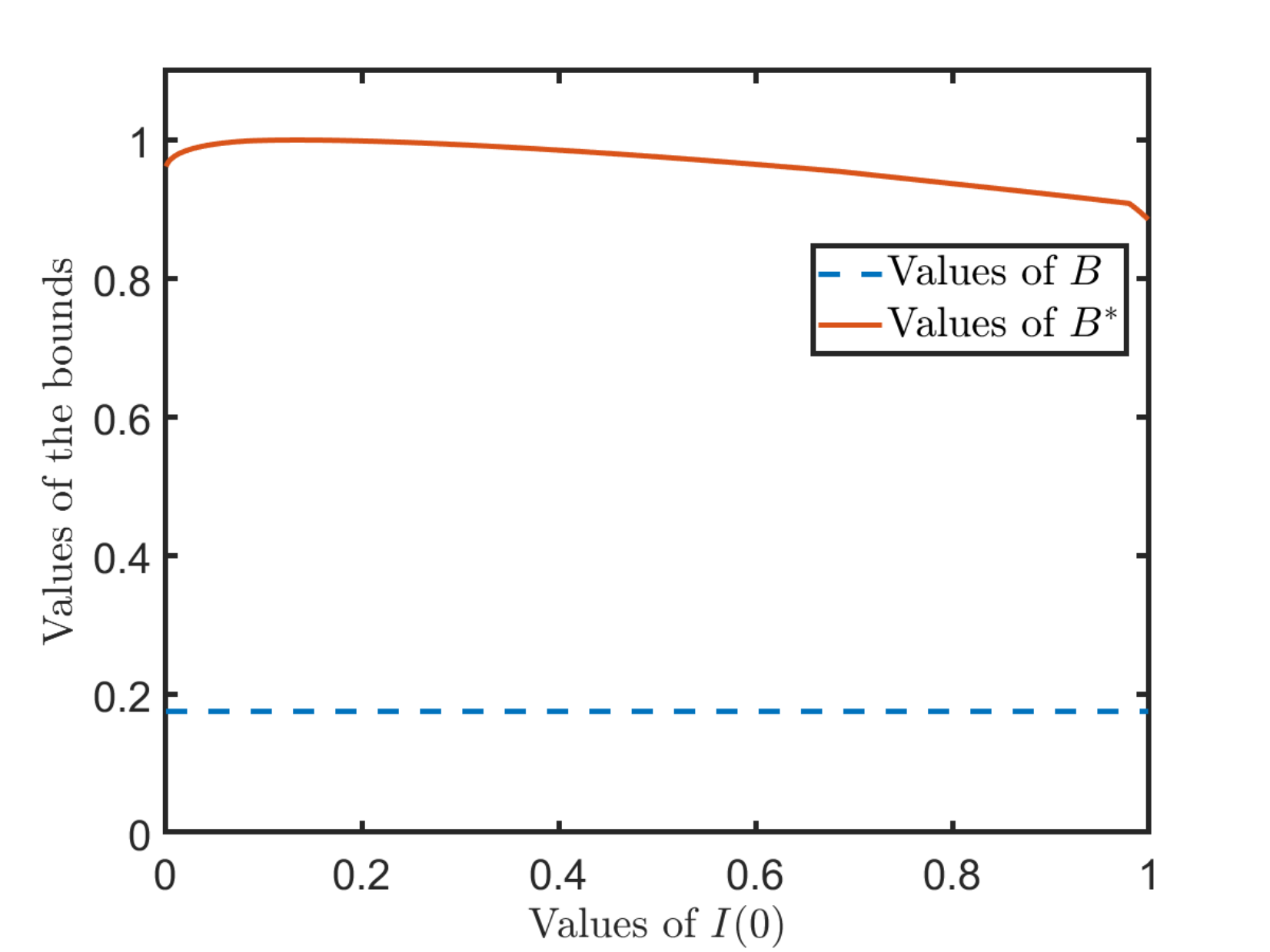}
    \caption{The sufficient bound $B:=\mathcal{C}B_{FE}$ and the ``real'' bounds $B^{*}$ for methods NGLp2q2s3k3 (top left), NGLp3q2s3k3 (top right), NGLp3q3s3k2 (middle left), NGLp4q3s3k3 (middle right), and NGLp4q4s3k3 (bottom) for the boundedness property with $\Pi=0$.}
    \label{fig:seir_bound}
\end{figure}

\subsubsection{The case $\Pi>0$}
Now let us consider the case $\Pi>0$. It turns out that the errors (and hence, the orders) of the methods along with the preservation of qualitative properties are similar to the case $\Pi=0$; therefore, we only focus on one specific property, mainly, the preservation of the sum $S^n + E^n + I^n + R^n$. From \eqref{eq:sir}, it is easy to see that for the exact solution, $S^n + E^n + I^n + R^n = \Pi t_n + \mathcal{M}$ should hold. In Figure \ref{fig:seir_sum}, we examine how well the different nonstandard methods approximate this value. In this case, the values $T=10$, $(S_0, E_0, I_0, R_0)=(0.8, 0, 0.2, 0)$, $\gamma=5$ were used, along with starting process $\psi_2$. Two different timesteps were observed, namely, $\Delta t=0.5$ for the left panels, and $\Delta t=0.1$ for the right ones. As we can see, in the case of the second- and third-order methods, the multistep multistage methods perform the best way, while for the fourth-order ones, the Runge-Kutta method is the best. For $\Delta t_2$, almost every method performs as expected, except for the multistep methods (but these also get close to the exact value for smaller values of $\Delta t$).

\begin{figure}[!htbp]
    \centering
    \includegraphics[width=0.45\linewidth]{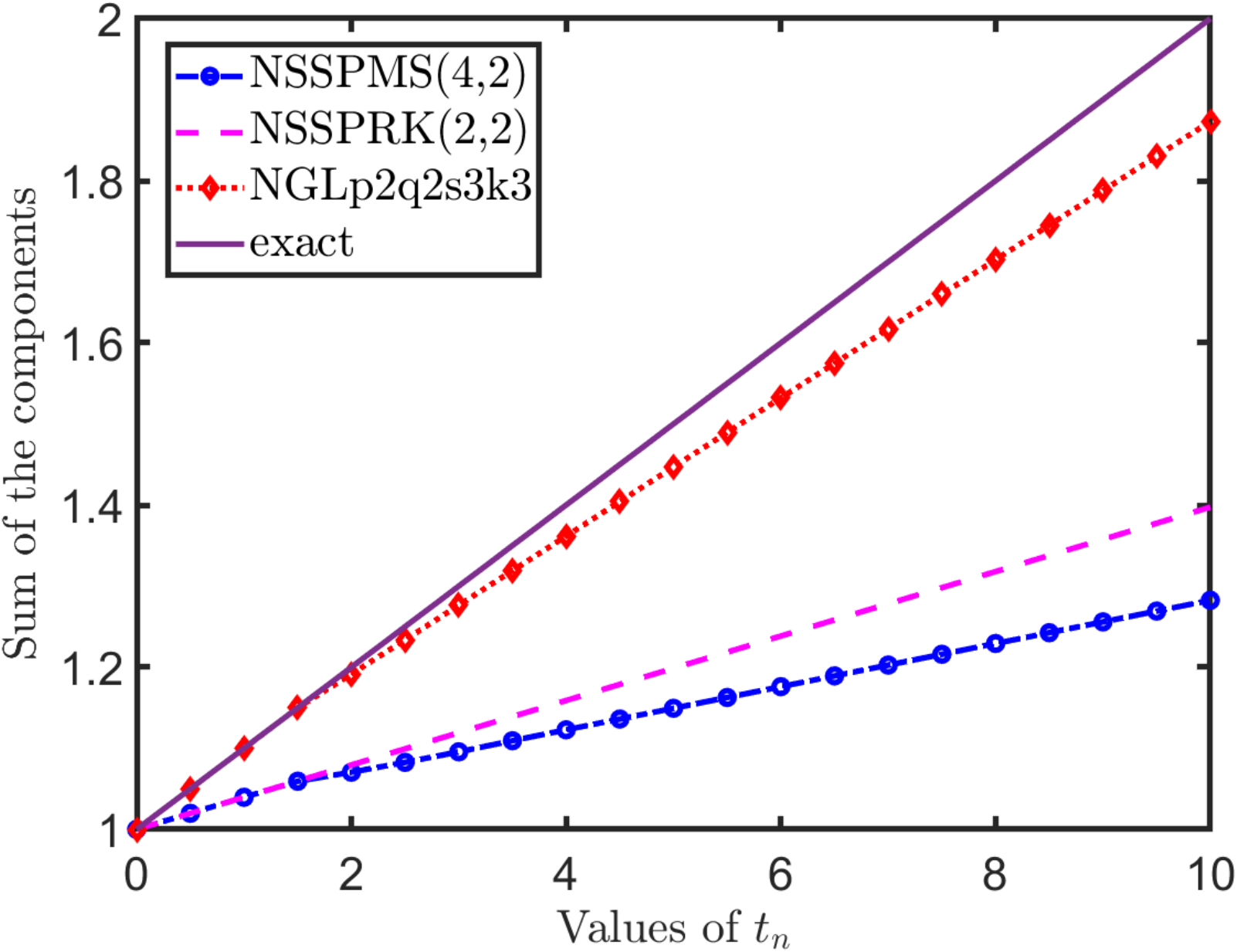}
    \includegraphics[width=0.45\linewidth]{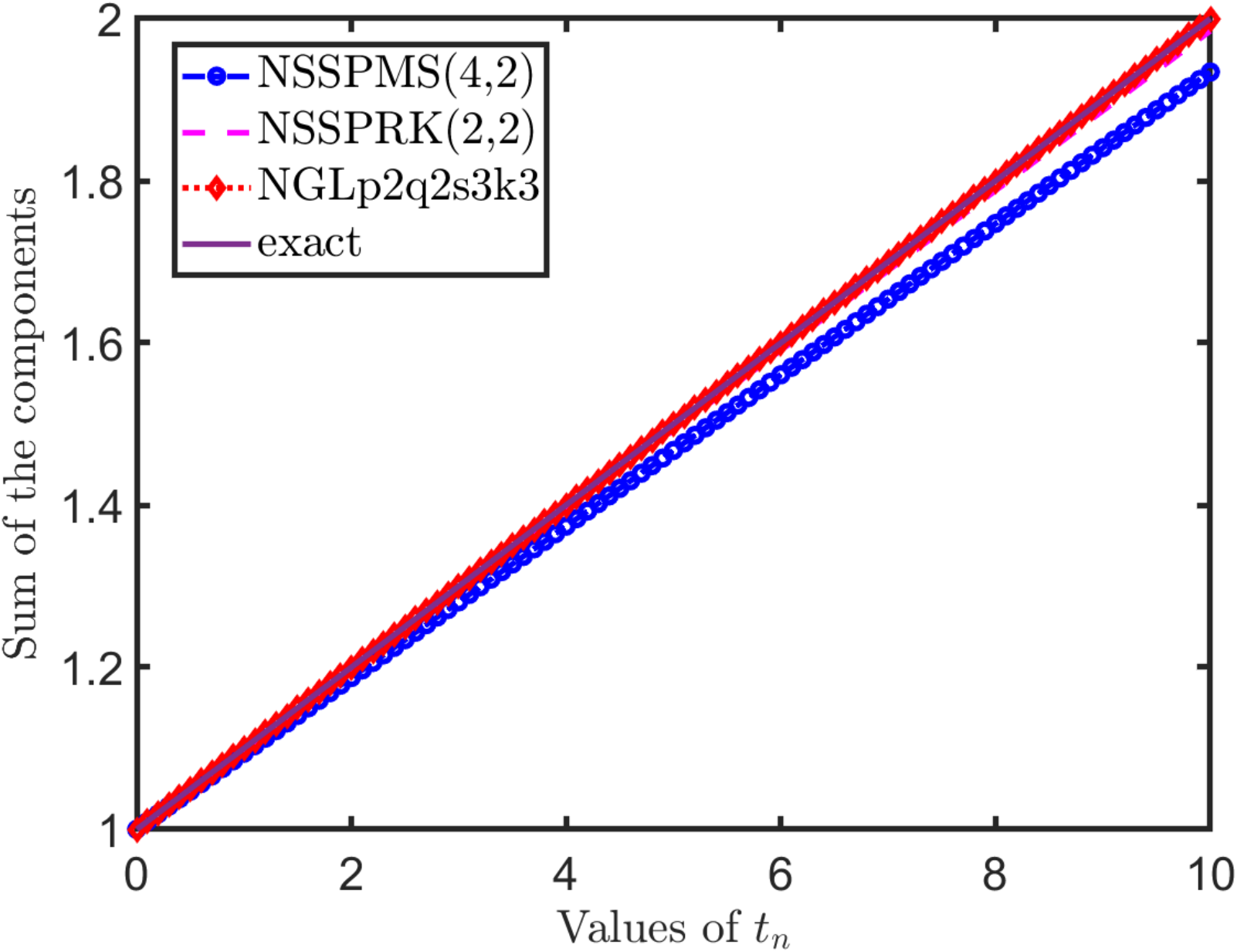}
    \includegraphics[width=0.45\linewidth]{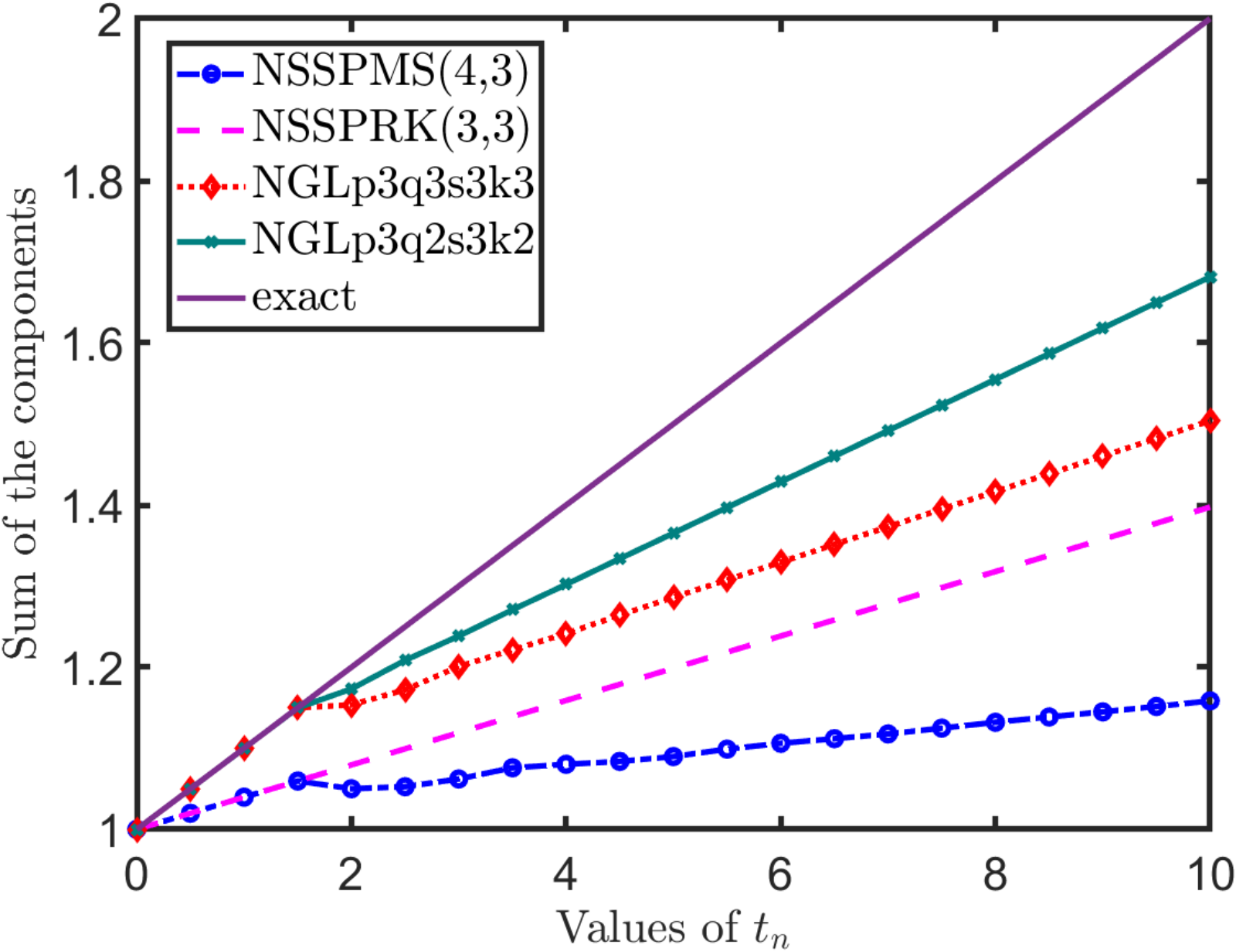}
    \includegraphics[width=0.45\linewidth]{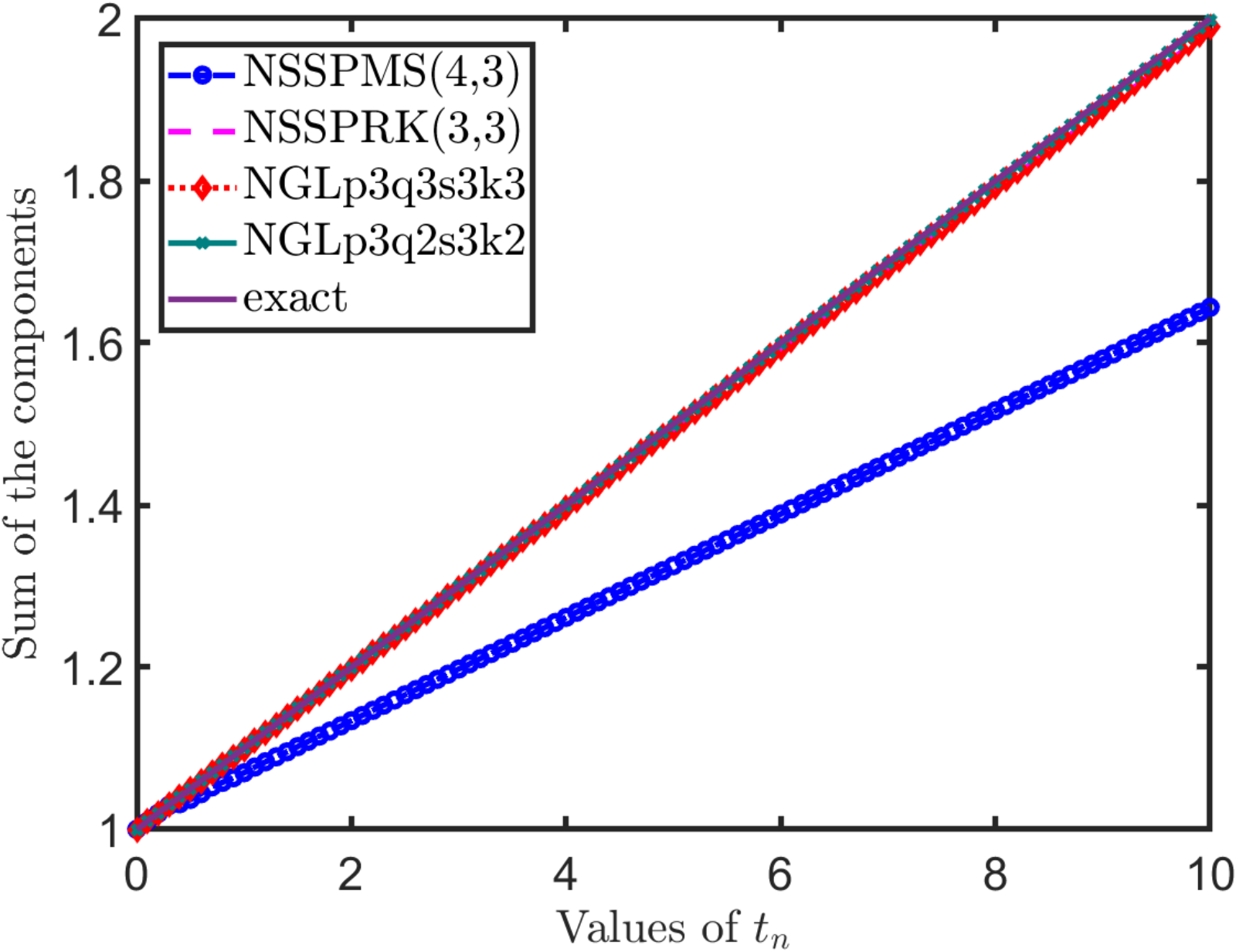}
    \includegraphics[width=0.45\linewidth]{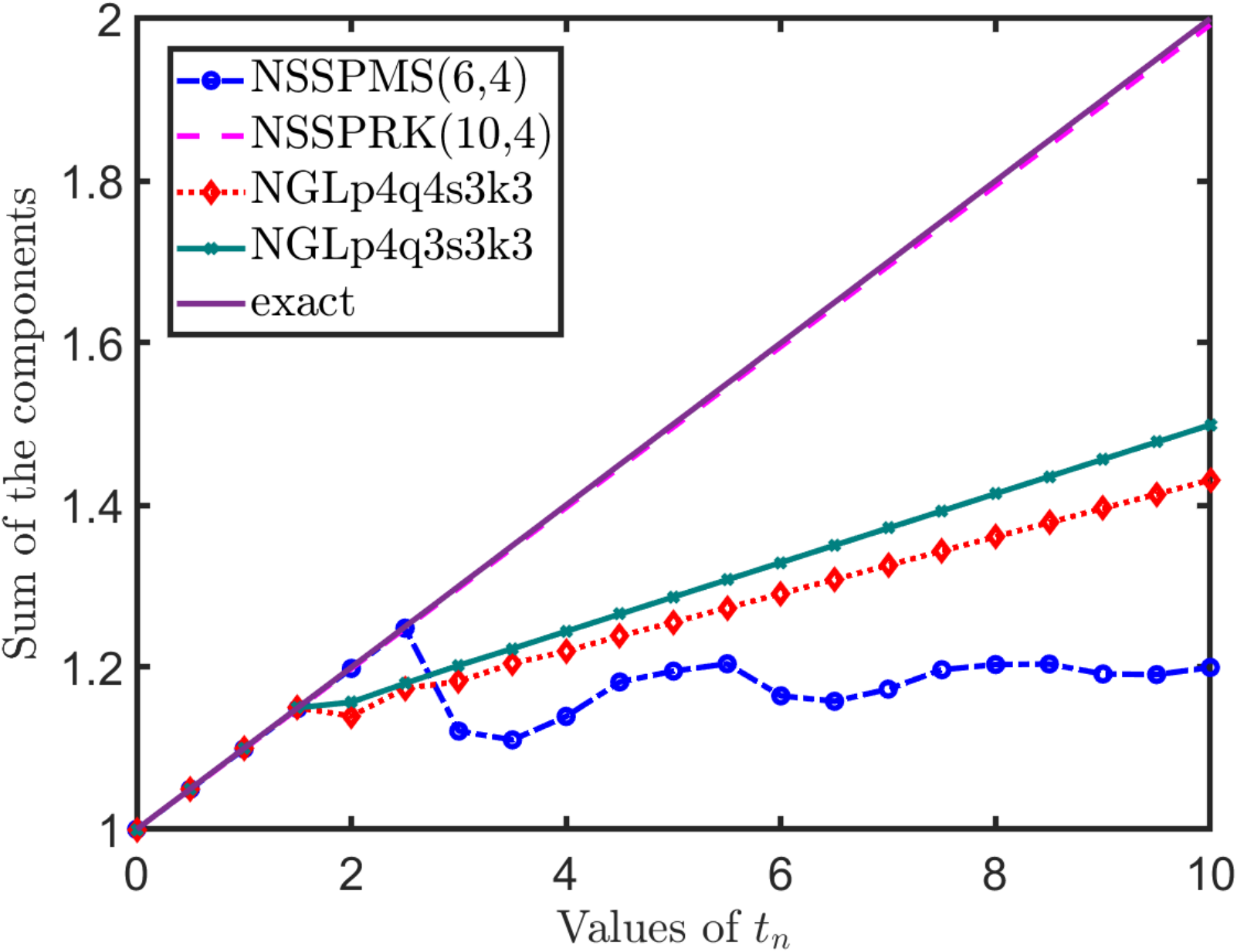}
    \includegraphics[width=0.45\linewidth]{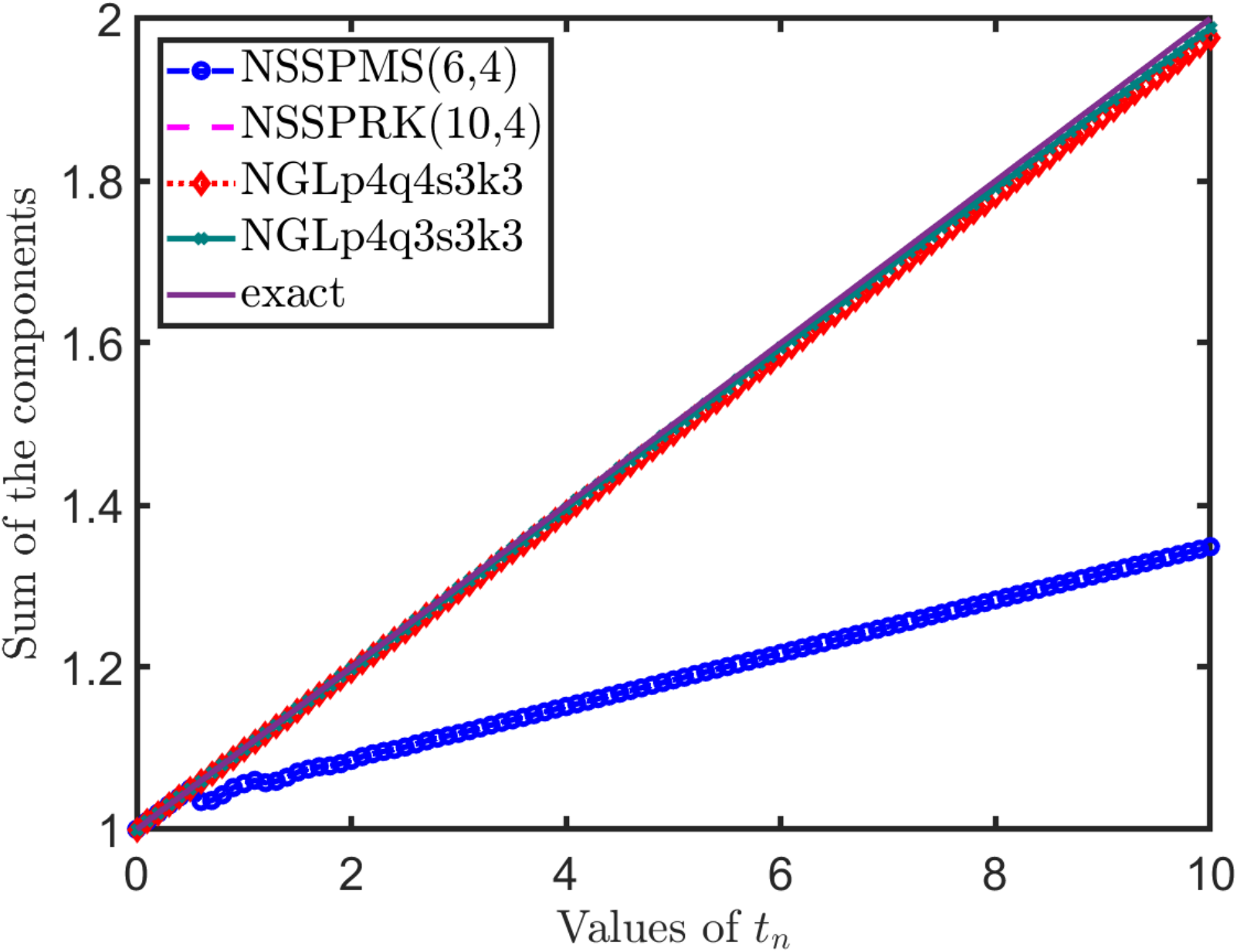}
    \caption{\rev{The sum $S^n+E^n+I^n+R^n$ for $\Pi=0.1$ with $\Delta t=0.5$ (left panels), and for $\Delta t=0.1$ (right panels). The methods are grouped according to their orders. Method NSSPRK(10,4) is always very close to the exact value, therefore not visible on the graphs.}}
    \label{fig:seir_sum}
\end{figure}

%\subsection{A partial differential equation}

\section{Conclusions and discussion}\label{sec:conc}
In the present work, we introduced nonstandard multistep multistage methods. As it turned out in Section \ref{sec:conv}, the convergence of these methods and their standard counterparts is equivalent. During the proof, the nonstandard versions of general linear methods were also defined. Moreover, in Section \ref{sec:order} it was shown that the order of these methods is also the same, if and only if the Taylor series of $\varphi$ and the function $f(x)=x$ coincide to a given number of terms. In Section \ref{sec:preserv}, it was also proven that if function $\varphi$ is bounded by some number which can be determined a priori, then the methods preserve some qualitative properties of the original, continuous model for every possible choice of $\Delta t$.

Furthermore, several numerical experiments were also conducted in Section \ref{sec:num}. These indicate that for small values of $\Delta t$, the errors of the nonstandard multistep multistage methods are between the errors of the nonstandard versions of Runge-Kutta and multistep methods, sometimes even performing better than the Runge-Kutta ones. However, their performance for larger values of $\Delta t$ seems to be the best out of these three. This favorable performance is evident in the second-order case, and it is less and less apparent as we increase the order of the method. One aspect that should be noted is that the computation of multistep multistage methods is more time-consuming than the others, which is the price one has to pay for the higher accuracies mentioned before. It was also clear that the starting process plays a crucial role in the performance of the scheme.

Another usual aspect of nonstandard methods (apart from the change of the denominator function) is the application of nonlocal approximations, which was omitted until this point. In a later work, it would be interesting to combine the approach of the present paper with some nonlocal ideas. It would also be interesting to utilize the more general denominator functions, i.e., those that can also depend on $y_n$ (see \cite{alal, alal23, gupta, hoang24, hoang24new, conte25}) in a later work, perhaps resulting in schemes with smaller errors or less restrictive bounds on $\varphi$.

It is also important to point out that the present methods are explicit, while implicit methods usually have better stability properties. It might be of interest to combine present concepts with those. Moreover, since in Section \ref{sec:conv} nonstandard general linear methods were also introduced, the use of such methods which do not fit into the framework of multistep multistage methods could also produce interesting results (the 'standard' forms of such methods can be found in e.g. \cite{butcherbook}). Finally, it should also be mentioned that the application of the present methods to semidiscretized partial differential equations might also be fruitful.

\bigskip

\noindent\textbf{Acknowledgements} 
The author would like to thank Prof. István Faragó for his insightful comments during the preparation of the article.

\bigskip

\noindent\textbf{Author Contributions} 
The author is the sole contributor to this paper.

\bigskip

\noindent\textbf{Funding:} This research has been supported by the National Research, Development and Innovation Office – NKFIH, grant no. K137699.
The research reported in this paper is also part of project no. BME-NVA-02, implemented with the support provided by the Ministry of Innovation and Technology of Hungary from the National Research, Development and Innovation Fund, financed under the TKP2021 funding scheme.

\bigskip

\noindent \textbf{Data Availability} The data used to support the findings of this study are available from the corresponding
author upon request.

%\section*{Acknowledgement}

\section*{Declarations}
\textbf{Conflict of interest:} The authors declare no competing interests.

%\begin{appendices}

\section*{Appendix A. The coefficients of the different methods}\label{secA1}
Here we list the coefficients of the different strongly-stable multistep and Runge-Kutta methods used in the paper, along with their SSP coefficients.
\begin{itemize}
    \item Explicit SSPMS(4,2) \cite{shu}, with $\mathcal{C}=2/3$:
    $$ \alpha_{[n-1]}^{2,1}=8/9,\quad \beta_{[n-1]}^{2,1}=4/3, \quad \alpha_{[n-4]}^{2,1}=1/9. $$
    \item Explicit SSPMS(4,3) \cite{shu}, with $\mathcal{C}=1/3$:
    $$ \alpha_{[n-1]}^{2,1}=16/27,\quad \beta_{[n-1]}^{2,1}=16/81, \quad \alpha_{[n-4]}^{2,1}=11/27, \quad \beta_{[n-4]}^{2,1}=4/9. $$
    \item Explicit SSPMS(6,4) \cite{ketcheson1,ketcheson2}, with $\mathcal{C}=0.1648$:
    $$ \alpha_{[n-1]}^{2,1}=0.342460855717007,\quad \beta_{[n-1]}^{2,1}=2.078553105578060,$$
    $$ \alpha_{[n-4]}^{2,1}=0.191798259434736,\quad \beta_{[n-4]}^{2,1}=1.164112222279710, $$
    $$ \alpha_{[n-5]}^{2,1}=0.093562124939008,\quad \beta_{[n-5]}^{2,1}=0.567871749748709, $$
    $$ \alpha_{[n-6]}^{2,1}=0.372178759909247. $$
    \item Explicit SSPRK(2,2) \cite{gottlieb98} with $\mathcal{C}=1$:
    $$ \alpha_{[n-1]}^{2,1}=1,\quad \beta_{[n-1]}^{2,1}=1,\quad \alpha_{[n-1]}^{3,1}=1/2, \quad \alpha_{[n-1]}^{3,2}=1/2,\quad \beta_{[n-1]}^{3,2}=1/2.$$
    %\begin{align*}
    %    u^{(1)} &= u^n + \Delta t f(u^n),\\
    %    u^{n+1} &= \dfrac{1}{2} u^n + \dfrac{1}{2} \left( u^{(1)} + \Delta t f(u^{(1)}) \right).
    %\end{align*}
    \item Explicit SSPRK(3,3) \cite{sspbook} with $\mathcal{C}=1$:
    $$ \alpha_{[n-1]}^{2,1}=\beta_{[n-1]}^{2,1}=1,\quad \alpha_{[n-1]}^{3,1}=3/4, \quad \alpha_{[n-1]}^{3,2}=\beta_{[n-1]}^{3,2}=1/4, $$
    $$ \alpha_{[n-1]}^{4,1}=1/3, \quad \alpha_{[n-1]}^{4,3}=\beta_{[n-1]}^{4,3}=2/3 .$$
    % \begin{align*}
    %     u^{(1)} &= u^n + \Delta t f(u^n),\\
    %     u^{(2)} &= \dfrac{3}{4} u^n + \dfrac{1}{4} u^{(1)} + \dfrac{1}{4} \Delta t f(u^{(1)}),\\
    %     u^{n+1} &= \dfrac{1}{3} u^n + \dfrac{2}{3} \left( u^{(2)} + \Delta t f(u^{(2)}) \right).
    % \end{align*}
    \item Explicit SSPRK(10,4) \cite{ketcheson2} with $\mathcal{C}=6$:
    $$ \alpha_{[n-1]}^{2,1}=\alpha_{[n-1]}^{3,2}=\alpha_{[n-1]}^{4,3}=\alpha_{[n-1]}^{5,4}=1,\quad \beta_{[n-1]}^{2,1}=\beta_{[n-1]}^{3,2}=\beta_{[n-1]}^{4,3}=\beta_{[n-1]}^{5,4}=1/6,$$
    $$ \alpha_{[n-1]}^{6,1} = 3/5, \quad \alpha_{[n-1]}^{6,5} = 2/5, \quad \beta_{[n-1]}^{6,5} = 1/15, $$
    $$ \alpha_{[n-1]}^{7,6}=\alpha_{[n-1]}^{8,7}=\alpha_{[n-1]}^{9,8}=\alpha_{[n-1]}^{10,9}=1,\quad \beta_{[n-1]}^{7,6}=\beta_{[n-1]}^{8,7}=\beta_{[n-1]}^{9,8}=\beta_{[n-1]}^{10,9}=1/6,$$
    $$ \alpha_{[n-1]}^{11,1}= 1/25, \quad \alpha_{[n-1]}^{11,5}= 9/25, \quad \beta_{[n-1]}^{11,5}= 3/50, \quad \alpha_{[n-1]}^{11,10}= 3/5, \quad \beta_{[n-1]}^{11,10}= 1/10. $$
    % \begin{align*}
    %     u^{(1)} &= u^n + \dfrac{1}{6}\Delta t f(u^n),\\
    %     u^{(j+1)} &= u^{(j)} + \dfrac{1}{6}\Delta t f(u^{(j)}), \qquad j=1, 2, 3, \\ 
    %     u^{(5)} &= \dfrac{3}{5} u^n + \dfrac{2}{5}u^{(4)} + \dfrac{1}{15} \Delta t f(u^{(4)}),\\
    %     u^{(j+1)} &= u^{(j)} + \dfrac{1}{6}\Delta t f(u^{(j)}), \qquad j=5,\dots, 8, \\ 
    %     u^{n+1} &= \dfrac{1}{25} u^n + \dfrac{9}{25} u^{(4)} + \dfrac{3}{5} u^{(9)} + \dfrac{3}{50} \Delta t f(u^{(4)}) + \dfrac{1}{10} \Delta t f(u^{(9)}) .
    % \end{align*}
     \item Explicit NGLp2q2s3k3 \cite{const} with $\mathcal{C} = 2.57$:
    $$ \alpha_{[n-1]}^{2,1} = 0.973398050642691, \quad \beta_{[n-1]}^{2,1} = 0.379405979378177, $$
    $$ \alpha_{[n-1]}^{3,2} = 0.979404360713112, \quad \beta_{[n-1]}^{3,2} = 0.381747087369108, $$
    $$ \alpha_{[n-1]}^{4,3} = 0.983666449265926 ,\quad \beta_{[n-1]}^{4,3} = 0.383408341858481, $$
    $$ \alpha_{[n-3]}^{2,1} = 0.026601949357309, \quad \alpha_{[n-3]}^{3,1} = 0.020595639286888, $$
    $$ \alpha_{[n-3]}^{4,1} = 0.016333550734074, $$
    $$ c = \left[ 0, \; 0.326202080663559, \; 0.660039549070913, \; 1\right]^T. $$
    \item Explicit NGLp3q2s3k2 \cite{const} with $\mathcal{C} = 1.65$:
    $$ \alpha_{[n-1]}^{2,1} = 0.857663370271785, \quad \beta_{[n-1]}^{2,1} = 0.519611900224726, $$
    $$ \alpha_{[n-1]}^{3,2} = 0.770413480757674, \quad \beta_{[n-1]}^{3,2} = 0.466751905900312, $$
    $$ \alpha_{[n-1]}^{4,3} = 0.841153332326449,\quad \beta_{[n-1]}^{4,3} = 0.509609360199215, $$
    $$ \alpha_{[n-2]}^{2,1} = 0.142336629728215,  $$
    $$ \alpha_{[n-2]}^{3,1} = 0.229586519242326, \quad \beta_{[n-2]}^{3,1} = 0.129608154625262, $$
    $$ \alpha_{[n-2]}^{4,1} = 0.158846667673551, \quad \beta_{[n-2]}^{3,1} = 0.096236614148583, $$
    $$ c = \left[ 0, \; 0.377275270496511, \; 0.657431495630257, \; 1\right]^T. $$
    \item Explicit NGLp3q3s2k3 \cite{const} with $\mathcal{C} = 1.10$:
    $$ \alpha_{[n-1]}^{2,1} = 0.803084592008657, \quad \beta_{[n-1]}^{2,1} = 0.729588628543267, $$
    $$ \alpha_{[n-1]}^{3,2} = 0.846696784194569, \quad \beta_{[n-1]}^{3,2} = 0.769209559888867, $$
    $$ \alpha_{[n-3]}^{2,1} = 0.196915407991343, \quad \beta_{[n-3]}^{2,1} = 0.140265790357552,  $$
    $$ \alpha_{[n-3]}^{3,1} = 0.153303215805431, \quad \beta_{[n-3]}^{3,1} = 0.134349217930499, $$
    $$ c = \left[ 0, \; 0.476023602918134, \;  1\right]^T. $$
    \item Explicit NGLp4q3s3k3 \cite{const} with $\mathcal{C} = 1.07$:
    $$ \alpha_{[n-1]}^{2,1} = 0.79779687008967, \quad \beta_{[n-1]}^{2,1} = 0.742235840146894, $$
    $$ \alpha_{[n-1]}^{3,2} = 0.685074051305928, \quad \beta_{[n-1]}^{3,2} = 0.637363385465199,, $$
    $$ \alpha_{[n-1]}^{4,1} = 0.39703332125451, \quad \beta_{[n-1]}^{4,1} = 0.369382698548981, $$
    $$ \alpha_{[n-1]}^{4,3} = 0.409097066488626, \quad \beta_{[n-1]}^{4,3} = 0.380606287428385, $$
    $$ \alpha_{[n-2]}^{3,1} = 0.267934431946272, \quad \beta_{[n-2]}^{3,1} = 0.249274653304665,  $$
    $$ \alpha_{[n-2]}^{4,1} = 0.149202105282063, \quad \beta_{[n-2]}^{4,1} = 0.138811211371724, $$
    $$ \alpha_{[n-3]}^{2,1} = 0.20220312991033, \quad \beta_{[n-3]}^{2,1} = 0.144131507391754, $$
    $$ \alpha_{[n-3]}^{3,1} = 0.0469915167478, \quad \alpha_{[n-3]}^{4,1} = 0.044667506974801, $$
    $$ c = \left[ 0, \; 0.481961087717987, \; 0.854899608262766, \;   1\right]^T. $$
    \item Explicit NGLp4q4s3k3 \cite{const} with $\mathcal{C} = 0.88$:
    $$ \alpha_{[n-1]}^{2,1} = 0.501452936754328, \quad \beta_{[n-1]}^{2,1} = 0.570650194053946, $$
    $$ \alpha_{[n-1]}^{3,2} = 0.571621756632096, \quad \beta_{[n-1]}^{3,2} = 0.65050185658275, $$
    $$ \alpha_{[n-1]}^{4,1} = 0.104408345813576, \quad \beta_{[n-1]}^{4,1} = 0.118816021270125, $$
    $$ \alpha_{[n-1]}^{4,3} = 0.555337610608053, \quad \beta_{[n-1]}^{4,3} = 0.631970603881811, $$
    $$ \alpha_{[n-2]}^{2,1} = 0.461766417377124, \quad \beta_{[n-2]}^{2,1} = 0.260645867579256,  $$
    $$ \alpha_{[n-2]}^{3,1} = 0.365441633624919, \quad \beta_{[n-2]}^{3,1} = 0.31755158184828, $$
    $$ \alpha_{[n-2]}^{4,1} = 0.267081022184514, \quad \beta_{[n-2]}^{4,1} = 0.303936473329277, $$
    $$ \alpha_{[n-3]}^{2,1} = 0.036780645868547, \quad \alpha_{[n-3]}^{3,1} = 0.062936609742985, $$
    $$ \alpha_{[n-3]}^{4,1} = 0.073173021393856,  $$
    $$ c = \left[ 0, \; 0.295968352518983, \; 0.645920534894549, \;   1\right]^T. $$
\end{itemize}

\section*{Appendix B. Tables of errors and orders}\label{sec:A2}
On the next pages, we list the errors and orders of the different methods discussed before.

\FloatBarrier
\begin{table}[!hbtp]
    %\ssmall
    \fontsize{6pt}{6pt}
    \centering
    \renewcommand{\arraystretch}{0.5}
    \begin{tabular}{c||c|c||c|c||c|c}
         & \multicolumn{2}{c||}{ \ssmall NSSPMS(4,2)} & \multicolumn{2}{c||}{ \ssmall NSSRK(2,2)} & \multicolumn{2}{c}{ \ssmall NGLp2q2s3k3}  \\ \hline
         %\\[-1em]
         $\Delta t$ & \ssmall errors & \ssmall orders & \ssmall errors & \ssmall orders & \ssmall errors & \ssmall orders  \\ \hline %\\[-1em]
        $0.01$ & $2.6078 \cdot 10^{-3}$ & - & $1.0868 \cdot 10^{-3}$ & -& $6.7697 \cdot 10^{-4}$ & - \\ \hline %\\[-1em]
        $0.01 \cdot 2^{-1}$ & $8.0817 \cdot 10^{-4}$ & $1.6901$ & $2.5850 \cdot 10^{-4}$ & $2.0719$ & $2.0773 \cdot 10^{-4}$ & $1.7044$ \\ \hline %\\[-1em]
        $0.01 \cdot 2^{-2}$ & $2.2041 \cdot 10^{-4}$ & $1.8745$ & $6.3542 \cdot 10^{-5}$ & $2.0244$ & $5.6810 \cdot 10^{-5}$ & $1.8704$ \\ \hline %\\[-1em]
        $0.01 \cdot 2^{-3}$ & $5.7273 \cdot 10^{-5}$ & $1.9442$ & $1.5786 \cdot 10^{-5}$ & $2.0091$ & $1.4813 \cdot 10^{-5}$ & $1.9392$ \\ \hline %\\[-1em]
        $0.01 \cdot 2^{-4}$ & $1.4580 \cdot 10^{-5}$ & $1.9738$ & $3.9362 \cdot 10^{-6}$ & $2.0037$ & $3.7799 \cdot 10^{-6}$ & $1.9705$ \\ \hline %\\[-1em]
        $0.01 \cdot 2^{-5}$ & $3.6773 \cdot 10^{-6}$ & $1.9873$ & $9.8290 \cdot 10^{-7}$ & $2.0017$ & $9.5451 \cdot 10^{-7}$ & $1.9854$ \\ \hline %\\[-1em]
        $0.01 \cdot 2^{-6}$ & $9.2331 \cdot 10^{-7}$ & $1.9938$ & $2.4559 \cdot 10^{-7}$ & $2.0008$ & $2.3982 \cdot 10^{-7}$ & $1.9928$ \\ \hline %\\[-1em]
        $0.01 \cdot 2^{-7}$ & $2.3132 \cdot 10^{-7}$ & $1.9969$ & $6.1381 \cdot 10^{-8}$ & $2.0004$ & $6.0104 \cdot 10^{-8}$ & $1.9964$ \\ \hline %\\[-1em]
        $0.01 \cdot 2^{-8}$ & $5.7891 \cdot 10^{-8}$ & $1.9985$ & $1.5343 \cdot 10^{-8}$ & $2.0002$ & $1.5044 \cdot 10^{-8}$ & $1.9982$ \\ \hline \hline
    \end{tabular}
    \begin{tabular}{c||c|c||c|c||c|c||c|c||}
         & \multicolumn{2}{c||}{\ssmall NSSPMS(4,3)} & \multicolumn{2}{c||}{ \ssmall NSSPRK(3,3)} & \multicolumn{2}{c||}{\ssmall NGLp3q3s3k3} & \multicolumn{2}{c}{\ssmall NGLp3q2s3k2}  \\ \hline %\\[-1.3em]
        $\Delta t$ & \ssmall errors & \ssmall orders & \ssmall errors & \ssmall orders & \ssmall errors & \ssmall orders & \ssmall errors & \ssmall orders \\ \hline %\\[-1.3em]
        $0.01$ & $3.5018 \cdot 10^{-3}$ & - & $4.6497 \cdot 10^{-5}$ & -& $9.1030 \cdot 10^{-5}$ & - & $2.1214 \cdot 10^{-5}$ & - \\ \hline %\\[-1.3em]
        $0.01 \cdot 2^{-1}$ & $2.5367 \cdot 10^{-4}$ & $3.7871$ & $2.7695 \cdot 10^{-6}$ & $4.0694$ & $1.0622 \cdot 10^{-5}$ & $3.0993$ & $2.5464 \cdot 10^{-6}$ & $3.0585$ \\ \hline %\\[-1.3em]
        $0.01 \cdot 2^{-2}$ & $1.8450 \cdot 10^{-5}$ & $3.7812$ & $1.5602 \cdot 10^{-7}$ & $4.1498$ & $1.2453 \cdot 10^{-6}$ & $3.0924$ & $3.0235 \cdot 10^{-7}$ & $3.0741$ \\ \hline %\\[-1.3em]
        $0.01 \cdot 2^{-3}$ & $1.4071 \cdot 10^{-6}$ & $3.7129$ & $7.6170 \cdot 10^{-9}$ & $4.3564$ & $1.4943 \cdot 10^{-7}$ & $3.0589$ & $3.6503 \cdot 10^{-8}$ & $3.0502$\\ \hline %\\[-1.3em]
        $0.01 \cdot 2^{-4}$ & $1.1750 \cdot 10^{-7}$ & $3.5820$ & $2.0933 \cdot 10^{-10}$ & $5.1853$ & $1.8256 \cdot 10^{-8}$ & $3.0331$ & $4.4728 \cdot 10^{-9}$ & $3.0287$\\ \hline %\\[-1.3em]
        $0.01 \cdot 2^{-5}$ & $1.0971 \cdot 10^{-8}$ & $3.4210$ & $2.0159 \cdot 10^{-11}$ & $3.3763$ & $2.2545 \cdot 10^{-9}$ & $3.0175$ & $5.5325 \cdot 10^{-10}$ & $3.0152$\\ \hline %\\[-1.3em]
        $0.01 \cdot 2^{-6}$ & $1.1369 \cdot 10^{-9}$ & $3.2704$ & $5.2145 \cdot 10^{-12}$ & $1.9508$ & $2.8017 \cdot 10^{-10}$ & $3.0084$ & $6.8869\cdot 10^{-11}$ & $3.0060$\\ \hline %\\[-1.3em]
        $0.01 \cdot 2^{-7}$ & $1.2756 \cdot 10^{-10}$ & $3.1559$ & $4.5119 \cdot 10^{-13}$ & $3.5307$ & $3.5137 \cdot 10^{-11}$ & $2.9952$ & $8.7903 \cdot 10^{-12}$ & $2.9699$\\ \hline %\\[-1.3em]
        $0.01 \cdot 2^{-8}$ & $1.5337 \cdot 10^{-11}$ & $3.0561$ & $6.4748 \cdot 10^{-13}$ & - & $4.8379 \cdot 10^{-12}$ & $2.8605$ & $1.5232 \cdot 10^{-12}$ & $2.5288$\\ \hline
    \end{tabular} %\linebreak
   \begin{tabular}{c||c|c||c|c||c|c||c|c||}
         & \multicolumn{2}{c||}{\ssmall NSSPRK(6,4)} & \multicolumn{2}{c||}{\ssmall NSSPRK(10,4)} & \multicolumn{2}{c||}{ \ssmall NGLp4q3s3k3} & \multicolumn{2}{c}{\ssmall NGLp4q4s3k3} \\ \hline %\\[-1.3em]
        $\Delta t$ & \ssmall errors & \ssmall orders & \ssmall errors & \ssmall orders & \ssmall errors & \ssmall orders & \ssmall errors & \ssmall orders \\ \hline %\\[-1.3em]
        $0.01$ & $3.9944 \cdot 10^{-2}$ & - & $7.2636 \cdot 10^{-8}$ & - & $2.7117 \cdot 10^{-5}$ & - & $6.2326 \cdot 10^{-5}$ & - \\ \hline %\\[-1.3em]
        $0.01 \cdot 2^{-1}$ & $3.5526 \cdot 10^{-3}$ & $3.4910$ & $4.5395 \cdot 10^{-9}$ & $4.0001$ & $1.9945 \cdot 10^{-6}$ & $3.7651$& $4.5357 \cdot 10^{-6}$ & $3.7805$ \\ \hline %\\[-1.3em]
        $0.01 \cdot 2^{-2}$ & $2.4080 \cdot 10^{-4}$ & $3.8830$ & $2.8371 \cdot 10^{-10}$ & $4.0001$ & $1.3406 \cdot 10^{-7}$ & $3.8951$& $3.0362 \cdot 10^{-7}$ & $3.9010$\\ \hline %\\[-1.3em]
        $0.01 \cdot 2^{-3}$ & $1.5667 \cdot 10^{-5}$ & $3.9421$ & $1.7750 \cdot 10^{-11}$ & $3.9985$ & $8.6731 \cdot 10^{-9}$ & $3.9502$& $1.9607 \cdot 10^{-8}$ & $3.9529$\\ \hline %\\[-1.3em]
        $0.01 \cdot 2^{-4}$ & $9.9818 \cdot 10^{-7}$ & $3.9723$ & $1.1421 \cdot 10^{-12}$ & $3.9580$ & $5.5125 \cdot 10^{-10}$ & $3.9758$& $1.2457 \cdot 10^{-9}$ & $3.9763$\\ \hline %\\[-1.3em]
        $0.01 \cdot 2^{-5}$ & $6.2981 \cdot 10^{-8}$ & $3.9863$  & $1.1457 \cdot 10^{-13}$ & $3.3174$ & $3.4720 \cdot 10^{-11}$ & $3.9889$& $7.9655 \cdot 10^{-11}$ & $3.9671$\\ \hline %\\[-1.3em]
        $0.01 \cdot 2^{-6}$ & $3.9566 \cdot 10^{-9}$ & $3.9926$ & $8.0824 \cdot 10^{-14}$ & $5.0343$ & $2.1414 \cdot 10^{-12}$ & $4.0191$& $7.3648 \cdot 10^{-12}$ & $3.4350$\\ \hline %\\[-1.3em]
        $0.01 \cdot 2^{-7}$ & $2.5133 \cdot 10^{-10}$ & $3.9766$  & $2.5313 \cdot 10^{-13}$ & -& $8.1712 \cdot 10^{-14}$ & $4.7119$& $5.1816 \cdot 10^{-12}$ & $5.0723$ \\ \hline %\\[-1.3em]
        $0.01\cdot 2^{-8}$ & $2.2678 \cdot 10^{-11}$ & $3.4702$ & $4.7251 \cdot 10^{-13}$ & - & $1.8474 \cdot 10^{-13}$ & - & $9.7993 \cdot 10^{-12}$ & -\\ \hline
    \end{tabular}
    \caption{The errors and the corresponding orders of the different nonstandard methods while solving the logistic equation with $c=10$, plotted in the left panels of Figure \ref{fig:nonstiff_orders}.}
    \label{tab:nonstiff_errors}
\end{table}
\FloatBarrier

%\FloatBarrier
\begin{table}[!htbp]
    %\ssmall
    \fontsize{6pt}{6pt}
    \centering
    \renewcommand{\arraystretch}{0.5}
    \begin{tabular}{c||c|c||c|c||c|c}
         & \multicolumn{2}{c||}{ \ssmall NSSPMS(4,2)} & \multicolumn{2}{c||}{\ssmall NSSRK(2,2)} & \multicolumn{2}{c}{\ssmall NGLp2q2s3k3}  \\ \hline %\\[-1.3em]
        $\Delta t$ & \ssmall errors & \ssmall orders & \ssmall errors & \ssmall orders & \ssmall errors & \ssmall orders  \\ \hline %\\[-1.3em]
        $0.01$ & $1.3039 \cdot 10^{-1}$ & - & $5.4342 \cdot 10^{-2}$ & -& $3.3848 \cdot 10^{-2}$ & - \\ \hline %\\[-1.3em]
        $0.01 \cdot 2^{-1}$ & $4.0408 \cdot 10^{-2}$ & $1.6901$ & $1.2925 \cdot 10^{-2}$ & $2.0719$ & $1.0386 \cdot 10^{-2}$ & $1.7044$ \\ \hline %\\[-1.3em]
        $0.01 \cdot 2^{-2}$ & $1.1020 \cdot 10^{-2}$ & $1.8745$ & $3.1771 \cdot 10^{-3}$ & $2.0244$ & $2.8405 \cdot 10^{-3}$ & $1.8705$ \\ \hline %\\[-1.3em]
        $0.01 \cdot 2^{-3}$ & $2.8637 \cdot 10^{-3}$ & $1.9442$ & $7.8928 \cdot 10^{-4}$ & $2.0091$ & $7.4069 \cdot 10^{-4}$ & $1.9392$ \\ \hline %\\[-1.3em]
        $0.01 \cdot 2^{-4}$ & $7.2904 \cdot 10^{-4}$ & $1.9738$ & $1.9681 \cdot 10^{-4}$ & $2.0037$ & $1.8899 \cdot 10^{-4}$ & $1.9705$ \\ \hline %\\[-1.3em]
        $0.01 \cdot 2^{-5}$ & $1.8387 \cdot 10^{-4}$ & $1.9873$ & $4.9145 \cdot 10^{-5}$ & $2.0017$ & $4.7726 \cdot 10^{-5}$ & $1.9855$ \\ \hline %\\[-1.3em]
        $0.01 \cdot 2^{-6}$ & $4.6166 \cdot 10^{-5}$ & $1.9938$ & $1.2280 \cdot 10^{-5}$ & $2.0008$ & $1.1991 \cdot 10^{-5}$ & $1.9928$ \\ \hline %\\[-1.3em]
        $0.01 \cdot 2^{-7}$ & $1.1566 \cdot 10^{-5}$ & $1.9969$ & $3.0691 \cdot 10^{-6}$ & $2.0004$ & $3.0052 \cdot 10^{-6}$ & $1.9964$ \\ \hline %\\[-1.3em]
        $0.01 \cdot 2^{-8}$ & $2.8946 \cdot 10^{-6}$ & $1.9985$ & $7.6718 \cdot 10^{-7}$ & $2.0002$ & $7.5221 \cdot 10^{-7}$ & $1.9983$ \\ \hline \hline 
    \end{tabular}
    \begin{tabular}{c||c|c||c|c||c|c||c|c||}
         & \multicolumn{2}{c||}{\ssmall NSSPMS(4,3)} & \multicolumn{2}{c||}{\ssmall NSSPRK(3,3)} & \multicolumn{2}{c||}{\ssmall NGLp3q3s3k3} & \multicolumn{2}{c}{\ssmall NGLp3q2s3k2}  \\ \hline %\\[-1.3em]
        $\Delta t$ & \ssmall errors & \ssmall orders & \ssmall errors & \ssmall orders & \ssmall errors & \ssmall orders & \ssmall errors & \ssmall orders \\ \hline %\\[-1.3em]
        $0.01$ & $1.7509 \cdot 10^{-1}$ & - & $2.3249 \cdot 10^{-3}$ & -& $9.1030 \cdot 10^{-5}$ & - & $1.0607 \cdot 10^{-3}$ & - \\ \hline %\\[-1.3em]
        $0.01 \cdot 2^{-1}$ & $1.2683 \cdot 10^{-2}$ & $3.7871$ & $1.3848 \cdot 10^{-4}$ & $4.0694$ & $1.0622 \cdot 10^{-5}$ & $3.0993$ & $1.2732 \cdot 10^{-4}$ & $3.0585$ \\ \hline %\\[-1.3em]
        $0.01 \cdot 2^{-2}$ & $9.2252 \cdot 10^{-4}$ & $3.7812$ & $7.8010 \cdot 10^{-6}$ & $4.1498$ & $1.2453 \cdot 10^{-6}$ & $3.0924$ & $1.5118 \cdot 10^{-5}$ & $3.0741$ \\ \hline %\\[-1.3em]
        $0.01 \cdot 2^{-3}$ & $7.0355 \cdot 10^{-5}$ & $3.7129$ & $3.8085 \cdot 10^{-7}$ & $4.3564$ & $1.4943 \cdot 10^{-7}$ & $3.0589$ & $1.8251 \cdot 10^{-6}$ & $3.0502$\\ \hline %\\[-1.3em]
        $0.01 \cdot 2^{-4}$ & $5.8752 \cdot 10^{-6}$ & $3.5820$ & $1.0467 \cdot 10^{-8}$ & $5.1853$ & $1.8256 \cdot 10^{-8}$ & $3.0331$ & $2.2364 \cdot 10^{-7}$ & $3.0287$\\ \hline %\\[-1.3em]
        $0.01 \cdot 2^{-5}$ & $5.4853 \cdot 10^{-7}$ & $3.4210$ & $1.0090 \cdot 10^{-9}$ & $3.3749$ & $2.2545 \cdot 10^{-9}$ & $3.0175$ & $2.7662 \cdot 10^{-8}$ & $3.0152$\\ \hline %\\[-1.3em]
        $0.01 \cdot 2^{-6}$ & $5.6847 \cdot 10^{-8}$ & $3.2704$ & $2.6370 \cdot 10^{-10}$ & $1.9359$ & $2.8017 \cdot 10^{-10}$ & $3.0084$ & $3.4440\cdot 10^{-9}$ & $3.0058$\\ \hline %\\[-1.3em]
        $0.01 \cdot 2^{-7}$ & $6.3780 \cdot 10^{-9}$ & $3.1559$ & $2.6034 \cdot 10^{-11}$ & $3.3404$ & $3.5137 \cdot 10^{-11}$ & $2.9952$ & $4.4059 \cdot 10^{-10}$ & $2.9666$\\ \hline %\\[-1.3em]
        $0.01 \cdot 2^{-8}$ & $7.6687 \cdot 10^{-10}$ & $3.0560$ & $2.9274 \cdot 10^{-11}$ & - & $4.8379 \cdot 10^{-12}$ & $2.8605$ & $7.6170 \cdot 10^{-11}$ & $2.5321$\\ \hline
    \end{tabular} %\linebreak
   \begin{tabular}{c||c|c||c|c||c|c||c|c||}
         & \multicolumn{2}{c||}{ \ssmall NSSPRK(6,4)} & \multicolumn{2}{c||}{\ssmall NSSPRK(10,4)} & \multicolumn{2}{c||}{\ssmall NGLp4q3s3k3} & \multicolumn{2}{c}{\ssmall NGLp4q4s3k3} \\ \hline %\\[-1.3em]
        $\Delta t$ & \ssmall errors & \ssmall orders & \ssmall errors & \ssmall orders & \ssmall errors & \ssmall orders & \ssmall errors & \ssmall orders \\ \hline %\\[-1.3em]
        $0.01$ & $1.9972 \cdot 10^{0}$ & - & $3.6318 \cdot 10^{-6}$ & - & $1.3558 \cdot 10^{-3}$ & - & $3.1163 \cdot 10^{-3}$ & - \\ \hline %\\[-1.3em]
        $0.01 \cdot 2^{-1}$ & $1.7763 \cdot 10^{-1}$ & $3.4910$ & $2.2698 \cdot 10^{-7}$ & $4.0001$ & $9.9726 \cdot 10^{-5}$ & $3.7651$& $2.2678 \cdot 10^{-4}$ & $3.7805$ \\ \hline %\\[-1.3em]
        $0.01 \cdot 2^{-2}$ & $1.2040 \cdot 10^{-2}$ & $3.8830$ & $1.4186 \cdot 10^{-8}$ & $4.0000$ & $6.7031 \cdot 10^{-6}$ & $3.8951$& $1.5181 \cdot 10^{-5}$ & $3.9010$\\ \hline %\\[-1.3em]
        $0.01 \cdot 2^{-3}$ & $7.8333 \cdot 10^{-4}$ & $3.9421$ & $8.8710 \cdot 10^{-10}$ & $3.9992$ & $4.3365 \cdot 10^{-7}$ & $3.9502$& $9.8034 \cdot 10^{-7}$ & $3.9529$\\ \hline %\\[-1.3em]
        $0.01 \cdot 2^{-4}$ & $4.9909 \cdot 10^{-5}$ & $3.9723$ & $5.6502 \cdot 10^{-11}$ & $3.9727$ & $2.7563 \cdot 10^{-8}$ & $3.9757$& $6.2285 \cdot 10^{-8}$ & $3.9763$\\ \hline %\\[-1.3em]
        $0.01 \cdot 2^{-5}$ & $3.1490 \cdot 10^{-6}$ & $3.9863$ & $6.2527 \cdot 10^{-12}$ & $3.1757$ & $1.7360 \cdot 10^{-9}$ & $3.9889$& $3.9822 \cdot 10^{-9}$ & $3.9672$\\ \hline %\\[-1.3em]
        $0.01 \cdot 2^{-6}$ & $1.9783 \cdot 10^{-7}$ & $3.9926$ & $4.1495 \cdot 10^{-12}$ & $5.9153$ & $1.0749 \cdot 10^{-10}$ & $4.0135$& $3.6761 \cdot 10^{-10}$ & $3.4373$\\ \hline %\\[-1.3em]
        $0.01 \cdot 2^{-7}$ & $1.2567 \cdot 10^{-8}$ & $3.9765$ & $1.2449 \cdot 10^{-11}$ & -& $1.3642 \cdot 10^{-12}$ & $6.3000$& $2.5824 \cdot 10^{-10}$ & $5.0945$ \\ \hline %\\[-1.3em]
        $0.01\cdot 2^{-8}$ & $1.1342 \cdot 10^{-9}$ & $3.4699$ & $2.1203 \cdot 10^{-11}$ & - & $9.2654 \cdot 10^{-12}$ & - & $4.9062 \cdot 10^{-10}$ & -\\ \hline
    \end{tabular}
    \caption{The errors and the corresponding orders of the different nonstandard methods while solving the logistic equation with $c=500$, plotted in the right panels of Figure \ref{fig:nonstiff_orders}.}
    \label{tab:stiff_errors}
\end{table}
%\FloatBarrier

%\enlargethispage{3cm}
%\FloatBarrier
\begin{table}[!htbp]
    %\ssmall
    \fontsize{6pt}{6pt}
    \centering
    \renewcommand{\arraystretch}{0.5}
    \begin{tabular}{c||c|c||c|c||c|c}
         & \multicolumn{2}{c||}{ \ssmall NSSPMS(4,2)} & \multicolumn{2}{c||}{\ssmall NSSRK(2,2)} & \multicolumn{2}{c}{\ssmall NGLp2q2s3k3}  \\ \hline %\\[-1.3em]
        $\Delta t$ & \ssmall errors & \ssmall orders & \ssmall errors & \ssmall orders & \ssmall errors & \ssmall orders  \\ \hline %\\[-1.3em]
        $0.05$ & $2.4440 \cdot 10^{-3}$ & - & $4.2992 \cdot 10^{-4}$ & -& $1.2024 \cdot 10^{-4}$ & - \\ \hline %\\[-1.3em]
        $0.05 \cdot 2^{-1}$ & $2.6209 \cdot 10^{-4}$ & $3.2211$ & $4.8459 \cdot 10^{-5}$ & $3.1492$ & $2.8221 \cdot 10^{-5}$ & $2.0910$ \\ \hline %\\[-1.3em]
        $0.05 \cdot 2^{-2}$ & $3.6849 \cdot 10^{-5}$ & $2.8303$ & $1.1345 \cdot 10^{-5}$ & $2.0947$ & $7.0190 \cdot 10^{-6}$ & $2.0075$ \\ \hline %\\[-1.3em]
        $0.05 \cdot 2^{-3}$ & $6.9473 \cdot 10^{-6}$ & $2.4071$ & $2.7744 \cdot 10^{-6}$ & $2.0318$ & $1.7588 \cdot 10^{-6}$ & $1.9967$ \\ \hline %\\[-1.3em]
        $0.05 \cdot 2^{-4}$ & $1.7148 \cdot 10^{-6}$ & $2.0184$ & $6.8803 \cdot 10^{-7}$ & $2.0116$ & $4.4062 \cdot 10^{-7}$ & $1.9970$ \\ \hline %\\[-1.3em]
        $0.05 \cdot 2^{-5}$ & $4.2660 \cdot 10^{-7}$ & $2.0071$ & $1.7145 \cdot 10^{-7}$ & $2.0047$ & $1.1029 \cdot 10^{-7}$ & $1.9981$ \\ \hline %\\[-1.3em]
        $0.05 \cdot 2^{-6}$ & $1.0643 \cdot 10^{-7}$ & $2.0030$ & $4.2800 \cdot 10^{-8}$ & $2.0021$ & $2.7592 \cdot 10^{-8}$ & $1.9990$ \\ \hline %\\[-1.3em]
        $0.05 \cdot 2^{-7}$ & $2.6581 \cdot 10^{-8}$ & $2.0014$ & $1.0693 \cdot 10^{-8}$ & $2.0010$ & $6.9006 \cdot 10^{-9}$ & $1.9995$ \\ \hline %\\[-1.3em]
        $0.05 \cdot 2^{-8}$ & $6.6425 \cdot 10^{-9}$ & $2.0006$ & $2.6723 \cdot 10^{-9}$ & $2.0005$ & $1.7255 \cdot 10^{-9}$ & $1.9997$ \\ \hline \hline 
    \end{tabular}
    \begin{tabular}{c||c|c||c|c||c|c||c|c||}
         & \multicolumn{2}{c||}{\ssmall NSSPMS(4,3)} & \multicolumn{2}{c||}{\ssmall NSSPRK(3,3)} & \multicolumn{2}{c||}{\ssmall NGLp3q3s3k3} & \multicolumn{2}{c}{\ssmall NGLp3q2s3k2}  \\ \hline %\\[-1.3em]
        $\Delta t$ & \ssmall errors & \ssmall orders & \ssmall errors & \ssmall orders & \ssmall errors & \ssmall orders & \ssmall errors & \ssmall orders \\ \hline %\\[-1.3em]
        $0.05$ & $2.0610 \cdot 10^{-2}$ & - & $3.1524 \cdot 10^{-4}$ & -& $2.0960 \cdot 10^{-4}$ & - & $7.2711 \cdot 10^{-5}$ & - \\ \hline %\\[-1.3em]
        $0.05 \cdot 2^{-1}$ & $1.5549 \cdot 10^{-3}$ & $3.7284$ & $1.9170 \cdot 10^{-5}$ & $4.0395$ & $1.0690 \cdot 10^{-5}$ & $4.2933$ & $2.8249 \cdot 10^{-6}$ & $4.6859$ \\ \hline %\\[-1.3em]
        $0.05 \cdot 2^{-2}$ & $9.9708 \cdot 10^{-5}$ & $3.9630$ & $1.1269 \cdot 10^{-6}$ & $4.0885$ & $4.4308 \cdot 10^{-7}$ & $4.5925$ & $9.8625 \cdot 10^{-8}$ & $4.8401$ \\ \hline %\\[-1.3em]
        $0.05 \cdot 2^{-3}$ & $6.2039 \cdot 10^{-6}$ & $4.0065$ & $6.1496 \cdot 10^{-8}$ & $4.1957$ & $3.1221 \cdot 10^{-8}$ & $3.8270$ & $9.9297 \cdot 10^{-9}$ & $3.3121$\\ \hline %\\[-1.3em]
        $0.05 \cdot 2^{-4}$ & $3.7664 \cdot 10^{-7}$ & $4.0419$ & $2.7269 \cdot 10^{-9}$ & $4.4952$ & $4.7503 \cdot 10^{-9}$ & $2.7164$ & $1.4435 \cdot 10^{-9}$ & $2.7822$\\ \hline %\\[-1.3em]
        $0.05 \cdot 2^{-5}$ & $2.1926 \cdot 10^{-8}$ & $4.1025$ & $1.9683 \cdot 10^{-10}$ & $3.7923$ & $8.1874 \cdot 10^{-10}$ & $2.5365$ & $2.2579 \cdot 10^{-10}$ & $2.6765$\\ \hline %\\[-1.3em]
        $0.05 \cdot 2^{-6}$ & $1.1617 \cdot 10^{-9}$ & $4.2384$ & $2.7238 \cdot 10^{-11}$ & $2.8532$ & $1.1638 \cdot 10^{-10}$ & $2.8146$ & $3.0952 \cdot 10^{-11}$ & $2.8669$\\ \hline %\\[-1.3em]
        $0.05 \cdot 2^{-7}$ & $5.0217 \cdot 10^{-11}$ & $4.5318$ & $3.4914 \cdot 10^{-12}$ & $2.9637$ & $1.5425 \cdot 10^{-11}$ & $2.9154$ & $4.0390 \cdot 10^{-12}$ & $2.9379$\\ \hline %\\[-1.3em]
        $0.05 \cdot 2^{-8}$ & $2.9863 \cdot 10^{-12}$ & $4.0717$ & $5.0260 \cdot 10^{-13}$ & $2.7963$ & $1.9964 \cdot 10^{-12}$ & $2.9498$ & $5.2758 \cdot 10^{-13}$ & $2.9366$\\ \hline
    \end{tabular} %\linebreak
   \begin{tabular}{c||c|c||c|c||c|c||c|c||}
         & \multicolumn{2}{c||}{ \ssmall NSSPRK(6,4)} & \multicolumn{2}{c||}{\ssmall NSSPRK(10,4)} & \multicolumn{2}{c||}{\ssmall NGLp4q3s3k3} & \multicolumn{2}{c}{\ssmall NGLp4q4s3k3} \\ \hline %\\[-1.3em]
        $\Delta t$ & \ssmall errors & \ssmall orders & \ssmall errors & \ssmall orders & \ssmall errors & \ssmall orders & \ssmall errors & \ssmall orders \\ \hline %\\[-1.3em]
        $0.05$ & $1.1739 \cdot 10^{-1}$ & - & $2.4392 \cdot 10^{-7}$ & - & $4.7556 \cdot 10^{-4}$ & - & $2.8627 \cdot 10^{-4}$ & - \\ \hline %\\[-1.3em]
        $0.05 \cdot 2^{-1}$ & $2.1800 \cdot 10^{-2}$ & $2.4289$ & $1.5246 \cdot 10^{-8}$ & $3.9999$ & $3.1925 \cdot 10^{-5}$ & $3.8969$& $1.6741 \cdot 10^{-5}$ & $4.0959$ \\ \hline %\\[-1.3em]
        $0.05 \cdot 2^{-2}$ & $1.6444 \cdot 10^{-3}$ & $3.7287$ & $9.5298 \cdot 10^{-10}$ & $3.9999$ & $2.0617 \cdot 10^{-6}$ & $3.9528$& $1.0093 \cdot 10^{-6}$ & $4.0520$\\ \hline %\\[-1.3em]
        $0.05 \cdot 2^{-3}$ & $1.0584 \cdot 10^{-4}$ & $3.9576$ & $5.9574 \cdot 10^{-11}$ & $3.9997$ & $1.3092 \cdot 10^{-7}$ & $3.9770$& $6.1920 \cdot 10^{-8}$ & $4.0268$\\ \hline %\\[-1.3em]
        $0.05 \cdot 2^{-4}$ & $6.6819 \cdot 10^{-6}$ & $3.9855$ & $3.7333 \cdot 10^{-12}$ & $3.9962$ & $8.2472 \cdot 10^{-9}$ & $3.9886$& $3.8337 \cdot 10^{-9}$ & $4.0136$\\ \hline %\\[-1.3em]
        $0.05 \cdot 2^{-5}$ & $4.1959 \cdot 10^{-7}$ & $3.9932$ & $2.4303 \cdot 10^{-13}$ & $3.9413$ & $5.1742 \cdot 10^{-10}$ & $3.9945$& $2.3848 \cdot 10^{-10}$ & $4.0068$\\ \hline %\\[-1.3em]
        $0.05 \cdot 2^{-6}$ & $2.6286 \cdot 10^{-8}$ & $3.9966$ & $5.3291 \cdot 10^{-14}$ & $2.1892$ & $3.2283 \cdot 10^{-11}$ & $4.0025$& $1.4883 \cdot 10^{-11}$ & $4.0022$\\ \hline %\\[-1.3em]
        $0.05 \cdot 2^{-7}$ & $1.6446 \cdot 10^{-9}$ & $3.9985$ & $6.7391 \cdot 10^{-14}$ & -& $1.7726 \cdot 10^{-12}$ & $4.1868$& $9.4275 \cdot 10^{-13}$ & $3.9806$ \\ \hline %\\[-1.3em]
        $0.05\cdot 2^{-8}$ & $1.0248 \cdot 10^{-10}$ & $4.0043$ & $1.0170 \cdot 10^{-13}$ & - & $1.5271 \cdot 10^{-12}$ & - & $8.1379 \cdot 10^{-14}$ & $3.5341$\\ \hline
    \end{tabular}
    \caption{The errors and the corresponding orders of the different nonstandard methods while solving the SEIR system with $\Pi=0$, plotted in Figure \ref{fig:seir_order}.}
    \label{tab:seir_errors}
\end{table}
%\FloatBarrier




%\end{appendices}



\begin{thebibliography}{58}
\expandafter\ifx\csname natexlab\endcsname\relax\def\natexlab#1{#1}\fi
\providecommand{\bibinfo}[2]{#2}
\ifx\xfnm\relax \def\xfnm[#1]{\unskip,\space#1}\fi

%Type = Article
\bibitem[{Hadjimichael, et al. (2016)}]{hadjimichael}
\bibinfo{author}{Y. Hadjimichael}, \bibinfo{author}{D. Ketchenson}, \bibinfo{author}{L. Lóczi}, \bibinfo{author}{A. Németh},
\newblock \bibinfo{title}{Strong stability preserving explicit linear multistep methods with variable step size},
\newblock \bibinfo{journal}{SIAM J. Num. Anal.,} \bibinfo{volume}{54(5)} (\bibinfo{year}{2016}) \bibinfo{pages}{2799--2832}. \url{https://doi.org/10.1137/15M101717X}

%Type = Article
\bibitem[{Nüßlein, et al. (2021)}]{nusslein}
\bibinfo{author}{S. Nüßlein}, \bibinfo{author}{H. Ranocha}, \bibinfo{author}{D. Ketcheson},
\newblock \bibinfo{title}{Positivity-preserving adaptive Runge–Kutta methods},
\newblock \bibinfo{journal}{Com. Appl. Math. Comp. Sci.}, \bibinfo{volume}{16(2)} (\bibinfo{year}{2021}) \bibinfo{pages}{155--179}. \url{https://doi.org/10.2140/camcos.2021.16.155}

%\FloatBarrier



%\FloatBarrier

%Type = Article
\bibitem[{Arévalo, et al. (2020)}]{arevalo}
\bibinfo{author}{C. Arévalo}, \bibinfo{author}{G. Söderlind}, \bibinfo{author}{Y. Hadjimichael}, \bibinfo{author}{I. Fekete},
\newblock \bibinfo{title}{Local error estimation and step size control in adaptive linear multistep methods},
\newblock \bibinfo{journal}{Num. Alg.} \bibinfo{volume}{86(2)} (\bibinfo{year}{2021}) \bibinfo{pages}{537--563}. \url{https://doi.org/10.1007/s11075-020-00900-1}

%Type = Article
\bibitem[{Charous and Lermusiaux (2024)}]{charous}
\bibinfo{author}{A. Charous}, \bibinfo{author}{P. Lermusiaux}, 
\newblock \bibinfo{title}{Stable rank-adaptive dynamically orthogonal Runge–Kutta schemes},
\newblock \bibinfo{journal}{SIAM J. Sci. Comp.} \bibinfo{volume}{46(1)} (\bibinfo{year}{2025}) \bibinfo{pages}{A529--A560}. \url{https://doi.org/10.1137/22M1534948}

\needspace{4\baselineskip}
%Type = Article
\bibitem[{Morani, et al. (2025)}]{morani}
\bibinfo{author}{A. H. Morani}, \bibinfo{author}{M. M. Saeed}, \bibinfo{author}{M. Aslam}, \bibinfo{author}{A. Mehmoud}, \bibinfo{author}{A. Shokri}, \bibinfo{author}{H. Mukalazi},
\newblock \bibinfo{title}{Local and global stability analysis of HIV/AIDS by using a nonstandard finite difference scheme},
\newblock \bibinfo{journal}{Sci. Rep.} \bibinfo{volume}{15(1)} (\bibinfo{year}{2025}) \bibinfo{pages}{4502}. \nopagebreak \linebreak \url{https://doi.org/10.1038/s41598-024-82872-z}


%Type = Article
\bibitem[{Zinihi, et al. (2025)}]{zinihi}
\bibinfo{author}{A. Zinihi}, \bibinfo{author}{M. Ehrhardt}, \bibinfo{author}{M. R. S. Ammi}, 
\newblock \bibinfo{title}{A nonstandard finite difference scheme for an SEIQR epidemiological PDE model},
\newblock \bibinfo{journal}{Appl. Math. Comp.} \bibinfo{volume}{520} (\bibinfo{year}{2026}) \bibinfo{pages}{129953}. \url{https://doi.org/10.1016/j.amc.2026.129953}

%Type = Article
\bibitem[{Marime, et al. (2026)}]{marime}
\bibinfo{author}{C. B. Marime}, \bibinfo{author}{J. B. Munyakazi},
\newblock \bibinfo{title}{A second-order nonstandard finite difference method for a malaria propagation model with control.},
\newblock \bibinfo{journal}{AppliedMath} \bibinfo{volume}{6(3)} (\bibinfo{year}{2026}) \bibinfo{pages}{36}. \url{https://doi.org/10.3390/appliedmath6030036}

%Type = Article
\bibitem[{Tassé, et al. (2025)}]{tasse}
\bibinfo{author}{A. J. O. Tassé}, \bibinfo{author}{V. B. Kubalasa}, \bibinfo{author}{B. Tsanou},
\newblock \bibinfo{title}{Nonstandard finite difference schemes for some epidemic optimal control problems},
\newblock \bibinfo{journal}{Math. Comp. Sim.} \bibinfo{volume}{228} (\bibinfo{year}{2025}) \bibinfo{pages}{1--22}. \url{https://doi.org/10.1016/j.matcom.2024.08.028}

%%%%%%%%%%%%%%%%%%%%%%%%%%%%%%%%%%%%%%%%%%%%%%%%%%%%%%%%%%%%%%%%%%%%%%%%%%%%%%%%%%%%

%\FloatBarrier

%Type = Article
\bibitem[{Wacker (2025)}]{wacker}
\bibinfo{author}{B. Wacker}, 
\newblock \bibinfo{title}{Qualitative Study of a Dynamical System for Computer Virus Propagation—A Nonstandard Finite-Difference-Methodological View},
\newblock \bibinfo{journal}{Math. Meth. Appl. Sci.} \bibinfo{volume}{48(8)} (\bibinfo{year}{2025}) \bibinfo{pages}{9272--9291}. \url{https://doi.org/10.1002/mma.10798}

%Type = Article
\bibitem[{Hoang (2026)}]{hoang26}
\bibinfo{author}{M. T. Hoang}, 
\newblock \bibinfo{title}{Mathematical analysis and numerical simulation of a generalized epidemiological model for malware propagation},
\newblock \bibinfo{journal}{Nonlin. Dyn.} \bibinfo{volume}{114(1)} (\bibinfo{year}{2026}) \bibinfo{pages}{53}. \url{https://doi.org/10.1007/s11071-025-11912-8}


%%%%%%%%%%%%%%%%%%%%%%%%%%%%%%%%%%%%%%%%%%%%%%%%%%%%%%%%%%%%%%%%%%%%%%%%%%%%%%%%%%%%

%Type = Article
\bibitem[{Faheem and Ghosh (2025)}]{faheem}
\bibinfo{author}{M. Faheem}, \bibinfo{author}{B. Ghosh}, 
\newblock \bibinfo{title}{Dynamics of a delayed discrete-time predator prey model proposed from a nonstandard finite difference scheme},
\newblock \bibinfo{journal}{J. Comp. Appl. Math.} \bibinfo{volume}{458} (\bibinfo{year}{2025}) \bibinfo{pages}{116346}. \url{https://doi.org/10.1016/j.cam.2024.116346}

%Type = Article
\bibitem[{Eskandari et al (2025)}]{eskandari}
\bibinfo{author}{Z. Eskandari}, \bibinfo{author}{Z. Avazzadeh}, \bibinfo{author}{R. Khoshsiar Ghaziani}, \bibinfo{author}{B. Li},
\newblock \bibinfo{title}{Dynamics and bifurcations of a discrete‐time Lotka–Volterra model using nonstandard finite difference discretization method},
\newblock \bibinfo{journal}{Math. Meth. Appl. Sci.} \bibinfo{volume}{48(7)} (\bibinfo{year}{2025}) \bibinfo{pages}{7197--7212}. \url{https://doi.org/10.1002/mma.8859}

%Type = Article
\bibitem[{Özdoğan and Arslan (2026)}]{ozdogan}
\bibinfo{author}{N. Özdoğan}, \bibinfo{author}{B. Arslan}, 
\newblock \bibinfo{title}{ Nonstandard Finite Difference Theta Approaches to the Predator–Prey System},
\newblock \bibinfo{journal}{Math. Meth. Appl. Sci.} \bibinfo{volume}{49(8)} (\bibinfo{year}{2026}) \bibinfo{pages}{8548--8561}. \url{https://doi.org/10.1002/mma.70486}

%%%%%%%%%%%%%%%%%%%%%%%%%%%%%%%%%%%%%%%%%%%%%%%%%%%%%%%%%%%%%%%%%%%%%%%%%%%%%%%%%%%%

%Type = Article
\bibitem[{Mohye, et al (2025)}]{mohye}
\bibinfo{author}{M. A. Mohye}, \bibinfo{author}{J. B. Munyakazi}, \bibinfo{author}{T. G. Dinka}, \bibinfo{author}{Y. H. Haji}, \bibinfo{author}{A. N. Ware}, \bibinfo{author}{J. M. Ahmed}, 
\newblock \bibinfo{title}{A new parameter-convergent nonstandard finite difference method for two-parameter singularly perturbed problems},
\newblock \bibinfo{journal}{Disc. Appl. Sci.} \bibinfo{volume}{7(11)} (\bibinfo{year}{2025}) \bibinfo{pages}{1248}. \url{https://doi.org/10.1007/s42452-025-07721-8}


%Type = Article
\bibitem[{Fazayel, et al (2026)}]{fazayel}
\bibinfo{author}{M. Fazayel}, \bibinfo{author}{F. Fakhar-Izadi}, \bibinfo{author}{M. Dehghan}, \bibinfo{author}{M. Abbaszadeh}, 
\newblock \bibinfo{title}{Numerical simulation of Klein–Gordon–Zakharov equations using conservative nonstandard finite difference method combined with scalar auxiliary variable scheme},
\newblock \bibinfo{journal}{Comp. Appl. Math.} \bibinfo{volume}{45(3)} (\bibinfo{year}{2026}) \bibinfo{pages}{110}. \url{https://doi.org/10.1007/s40314-025-03518-y}

%Type = Article
\bibitem[{Rehman, et al (2025)}]{rehman}
\bibinfo{author}{B. Rehman}, \bibinfo{author}{M. A. B. Iqbal}, \bibinfo{author}{A. Khan}, \bibinfo{author}{D. K. Almutairi}, \bibinfo{author}{T. Abdeljawad}, 
\newblock \bibinfo{title}{Nonstandard Finite Difference Predictor Corrector Method for Quadratic Riccati Differential Equation},
\newblock \bibinfo{journal}{Eur. J. Pure Appl. Math.} \bibinfo{volume}{18(2)} (\bibinfo{year}{2025}) \bibinfo{pages}{5703}. \url{https://doi.org/10.29020/nybg.ejpam.v18i2.5703}

%%%%%%%%%%%%%%%%%%%%%%%%%%%%%%%%%%%%%%%%%%%%%%%%%%%%%%%%%%%%%%%%%%%%%%%%%%%%%%%%%%%%

%Type = Book
\bibitem[{Mickens(1993)}]{mickens}
\bibinfo{author}{R.~E. Mickens}, \bibinfo{title}{Nonstandard finite difference models of differential equations}, \bibinfo{publisher}{World Scientific}, \bibinfo{address}{Singapore}, \bibinfo{year}{1993}. \url{https://doi.org/10.1142/2081}

%Type = Article
\bibitem[{Dimitrov and Kojouharov(2005)}]{dimitrov}
\bibinfo{author}{D.~T. Dimitrov}, \bibinfo{author}{H.~V. Kojouharov},
\newblock \bibinfo{title}{Nonstandard finite-difference schemes for general two-dimensional autonomous dynamical systems},
\newblock \bibinfo{journal}{Appl. Math. Lett.} \bibinfo{volume}{18} (\bibinfo{year}{2005}) \bibinfo{pages}{769--774}. \url{https://doi.org/10.1016/j.aml.2004.08.011}

%Type = Article
\bibitem[{Dimitrov and Kojouharov(2006)}]{dimitrov2}
\bibinfo{author}{D.~T. Dimitrov}, \bibinfo{author}{H.~V. Kojouharov},
\newblock \bibinfo{title}{Positive and elementary stable nonstandard numerical methods with applications to predator--prey models},
\newblock \bibinfo{journal}{J. Comp. Appl. Math.} \bibinfo{volume}{189} (\bibinfo{year}{2006}) \bibinfo{pages}{98--108}. \url{https://doi.org/10.1016/j.cam.2005.04.003}

%Type = Article
\bibitem[{Gupta et~al.(2021)Gupta, Slezak, Alalhareth, Roy, and Kojouharov}]{gupta2}
\bibinfo{author}{M.~Gupta}, \bibinfo{author}{J.~M. Slezak}, \bibinfo{author}{F.~K. Alalhareth}, \bibinfo{author}{S.~Roy}, \bibinfo{author}{H.~V. Kojouharov},
\newblock \bibinfo{title}{Second-order modified nonstandard Runge-Kutta and theta methods for one-dimensional autonomous differential equations},
\newblock \bibinfo{journal}{Applic. Appl. Math. Intern. J. (AAM)} \bibinfo{volume}{16} (\bibinfo{year}{2021}) \bibinfo{pages}{1}. \bibinfo{note}{https://digitalcommons.pvamu.edu/aam/vol16/iss2/1}. (Accessed 8 July 2026)

%Type = Article
\bibitem[{Kojouharov et~al.(2021)Kojouharov, Roy, Gupta, Alalhareth, and Slezak}]{kojouharov}
\bibinfo{author}{H.~V. Kojouharov}, \bibinfo{author}{S.~Roy}, \bibinfo{author}{M.~Gupta}, \bibinfo{author}{F.~Alalhareth}, \bibinfo{author}{J.~M. Slezak},
\newblock \bibinfo{title}{A second-order modified nonstandard theta method for one-dimensional autonomous differential equations},
\newblock \bibinfo{journal}{Appl. Math. Lett.} \bibinfo{volume}{112} (\bibinfo{year}{2021}) \bibinfo{pages}{106775}. \url{https://doi.org/10.1016/j.aml.2020.106775}

%Type = Article
\bibitem[{Dimitrov and Kojouharov(2007)}]{dimitrov3}
\bibinfo{author}{D.~T. Dimitrov}, \bibinfo{author}{H.~V. Kojouharov},
\newblock \bibinfo{title}{Stability-preserving finite-difference methods for general multi-dimensional autonomous dynamical systems},
\newblock \bibinfo{journal}{Int. J. Numer. Anal. Model} \bibinfo{volume}{4} (\bibinfo{year}{2007}) \bibinfo{pages}{282--292}. \bibinfo{note}{https://www.math.ualberta.ca/ijnam/Volume-4-2007/No-2-07/2007-02-06.pdf}. \rev{(Accessed 8 July 2026)}

%Type = Inproceedings
\bibitem[{Anguelov et~al.(2003)Anguelov, Kama, and Lubuma}]{anguelov03}
\bibinfo{author}{R.~Anguelov}, \bibinfo{author}{P.~Kama}, \bibinfo{author}{J.~Lubuma},
\newblock \bibinfo{title}{Nonstandard theta method and related discrete schemes for the reaction--diffusion equation},
\newblock in: \bibinfo{booktitle}{Proceedings of the International Conference of Computational Methods in Sciences and Engineering, World Scientific, Singapore}, \rev{Vol.} \bibinfo{volume}{1}, \rev{\bibinfo{year}{2003}}, pp. \bibinfo{pages}{24--27}. \url{https://doi.org/10.1142/9789812704658_0008}

%Type = Article
\bibitem[{Anguelov et~al.(2005)Anguelov, Kama, and Lubuma}]{anguelov2}
\bibinfo{author}{R.~Anguelov}, \bibinfo{author}{P.~Kama}, \bibinfo{author}{J.-S. Lubuma},
\newblock \bibinfo{title}{On non-standard finite difference models of reaction--diffusion equations},
\newblock \bibinfo{journal}{J. Comp. Appl. Math.} \bibinfo{volume}{175} (\bibinfo{year}{2005}) \bibinfo{pages}{11--29}. \url{https://doi.org/10.1016/j.cam.2004.06.002}

%Type = Incollection
\bibitem[{Lubuma and Patidar(2005)}]{lubuma}
\bibinfo{author}{J.~M.-S. Lubuma}, \bibinfo{author}{K.~C. Patidar},
\newblock \bibinfo{title}{Contributions to the theory of non-standard finite difference methods and applications to singular perturbation problems},
\newblock in: \bibinfo{booktitle}{Advances in the Applications of Nonstandard Finite Difference Schemes}, \bibinfo{publisher}{World Scientific}, \bibinfo{address}{Singapore}, \bibinfo{year}{2005}, pp. \bibinfo{pages}{513--560}. \url{https://doi.org/10.1142/9789812703316_0012}

%Type = Inproceedings
\bibitem[{Wacker(2024)}]{wacker2}
\bibinfo{author}{B. Wacker},
\newblock \bibinfo{title}{Construction of High-Order Non-Standard Finite-Difference-Methods for Epidemiological Models},
\newblock in: \bibinfo{booktitle}{ International Conference on Mathematical Modeling in Physical Sciences}, \bibinfo{pages}{693--706}, \bibinfo{year}{2024}. \url{https://doi.org/10.1007/978-3-032-00914-2_46}

%Type = Article
\bibitem[{Bassenne et~al.(2021)Bassenne, Fu, and Mani}]{bassenne}
\bibinfo{author}{M.~Bassenne}, \bibinfo{author}{L.~Fu}, \bibinfo{author}{A.~Mani},
\newblock \bibinfo{title}{Time-accurate and highly-stable explicit operators for stiff differential equations},
\newblock \bibinfo{journal}{J. Comp. Phys.} \bibinfo{volume}{424} (\bibinfo{year}{2021}) \bibinfo{pages}{109847}. \url{https://doi.org/10.1016/j.jcp.2020.109847}

%Type = Article
\bibitem[{Gonz{\'a}lez-Parra et~al.(2010)Gonz{\'a}lez-Parra, Arenas, and Chen-Charpentier}]{gonzalez}
\bibinfo{author}{G.~Gonz{\'a}lez-Parra}, \bibinfo{author}{A.~J. Arenas}, \bibinfo{author}{B.~M. Chen-Charpentier},
\newblock \bibinfo{title}{Combination of nonstandard schemes and richardson’s extrapolation to improve the numerical solution of population models},
\newblock \bibinfo{journal}{Math. Comp. Model.} \bibinfo{volume}{52} (\bibinfo{year}{2010}) \bibinfo{pages}{1030--1036}. \url{https://doi.org/10.1016/j.mcm.2010.03.015}

%Type = Article
\bibitem[{Alalhareth et~al.(2024)Alalhareth, Gupta, Kojouharov, and Roy}]{alal}
\bibinfo{author}{F.~K. Alalhareth}, \bibinfo{author}{M.~Gupta}, \bibinfo{author}{H.~V. Kojouharov}, \bibinfo{author}{S.~Roy},
\newblock \bibinfo{title}{Second-order modified nonstandard explicit euler and explicit runge--kutta methods for n-dimensional autonomous differential equations},
\newblock \bibinfo{journal}{Comp.} \bibinfo{volume}{12} (\bibinfo{year}{2024}) \bibinfo{pages}{183}. \url{https://doi.org/10.3390/computation12090183}

%Type = Article
\bibitem[{Alalhareth et~al.(2023)Alalhareth, Gupta, Roy, and Kojouharov}]{alal23}
\bibinfo{author}{F.~K. Alalhareth}, \bibinfo{author}{M.~Gupta}, \bibinfo{author}{S.~Roy}, \bibinfo{author}{H.~V. Kojouharov},
\newblock \bibinfo{title}{Second-order modified positive and elementary stable nonstandard numerical methods for n-dimensional autonomous differential equations},
\newblock \bibinfo{journal}{Math. Meth. Appl. Sci.} \rev{\bibinfo{volume}{48.7}} (\bibinfo{year}{2025}) \rev{\bibinfo{pages}{8037--8057}}. \url{https://doi.org/10.1002/mma.9560}

%Type = Inproceedings
\bibitem[{Gupta et~al.(2020)Gupta, Slezak, Alalhareth, Roy, and Kojouharov}]{gupta}
\bibinfo{author}{M.~Gupta}, \bibinfo{author}{J.~Slezak}, \bibinfo{author}{F.~Alalhareth}, \bibinfo{author}{S.~Roy}, \bibinfo{author}{H.~Kojouharov},
\newblock \bibinfo{title}{Second-order nonstandard explicit euler method},
\newblock in: \bibinfo{booktitle}{AIP Conference Proceedings}, \rev{Vol.} \bibinfo{volume}{2302}, \bibinfo{organization}{AIP Publishing}, \rev{\bibinfo{year}{2020}}. \url{https://doi.org/10.1063/5.0033534}

%Type = Article
\bibitem[{Hoang(2024)}]{hoang24}
\bibinfo{author}{M.~T. Hoang},
\newblock \bibinfo{title}{High-order nonstandard finite difference methods preserving dynamical properties of one-dimensional dynamical systems},
\newblock \bibinfo{journal}{Num. Alg.}  (\bibinfo{year}{2024}) \bibinfo{pages}{1--31}. \url{https://doi.org/10.1007/s11075-024-01792-1}

%Type = Article
\rev{\bibitem[{Hoang et~al.(2024)}]{hoang24new}
\bibinfo{author}{M.~T. Hoang}, \bibinfo{author}{M. Ehrhardt},
\newblock \bibinfo{title}{A second-order nonstandard finite difference method for a general Rosenzweig–MacArthur predator–prey model.},
\newblock \bibinfo{journal}{J. Comp. Appl. Math.} \bibinfo{volume}{444} (\bibinfo{year}{2024}) \bibinfo{pages}{115752}. \url{https://doi.org/10.1016/j.cam.2024.115752}

%Type = Article
\bibitem[{Hoang et~al.(2024)}]{conte25}
\bibinfo{author}{D. Conte}, \bibinfo{author}{G. Pagano}, \bibinfo{author}{T. Roldán},
\newblock \bibinfo{title}{High order nonstandard finite-difference methods},
\newblock \bibinfo{journal}{Appl. Math. Comp.} \bibinfo{volume}{510} (\bibinfo{year}{2026}) \bibinfo{pages}{129681}.} \url{https://doi.org/10.1016/j.amc.2025.129681}

%Type = Article
\bibitem[{Anguelov and Lubuma(2022)}]{anguelov22}
\bibinfo{author}{R.~Anguelov}, \bibinfo{author}{J.~Lubuma},
\newblock \bibinfo{title}{Second-order nonstandard finite difference schemes for a class of models in bioscience},
\newblock \bibinfo{journal}{arXiv preprint arXiv:2207.11618}  (\bibinfo{year}{2022}). \url{https://doi.org/10.48550/arXiv.2207.11618}

%Type = Article
\bibitem[{Hoang(2022)}]{hoang22}
\bibinfo{author}{M.~T. Hoang},
\newblock \bibinfo{title}{A novel second-order nonstandard finite difference method for solving one-dimensional autonomous dynamical systems},
\newblock \bibinfo{journal}{Comm. Nonlin. Sci. Num. Sim.} \bibinfo{volume}{114} (\bibinfo{year}{2022}) \bibinfo{pages}{106654}. \url{https://doi.org/10.2139/ssrn.3958689}

%Type = Article
\bibitem[{Chen-Charpentier et~al.(2006)Chen-Charpentier, Dimitrov, and Kojouharov}]{chen}
\bibinfo{author}{B.~M. Chen-Charpentier}, \bibinfo{author}{D.~T. Dimitrov}, \bibinfo{author}{H.~V. Kojouharov},
\newblock \bibinfo{title}{Combined nonstandard numerical methods for ODEs with polynomial right-hand sides},
\newblock \bibinfo{journal}{Math. Comp. Sim.} \bibinfo{volume}{73} (\bibinfo{year}{2006}) \bibinfo{pages}{105--113}. \url{https://doi.org/10.1016/j.matcom.2006.06.008}

%Type = Article
\bibitem[{Kojouharov and Chen-Charpentier(2004)}]{kojo04}
\bibinfo{author}{H.~V. Kojouharov}, \bibinfo{author}{B.~M. Chen-Charpentier},
\newblock \bibinfo{title}{Nonstandard Eulerian--Lagrangian methods for multi-dimensional reactive transport problems},
\newblock \bibinfo{journal}{Appl. Num. Math.} \bibinfo{volume}{49} (\bibinfo{year}{2004}) \bibinfo{pages}{225--243}. \url{https://doi.org/10.1016/j.apnum.2002.04.001}

%Type = Article
\bibitem[{Mart{\'\i}n-Vaquero et~al.(2018)Mart{\'\i}n-Vaquero, Queiruga-Dios, del Rey, Encinas, Guillen, and Sanchez}]{martin}
\bibinfo{author}{J.~Mart{\'\i}n-Vaquero}, \bibinfo{author}{A.~Queiruga-Dios}, \bibinfo{author}{A.~M. del Rey}, \bibinfo{author}{A.~H. Encinas}, \bibinfo{author}{J.~H. Guillen}, \bibinfo{author}{G.~R. Sanchez},
\newblock \bibinfo{title}{Variable step length algorithms with high-order extrapolated non-standard finite difference schemes for a SEIR model},
\newblock \bibinfo{journal}{J. Comp. Appl. Math.} \bibinfo{volume}{330} (\bibinfo{year}{2018}) \bibinfo{pages}{848--854}. \url{https://doi.org/10.1016/j.cam.2017.03.031}

%Type = Article
\bibitem[{Ejere (2026)}]{ejere}
\bibinfo{author}{A. H. Ejere},
\newblock \bibinfo{title}{A nonstandard finite difference scheme for singularly perturbed reaction–diffusion differential equations with large negative and advance shifts},
\newblock \bibinfo{journal}{Math. Open} \bibinfo{volume}{5} (\bibinfo{year}{2026}) \bibinfo{pages}{2550021}. \url{https://doi.org/10.1142/s281100722550021x}

%Type = Article
\bibitem[{Dang and Hoang(2020)}]{dang}
\bibinfo{author}{Q.~A. Dang}, \bibinfo{author}{M.~T. Hoang},
\newblock \bibinfo{title}{Positive and elementary stable explicit nonstandard Runge-Kutta methods for a class of autonomous dynamical systems},
\newblock \bibinfo{journal}{Intern. J. Comp. Math.} \bibinfo{volume}{97} (\bibinfo{year}{2020}) \bibinfo{pages}{2036--2054}. \url{https://doi.org/10.1080/00207160.2019.1677895}

%Type = Inproceedings
\bibitem[{Faragó (2026) Faragó, Mosleh}]{mosleh26}
\bibinfo{author}{Faragó, I.}, \bibinfo{author}{Mosleh, R.},
\newblock \bibinfo{title}{Convergence Analysis of the Explicit Nonstandard Runge-Kutta Methods},
\newblock in: \bibinfo{booktitle}{hallenges in Design Methods, Numerical Tools and
Technologies for Sustainable Aviation, Transport and Industry: Commemorative
publication dedicated to the 80th Jubilee of Prof. Jacques Periau}, \bibinfo{organization}{Springer
Nature Switzerland}, \bibinfo{year}{2026}, \bibinfo{pages}{179--194}. \url{https://doi.org/10.1007/978-3-031-98675-8_12}

%Type = Article
\bibitem[{Anguelov and Lubuma(2001)}]{anguelov}
\bibinfo{author}{R.~Anguelov}, \bibinfo{author}{J.~M.-S. Lubuma},
\newblock \bibinfo{title}{Contributions to the mathematics of the nonstandard finite difference method and applications},
\newblock \bibinfo{journal}{Num. Meth. Part. Diff. Eq. Intern. J.} \bibinfo{volume}{17} (\bibinfo{year}{2001}) \bibinfo{pages}{518--543}. \url{https://doi.org/10.1002/num.1025}

%Type = Article
\bibitem[{Takacs, B. (2026)}]{takacs26}
\bibinfo{author}{B. Takacs}
\newblock \bibinfo{title}{An insight on some properties of high order nonstandard linear multistep methods},
\newblock \bibinfo{journal}{Math. Comp. Sim.} \bibinfo{volume}{245} (\bibinfo{year}{2026}) \bibinfo{pages}{337--365}. \url{https://doi.org/10.1016/j.matcom.2026.01.015}

%Type = Book
\bibitem[{Gottlieb et~al.(2011)Gottlieb, Ketcheson, and Shu}]{sspbook}
\bibinfo{author}{S. Gottlieb}, \bibinfo{author}{D. Ketcheson}, \bibinfo{author}{C. W. Shu}, \bibinfo{title}{Strong stability preserving Runge-Kutta and multistep time discretizations}, \bibinfo{publisher}{World Scientific}, \bibinfo{address}{Singapore}, \bibinfo{year}{2011}. \url{https://doi.org/10.1142/7498}

%Type = Article
\bibitem[{Anguelov et~al.(2010)Anguelov, Lubuma, and Minani}]{anguelov3}
\bibinfo{author}{R.~Anguelov}, \bibinfo{author}{J.~M.-S. Lubuma}, \bibinfo{author}{F.~Minani},
\newblock \bibinfo{title}{Total variation diminishing nonstandard finite difference schemes for conservation laws},
\newblock \bibinfo{journal}{Math. Comp. Model.} \bibinfo{volume}{51} (\bibinfo{year}{2010}) \bibinfo{pages}{160--166}. \url{https://doi.org/10.1016/j.mcm.2009.08.038}

% %Type = Article
% \bibitem[{Mehdizadeh~Khalsaraei and Khodadosti(2014)}]{khal14}
% \bibinfo{author}{M.~Mehdizadeh~Khalsaraei}, \bibinfo{author}{F.~Khodadosti},
% \newblock \bibinfo{title}{A new total variation diminishing implicit nonstandard finite difference scheme for conservation laws},
% \newblock \bibinfo{journal}{Com. Meth. Diff. Eq.} \bibinfo{volume}{2} (\bibinfo{year}{2014}) \bibinfo{pages}{91--98}. \url{https://dor.isc.ac/dor/20.1001.1.23453982.2014.2.2.5.4} (Accessed 8 July 2026)

%Type = Article
\bibitem[{Mehdizadeh~Khalsaraei(2017)}]{khal17}
\bibinfo{author}{M.~Mehdizadeh~Khalsaraei},
\newblock \bibinfo{title}{Nonstandard explicit third-order runge-kutta method with positivity property},
\newblock \bibinfo{journal}{Intern. J. Nonlin. Anal. Appl.} \bibinfo{volume}{8} (\bibinfo{year}{2017}) \bibinfo{pages}{37--46}. \url{https://ijnaa.semnan.ac.ir/article_480.html} (Accessed 8 July 2026)

%Type = Article
\bibitem[{Constantinescu, E. M.  and Sandu, A.(2010)}]{const}
\bibinfo{author}{E. M. Constantinescu}, \bibinfo{author}{A. Sandu},
\newblock \bibinfo{title}{Optimal explicit strong-stability-preserving general linear methods},
\newblock \bibinfo{journal}{SIAM J. Sci. Comp.} \bibinfo{volume}{32.5} (\bibinfo{year}{2010}) \bibinfo{pages}{3130--3150}. \url{https://doi.org/10.1137/090766206}

%Type = Book
\bibitem[{Agarwal and Lakshmikantham(1993)}]{agarwal}
\bibinfo{author}{R. P. Agarwal}, \bibinfo{author}{V. Lakshmikantham}, \bibinfo{title}{Uniqueness and nonuniqueness criteria for ordinary differential equations}, \bibinfo{publisher}{World Scientific}, \bibinfo{year}{1993}. \url{https://doi.org/10.1142/1988}

%Type = Book
\bibitem[{Butcher(2016)}]{butcherbook}
\bibinfo{author}{J. C. Butcher}, \bibinfo{title}{Numerical methods for ordinary differential equations}, 2nd ed., \bibinfo{publisher}{John Wiley \& Sons}, \bibinfo{year}{2016}. \url{https://doi.org/10.1002/9780470753767}

%Type = Article
\bibitem[{Kermack and McKendrick(1927)}]{kermack}
\bibinfo{author}{W.~O. Kermack}, \bibinfo{author}{A.~G. McKendrick},
\newblock \bibinfo{title}{A contribution to the mathematical theory of epidemics},
\newblock \bibinfo{journal}{Proc. Royal Soc. London. Series A.} \bibinfo{volume}{115} (\bibinfo{year}{1927}) \bibinfo{pages}{700--721}. \url{https://doi.org/10.1098/rspa.1927.0118}

%Type = Book
\bibitem[{Capasso(1993)}]{capasso}
\bibinfo{author}{V.~Capasso}, \bibinfo{title}{Mathematical structures of epidemic systems}, \rev{Vol.} \bibinfo{volume}{97}, \bibinfo{publisher}{Springer}, \bibinfo{address}{Berlin}, \bibinfo{year}{1993}. \url{https://doi.org/10.1007/978-3-540-70514-7}

%Type = Article
\bibitem[{Tak{\'a}cs et~al.(2024)Tak{\'a}cs, Sebesty{\'e}n, and Farag{\'o} (2024)}]{takacs24}
\bibinfo{author}{B.~M. Takacs}, \bibinfo{author}{G.~S. Sebestyen}, \bibinfo{author}{I.~Farago},
\newblock \bibinfo{title}{High-order reliable numerical methods for epidemic models with non-constant recruitment rate},
\newblock \bibinfo{journal}{Appl. Num. Math.} \bibinfo{volume}{206} (\bibinfo{year}{2024}) \bibinfo{pages}{75--93}. \url{https://doi.org/10.1016/j.apnum.2024.08.008}

%Type = Article
\bibitem[{Shu(1988)}]{shu}
\bibinfo{author}{C.-W. Shu},
\newblock \bibinfo{title}{Total-variation-diminishing time discretizations},
\newblock \bibinfo{journal}{SIAM J. Sci. Stat. Comp.} \bibinfo{volume}{9} (\bibinfo{year}{1988}) \bibinfo{pages}{1073--1084}. \url{https://doi.org/10.1137/0909073}

%Type = Article
\bibitem[{Ketcheson(2009{\natexlab{a}})}]{ketcheson1}
\bibinfo{author}{D.~Ketcheson},
\newblock \bibinfo{title}{Computation of optimal monotonicity preserving general linear methods},
\newblock \bibinfo{journal}{Math. Comp.} \bibinfo{volume}{78} (\bibinfo{year}{2009}{\natexlab{a}}) \bibinfo{pages}{1497--1513}. \url{https://doi.org/10.1090/s0025-5718-09-02209-1}

%Type = Book
\bibitem[{Ketcheson(2009{\natexlab{b}})}]{ketcheson2}
\bibinfo{author}{D.~Ketcheson}, \bibinfo{title}{High order strong stability preserving time integrators and numerical wave propagation methods for hyperbolic PDEs}, \bibinfo{publisher}{University of Washington}, \bibinfo{address}{Seattle}, \bibinfo{year}{2009}{\natexlab{b}}. \url{https://faculty.washington.edu/rjl/students/ketcheson/thesis.pdf} (Accessed 8 July 2026)

%Type = Article
\bibitem[{Gottlieb and Shu(1998)}]{gottlieb98}
\bibinfo{author}{S.~Gottlieb}, \bibinfo{author}{C.-W. Shu},
\newblock \bibinfo{title}{Total variation diminishing Runge-Kutta schemes},
\newblock \bibinfo{journal}{Math. Comp.} \bibinfo{volume}{67} (\bibinfo{year}{1998}) \bibinfo{pages}{73--85}. \url{https://doi.org/10.1090/s0025-5718-98-00913-2}


\end{thebibliography}
\end{document}